\documentclass[11pt,reqno]{amsart}
 
 \usepackage[margin=1in]{geometry}
 \usepackage{amsmath,amssymb}
 \usepackage{algpseudocode}
 \usepackage{algorithm}
 \usepackage{algorithmicx}
 \usepackage{caption}
 \usepackage{subcaption}
 \usepackage{amsthm}
 \usepackage{mathrsfs}
 \usepackage{amsfonts}
 \usepackage{graphicx}
\usepackage[colorlinks=true, pdfstartview=FitV, linkcolor=blue,citecolor=blue, urlcolor=blue]{hyperref}
 \usepackage{makecell}
 \usepackage{longtable}
 \usepackage{times}
 \usepackage{xfrac}

\newtheorem{theorem}{Theorem}
  \numberwithin{equation}{section}
\newtheorem{lemma}{Lemma}[section]
\newtheorem{proposition}{Proposition}[section]

\newcommand \mw[1]{  {\mathrm{#1}}  }

 \usepackage{mathtools}
\renewcommand{\div}{\operatorname{div}}
\renewcommand{\AA}{ \mathbf{A} }

 \newcommand{\bseq}{\begin{subequations}}
 \newcommand{\eseq}{\end{subequations}}

  \renewcommand{\sin}{s_{in}}

  \newcommand{\xc}{\zeta}

\title{Vorticity blowup in compressible Euler equations in $\mathbb{R}^d, d \geq 3$}

\author[J.~Chen]{Jiajie Chen}
\address{Courant Institute of Mathematical Sciences, New York University, New York, NY 10012.}
\email{\href{jiajie.chen@cims.nyu.edu}{jiajie.chen@cims.nyu.edu}}

\date{\today}

\usepackage{amsmath}
\usepackage{amssymb}
\usepackage{amsthm}
\usepackage{amsmath}
\usepackage{setspace}
\usepackage{xcolor}
\usepackage{fancyhdr}
\usepackage{amssymb}
\usepackage{amsthm}
\usepackage{listings}
    \usepackage{cite}
    \usepackage{tabularx}
    \usepackage{graphicx}
    \usepackage{epstopdf}    
    \usepackage{epsfig}
    \usepackage{float}
    \usepackage{ listings}
    \usepackage{appendix}
 \theoremstyle{plain}
 \newtheorem{prop}{Proposition}[section]

 \theoremstyle{definition}

 \theoremstyle{remark}
 \newtheorem{remark}[prop]{Remark}

 \let\pa=\partial
 \let\al=\alpha
 \let\b=\beta
 \let\d=\delta
 \let\g=\gamma
 \let\e=\varepsilon

 \let \kp = \kappa
 \let\lam=\lambda
 
 \let\s=\sigma
 \let\f=\frac
 \let \les = \lesssim
  \let \gtr = \gtrsim
 \let\om=\omega
 
 \let \th = \theta
 \let \pr = \prime
 \let \vp = \varphi
 \let\G= \Gamma
\let\B = \Big
 \let\D=\Delta
 
 \let\S=\Sigma
 
 \let\Om=\Omega
 \let\td = \tilde
 
 \let\wh=\widehat
 \let \mr = \mathring

 \let\teq \triangleq
 
 \let\pa=\partial

 \def\cA{{\mathcal A}}

 \def\cD{{\mathcal D}}
 \def\cE{{\mathcal E}}

 \def\cK{{\mathcal K}}
 \def\cL{{\mathcal L}}
 
 \def\cN{{\mathcal N}}

 \def\cR{{\mathcal R}}
 \def\cS{{\mathcal S}}
 \def\cT{{\mathcal T}}

 \def\cW{{\mathcal W}}
 \def\cX{{\mathcal X}}

 \def\cN{{\mathcal N}}

 \def\na{\nabla}
 \def\la{\langle}
 \def\ra{\rangle}

\def\one{\mathbf{1}}

  \def\Re{\mathrm{Re}}
  
\def\rsarrow{ \rightsquigarrow}

 \newcommand{\beq}{\begin{equation}}
 \newcommand{\eeq}{\end{equation}}
  \newcommand{\bal}{\begin{aligned} }
  \newcommand{\eal}{\end{aligned}}
  \newcommand{\bga}{\begin{gathered} }
  \newcommand{\ega}{\end{gathered}}
 \newcommand{\ben}{\begin{eqnarray}}
 \newcommand{\een}{\end{eqnarray}}
 \newcommand{\beno}{\begin{eqnarray*}}
 \newcommand{\eeno}{\end{eqnarray*}}

 \newcommand{\ee}{\mathbf{e}}

 \newcommand{\uu}{\mathbf{u}}
 \newcommand{\vv}{\mathbf{v}}

\renewcommand{\AA}{ \mathbf{A} }

 \newcommand{\FF}{\mathbf{F}}

 \newcommand{\R}{\mathbb{R}}
 
 \newcommand{\UU}{\mathbf{U}} 
\newcommand{\VV}{\mathbf{V}}

 \newcommand{\supp}{\mathrm{supp}}

\begin{document}

\begin{abstract}
We prove finite-time vorticity blowup in the compressible Euler equations in $\mathbb{R}^d$ for any $d \geq 3$, starting from smooth, localized, and non-vacuous initial data. This is achieved by lifting the vorticity blowup result from \cite{chen2024Euler} in $\mathbb{R}^2$ to $\mathbb{R}^d$ and utilizing the axisymmetry in $\mathbb{R}^d$. At the time of the first singularity, both vorticity blowup and implosion occur on a sphere $S^{d-2}$. Additionally, the solution exhibits a non-radial implosion, accompanied by a stable swirl velocity that is sufficiently strong to 
initially dominate the non-radial components and to generate the vorticity blowup.
\end{abstract}

\maketitle

\section{Introduction}

We consider the isentropic compressible Euler equations on $\R^d$ with $d \geq 2$ 
\bseq\label{euler0}
\begin{align}
 \pa_t (\rho \uu) + \operatorname{div}(\rho \uu \otimes \uu) + \na p& = 0,  \\
\pa_t \rho + \nabla \cdot (\rho \uu) & = 0 , 
\end{align}
where
$\uu \colon \R^d \times [0,\infty) \to \R^d$ is the velocity vector field and  $\rho \colon \R^d \times [0,\infty) \to \R_+$  is the strictly positive density function.  The pressure  $p\colon \R^d \times [0,\infty) \to \R_+$ is determined from the density via the ideal gas law
\begin{equation}
p = \tfrac{1}{\gamma} \rho^\gamma, 
\end{equation}
\eseq
where  $\gamma>1$ is the adiabatic exponent. We consider the Euler equations \eqref{euler0} with smooth initial data $\uu_0, \s_0 > 0$ with $(\uu_0, \s_0)$ converging to $(0, c)$ with $c>0$
sufficiently fast as $|x| \to \infty$. 

It is known that the compressible Euler equations can develop finite time singularity.
In \cite{lax1964development}, Lax used Riemann invariants to show that shocks may develop for Euler in 1D. 
Later, John \cite{Jo1974} and Liu \cite{Liu1979} generalized and improved Lax's result, and Sideris \cite{Si1985} proved singularity formation in 2D and 3D using a proof by contradiction.
In more recent years, researchers have developed more precise descriptions of the first singularity from smooth initial data, which can be classified into two classes. 
The first class is the shock wave \cite{Ch2007,ChMi2014,BuShVi2022,BuDrShVi2022,BuShVi2023a,BuShVi2023b,LuSp2018,LuSp2024}, where $(\uu, \rho)$ remain bounded but their gradients blow up. The second class is a smooth implosion singularity ~\cite{merle2022implosion1,merle2022implosion2,cao2023non,buckmaster2022smooth}, 
which was first established in the groundbreaking works of Merle-Raphael-Rodnianski-Szeftel ~\cite{merle2022implosion1,merle2022implosion2},
where both $|| \uu ||_{L^{\infty}}$ and $||\rho||_{L^{\infty}}$ blow up. We defer the detailed discussions on these results to Section \ref{sec:work}. Nevertheless, it is an outstanding open problem to determine whether $\limsup_{t \to T} \| \omega(\cdot,t)\|_{L^\infty(\R^d)} = \infty$, where $\omega = \nabla \times \uu$ denotes the fluid vorticity, a physically important quantity. 

For the shock singularity established in \cite{BuShVi2022,BuDrShVi2022,BuShVi2023a,BuShVi2023b}, the authors show that the vorticity $\om(\cdot, t)$ remains uniformly bounded in $C^{1/3}$. 
The implosion singularity constructed in \cite{merle2022implosion1,merle2022implosion2,buckmaster2022smooth} is confined to the radial symmetry, where 
vorticity remains $0$. In \cite{cao2023non}, the authors established implosion for a finite co-dimension set of smooth initial data, where the implosion is a non-radial perturbation of the radial implosion \cite{merle2022implosion1,merle2022implosion2,buckmaster2022smooth}. Yet, since the imploding solution \cite{merle2022implosion1,buckmaster2022smooth,cao2023non}  is potentially highly unstable, the unstable modes can potentially either lead to trivial vorticity or destroy the smooth implosion mechanism. This presents a major challenge in extending the implosion to prove vorticity blowup. We refer to Sections \ref{sec:implod} and \ref{sec:intro_idea} for more discussions of this difficulty and \cite[Section 2.7]{chen2024Euler} for a \textit{quantitative} estimate.  





The first vorticity blowup result for compressible Euler was established recently in \cite{chen2024Euler} in $\R^2$. The authors introduced the 2D axisymmetry solution and reduced \eqref{euler0} to a system of three variables: the radial velocity, the swirl 
velocity, and the rescaled sound speed (see \eqref{eq:sig}) which only depend on the radial 
and temporal variables. Moreover, the authors made use of the properties that the specific vorticity $\f{\om}{\rho}$ with $\om = \na \times \uu$ is transported along the fluid-velocity
and \textit{preserved} at the origin: $(\f{\om}{\rho})(0, t) = (\f{\om}{\rho})(0, 0)$, where 
the density blows up: $\lim_{t\to\infty}\rho(0,t) =\infty$. Moreover, the vorticity is \textit{only} generated from the swirl velocity, which is a small perturbation of the radial implosion. 

Compared to \cite{chen2024Euler}, we face the following new difficulties in $\R^d, d \geq 3$. 
Firstly, there is no direct generalization of 2D axisymmetry in higher dimension that allows us to reduce \eqref{euler0} to a 1D PDE system of a few variables. Thus, we need to study a generically \textit{non-radial} problem. 
Secondly, due to the breakdown of the 2D symmetry, there is no preservation of the specific vorticity. Moreover, there are much more directions in the non-radial setting that can affect the vorticity.

In this paper, our goal is to construct vorticity blowup in $\R^d$ for any $d \geq 3$ by lifting the result in \cite{chen2024Euler} from $\R^2$ to $\R^d$ using the axisymmetry in $\R^d$.  We analyze the implosion with non-radial perturbation that calibrates the difference between the $\R^2$ and $\R^d$ settings. On top of that, we construct an analog of swirl velocity in 2D \cite{chen2024Euler} that enjoys full stability and is much stronger than other non-radial perturbation to drive the vorticity blowup. Note that this lifting from $\R^2$ to $\R^d$ is nontrivial. The error introduced by the lifting could potentially destroy the highly unstable implosion mechanism. See Section \ref{sec:compare} for more discussion.

\subsection{Axisymmetry in $\R^d$}\label{sec:intro_axi}

To introduce our setting, we first define the rescaled sound speed $\s$ 
\beq\label{eq:sig}
\s = \f{1}{\al} \rho^{\al},\quad \al = \f{\g-1}{2},
\eeq
and reformulate the isentropic Euler equations \eqref{euler0} equivalently as a system 
of the velocity $\uu$ and the rescaled sound speed $\s$
\bseq\label{euler} 
\begin{align}
	\pa_t \s +  \uu \cdot \na \s + \al \s \cdot \mathrm{ div_d} \uu  &= 0 , \label{euler_s} \\
	\pa_t \uu + \uu \cdot \na \uu + \al \s \cdot \na \s & = 0 , \label{euler_u} 
\end{align}
\eseq
where we use $\div_d$ to denote the divergence in $\R^d$.

We introduce the following coordinates for $x \in \R^d$ 
\beq\label{eq:axi_rz}
\bal
r & =  (x_1^2 + x_2^2 + ... + x_{d-1}^2)^{1/2}, 
\quad z = x_n, \quad 
\ee_r  = \f{1}{r} (x_1, x_2, .., x_{d-1}, 0), \quad \ee_z = (0, 0, .., 0, 1) ,
\eal
\eeq
and consider axisymmetric solution in $\R^d$ without swirl 
\footnote{
It appears that the terminology of axisymmetry and swirl is mostly used for 3D incompressible Euler / Navier-Stokes equations, where one has the decomposition $\uu = u^r \ee_r + u^{\th} \ee_{\th} + u^z \ee_z$ with $u^r, u^{\th}, u^z$ depending only on $r, z$, see e.g, \cite[Section 2.3]{majda2002vorticity}. The component $u^{\th} \ee_{\th}$ with $\ee_{\th} = (-x_2 , x_1, 0)/r$ is called the swirl velocity. 
We generalize this terminology to denote solution with $\uu, \s$ admitting the decomposition \eqref{eq:axi_sym}.
}
\beq\label{eq:axi_sym}
\bal
\uu(x, t)  & = u^r(r, z, t) \ee_r + u^z(r, z,t) \ee_z, \quad 
\s = \s(r, z, t). 
\eal
\eeq
The solution with this symmetry is \textit{radially symmetric} in the first $d-1$ coordinates. 
We define the angular vorticity generalizing the notion for 3D Euler equations 
\footnote{
For 3D Euler, the angular vorticity is defined as $\om^{\th} = \pa_z u^r - \pa_r u^z$ see e.g. \cite[Section 2.3]{majda2002vorticity}, \cite{ChenHou2023a,luo2013potentially-2}. 
}
\beq\label{eq:vor_angle}
\om^{\th} = \sum_{i=1}^{d-1} \f{x_i}{r} ( \pa_z u_i - \pa_i u^z  )
 = \pa_z \sum_{i=1}^{d-1} \f{x_i}{r} \cdot u_i 
- \pa_r u^z  = \pa_z u^r - \pa_r u^z.
\eeq
If $\om^{\th}$ blows up, using triangle inequality, we obtain that the vorticity matrix $ \na \uu - (\na \uu)^{\perp}$ blows up.

To show that the above symmetries are preserved by \eqref{euler}, we can use the property that \eqref{euler} enjoys the rotational symmetry. Alternatively, we can verify this by deriving the evolution equations for  $(u^r, u^z, \s)(r, z, t)$. We first recall the following identities for the axisymmetric velocity $\uu$ and a smooth function $f = f(r, z)$
\[
\bga
  \ee_r \cdot \na f = \pa_r f ,\quad \ee_z \cdot \na f = \pa_z f ,
\quad  \uu \cdot \na f = ( u^r \pa_r + u^z \pa_z ) f ,
\quad \na f(r, z) =  \pa_r f \ee_r + \pa_z f   \ee_z, \\
  \div_d( \uu) = \pa_r u^r + \f{d-2}{r} u^r + \pa_z u^z 
 = \div_{2D}( \uu ) + \f{d-2}{r} u^r , \\
\ega
\]
where we use $\div_{2D}$ to denote the divergence in $(r, z)$
\[
  \div_{2D}( \uu ) = \pa_r u^r + \pa_z u^z .
\]
Using the above identities, we rewrite \eqref{euler} as a closed system of 
$(u^r, u^z, \s)(r, z, t)$
\beq\label{eq:euler_axi}
\bal
 \pa_t \s +  (u^r \pa_r + u^z \pa_z) \s + \al \s \cdot \div_{2D}(\uu) & = - \al(d-2) \f{ u^r \s}{r}, \\
 \pa_t u^r + (u^r \pa_r + u^z \pa_z ) u^r + \al \s \pa_r \s& = 0,  \\
  \pa_t u^z + (u^r \pa_r + u^z \pa_z ) u^z + \al \s \pa_z \s& = 0.
\eal
\eeq

Since in the rest of the paper, 
we do not use the operator $\div_{d}$,  
we simplify $\div_{2D}$ as $\div$ to simplify the notation. To draw connection between \eqref{eq:euler_ssb} and 2D Euler equations, we consider a change of variable
\beq\label{eq:axi_recenter}
 r = R_c + \xc_1, \quad z = \xc_2,
 \quad 
(u_{\xc, 1}, u_{\xc, 2}, \s_{\xc})(\xc)
= (u^r, u^z, \s)( \xc_1 + R_c, \xc_2 ) 
\eeq
 with $R_c \gg 1$. We treat $(r, z)$ as $(\xc_1, \xc_2)$ in $\R^2$, $(\pa_r, \pa_z)$ as $(\pa_{\xc_1}, \pa_{\xc_2})$. In the rest of the paper, we will focus on the system \eqref{eq:euler_axi} and do not use \eqref{euler} for the analysis. By abusing the notation, we simplify $\pa_{\xc_i}$ as $\pa_i$, $u_{\xc, i}$ as $u_i$, and write $\uu =(u_1, u_2)$. Then we can rewrite the axisymmetric Euler \eqref{eq:euler_axi} as 
\bseq\label{eq:euler_2D}
\begin{align}
 \pa_t \s + \uu \cdot \na \s + \al \s \div(\uu) &= \e_d, \quad \e_d = -  \al (d-2) \f{ u_1 \s}{R_c + \xc_1}, \label{eq:euler_2Da} \\
  \pa_t \uu + \uu \cdot \na \uu + \al \s \na \s  & =  0 .
\end{align}
\eseq
The angular vorticity $\om^{\th}$ \eqref{eq:vor_angle} becomes 
\beq\label{eq:vor_id}
\om(\uu) = \pa_z u^r - \pa_r u^z = \pa_2 u_1 - \pa_1 u_2.
\eeq

We will consider a blowup solution with $\uu , \na \s$ supported near $(r, z) = (R_c, 0)$ with $R_c$ sufficiently large. Then \eqref{eq:euler_2D} can be seen as the 2D isentropic compressible Euler equations with a small perturbation $\e_d$, and we can perturb the setting in \cite{chen2024Euler} to construct vorticity blowup. We remark that the idea of lifting a blowup solution from $\R^2$ to higher dimension $\R^d$ using axisymmetry in $\R^d$ and localizing the solution has been used to establish singularity formation for 3D incompressible Euler equations \cite{ChenHou2023a,ChenHou2023b,chen2019finite2}.

\subsection{The main result}


Similar to \cite{chen2024Euler}, we will construct blowup by perturbing a smoothly imploding self-similar solution to 2D compressible Euler, as was recently constructed in~\cite{merle2022implosion1}. The authors \cite{merle2022implosion1} constructed the self-similar profiles for  Euler in $\R^d$ with any $d \geq 2$ and adiabatic exponents $\g > 1$ not in an exceptional countable sequence $\G$. For each $\g \notin \G$, there exists a discrete sequence of admissible blowup speeds $\{ r_k \}_{k\geq 1} \subset (1,r_{\mathsf{eye}}(\alpha))$, accumulating at the value $r_{\mathsf{eye}}(\alpha)$ (see \eqref{r-range} for the formula of the 2D case), such that the Euler equations~\eqref{euler} exhibit smooth, radially symmetric, globally self-similar, imploding solutions with similarity exponent $r_k$ and profile $(\bar U, \bar \Sigma)$. See Theorem~\ref{thm:merle:implosion} below for a more precise statement.

With these notations, we state our main result.



\begin{theorem}[\bf Main result]\label{thm:blowup}
Fix a non-exceptional adiabatic exponent $\gamma \not \in \Gamma$, let $\alpha = (\gamma-1)/2$, and $r>1$ be an admissible blowup speed satisfying~\eqref{r-range}, as in~\cite{merle2022implosion1}. There exists absolute constants 
\footnote{
After we fix the adiabatic exponent $\g$ and blowup speed $r$, and its associated blowup profile in Theorem \ref{thm:merle:implosion}, we treat any constants only depending on these parameters and profiles as absolute constants.  
}
$C_s$ and $R_c$ large enough, and initial data $(\uu_0, \s_0) \in C^{\infty}$ with 
$ \uu_0 \in C_c^{\infty}$ and $\s_0 \geq c >0$, 
such that the solution $\uu, \s$ to \eqref{eq:euler_2D} (equivalent to \eqref{eq:euler_axi} via \eqref{eq:axi_recenter}) and the angular vorticity $\om^{\th}$ blow up at finite time $T$. The blowup is asymptotically self-similar in the following sense 
\beq\label{eq:blow_asym}
\bal
\lim_{t \to T_-} r (T-t)^{1-1/r} \uu( (T-t)^{1/r} y, t) & = \bar \UU(y), \\
\lim_{t \to T_-} r (T-t)^{1-1/r} \s( (T-t)^{1/r} y, t) &= \bar \S(y) , \\
\eal
\eeq
for any $y \in \R^2$. There exists absolute constants $c_2 > c_1 > 0$ such that the angular vorticity $\om^{\th}$ blows up with
\beq\label{eq:vor_blowup}
    c_1 (T-t)^{ - \f{r-1}{r \al}} \leq |\om^{\th}(0, t)| \leq     c_2 (T-t)^{ - \f{r-1}{r \al}} .
\eeq
for any $t \in [0, T)$. 
Due to the axisymmetry and the relation \eqref{eq:axi_recenter}, the solution implodes and the vorticity blows up on the $S^{d-2}$ sphere $\mathbf{S} = \{ (r, z) = (R_c, 0) \} $. 
Furthermore, for any $t \in [0, T)$, 
the solution $(u^r, u^z,\s)$ to \eqref{eq:euler_axi} in the original $(r, z)$ coordinate is uniformly supported near $\mathbf{S}$
\footnote{
Due to the axisymmetry, the set of points $x \in \R^d$ with coordinates $(r, z) = (R_c, 0)$ 
is the sphere $ |(x_1, .., x_{d-1})| = R_c, x_d = 0$. By scaling symmetry of the equations, the radius of the support can be made arbitrary small compared to the center $R_c$. 
}
in the following sense:
\beq\label{eq:thm_supp}
\supp( (u^r, u^z)(t) ) \cup \supp( \s(t) - \s_{\infty}) \subset  \{ (r, z) : |(r, z) - (R_c, 0) | < C_s R_c^{1/2} < R_c / 2 \}.
\eeq
In particular, the velocity is $0$ and the density is constant away from the sphere  $\mathbf{S}$ uniformly up to $T$. 
\end{theorem}

\begin{remark}[\bf The set of initial data] 
\label{rem:intro_init_data}
As explained in Remark~\ref{rem:data}, there exists an open set $Z_2$ (a ball in a weighted Sobolev space) with radial symmetry and a finite co-dimension set $Z_1$ without symmetry assumption such that the initial data to \eqref{eq:euler_2D} in Theorem~\ref{thm:blowup} may be taken as $(\uu_0, \s_0) = (\vv_0, \s_0 ) + (\uu_0^{\th}, 0)$ 
with $ (\vv_0, \s_0 ) \in Z_1$ and 
$\uu_0^{\th} =  u_0^{\th}(|\xc|)  (-\f{\xc_2}{|\xc|}, \f{ \xc_1}{|\xc|} ) \in Z_2$. 
Since $\uu_0^{\th}$ belongs to an open set with radial symmetry, by choosing $\uu_0^{\th}$ much larger than $\vv_0$, we ensure that the initial vorticity $\omega_0$ does not vanish near $\xc = 0$. 
\end{remark}

\begin{remark}[\bf Upper bound of the blowup rates]

Using the nonlinear estimates established in Theorem \ref{thm:non} and the estimates in \cite[Section 4.5]{chen2024Euler}, we can establish the sharp spatial decay estimates of 
$\UU, \na \UU, \na \S$, where  $\UU, \S$ are the self-similar variables of $\uu, \s$ defined in \eqref{ss-var:b}. Using these estimates, the relation \eqref{ss-var:b}, and the estimates of specific vorticity (see Section \ref{sec:intro_idea} and \eqref{eq:vor_equi}, \eqref{eq:blow_pf2} in Section \ref{sec:thm1_pf}), we can deduce 
\[
 || (\uu, \s)(t)||_{L^{\infty}} \les (T-t)^{1/r-1}, \quad \| \om^{\th}(t) ||_{L^{\infty}} 
 \les C(\rho_0) (T-t)^{- \f{r-1}{r\al}}.
\]
The rates in the upper bounds are consistent with the rates in the asymptotics \eqref{eq:blow_asym} and \eqref{eq:vor_blowup}. 
\end{remark}

\subsection{Recent results on singularity formation for Euler}\label{sec:work}



The literature on singularity formation in compressible Euler equations is too vast to review here. In addition to the classical results  discussed earlier in the paper, 
we will focus on more recent developments in the multi-dimensional problem.

\subsubsection{Shock formation}

The typical singularity for the Euler equations is a shock wave. For irrotational solutions to the Euler equations, shock formation based on ideas developed for second-order quasilinear wave equations and general relativity were established in by Alinhac \cite{Al1999a,Al1999b}, 
Christodoulou \cite{Ch2007}, and Christodoulou-Miao \cite{ChMi2014}. 
In \cite{LuSp2018,LuSp2024}, Luk-Speck extended the works \cite{Ch2007,ChMi2014} to allow for non-trivial vorticity and entropy. 
Using the self-similar method, Buckmaster, Iyer, Shkoller, and Vicol \cite{BuShVi2022,BuIy2022} established the shock formation for 2D Euler with azimuthal symmetry; Buckmaster-Shkoller-Vicol 
\cite{BuShVi2023a,BuShVi2023b} established the shock formation for 3D Euler without symmetry assumption. Under symmetry assumptions, shock development was established in various settings \cite{Yi2004,ChLi2016,BuDrShVi2022}. Very recently, shock formation without symmetry assumption past the time of the first singularity was analyzed in~\cite{ShVi2023,AbSp2022}.




\subsubsection{Smooth Implosions}\label{sec:implod}
Shock formation is not the only possible singularity developed in Euler equations. 
Guderley ~\cite{Gu1942} first constructed radial, non-smooth, imploding self-similar singularities 
along with a converging shock wave. 
 Although Guderley's setting has been studied extensive in the literature, the existence of a finite-time {\em smooth implosion} (without a shock wave) was established only recently in  
 \cite{merle2022implosion1,merle2022implosion2}, and further developed in~\cite{buckmaster2022smooth,cao2023non}. While these results are inspired by the ansatz in~\cite{Gu1942}, establishing the existence of $C^\infty$ self-similar implosion profiles is very challenging due to the degeneracy of the ODE at the sonic point, and requires a great deal of mathematical sophistication \cite{merle2022implosion1} or computer-assistance \cite{buckmaster2022smooth}. This degeneracy renders the imploding solutions potentially highly unstable, and their stability has only been established for initial perturbations from a finite co-dimensional set, which are not precisely quantified. 
See \cite{biasi2021self} for the numerical study of these instabilities.




\subsubsection{Other related singularities}

In contrast to compressible Euler equations, relatively few results have been established for 
singularity formation of the 3D incompressible Euler equations. It is well-known that the solution blows up \textit{if and only if} the vorticity blows up \cite{beale1984remarks}. In the case of smooth data, singularity formation has only been established recently in the domain with smooth boundary by the author and Hou \cite{ChenHou2023a,ChenHou2023b}. Singularity formation has also been established for low regularity data in various settings \cite{chen2019finite2,elgindi2019finite,elgindi2019stability,elgindi2023instability,cordoba2023finite}. 



Using the imploding blowup profile constructed in \cite{merle2022implosion1} and 
introducing 
a front compression mechanism, the authors \cite{merle2022blow} established blowup for defocusing nonlinear Schrodinger equation. 
Building on the front compression mechanism \cite{merle2022implosion1,merle2022blow}, 
recently, Shao-Wei-Zhang \cite{shao2024self,shao2024blow} constructed the imploding self-similar singularity for relativistic Euler equations 
and proved blowup for the defocusing nonlinear wave equation. 





\vspace{0.1in}
\paragraph{\bf{Organization of the paper}}

The rest of the paper is organized as follows. In Section \ref{sec:setup}, we reformulate the equations using the self-similar transform and discuss the ideas of the proof. We perform the linear stability analysis in Section \ref{sec:lin}, and decay estimates of the semigroup in Section \ref{sec:decay}. In Section \ref{sec:non}, we close the nonlinear estimates and prove Theorem \ref{thm:blowup}.

\section{Setup in the self-similar variables and ideas of the proof}\label{sec:setup}

We first reformulate the equations using the self-similar coordinates and variables and then perform linearization around the self-similar profiles in Sections \ref{sec:ss_coor}-\ref{sec:ss_lin}. In Section \ref{sec:intro_idea}, we discuss the ideas of the proof, which will crucially use the axisymmetry in $\R^d$ and the structure of the equations.

\subsection{Self-similar coordinates}\label{sec:ss_coor}
For $r > 1$ satisfying \eqref{r-range}, we introduce the self-similar coordinates and variables for the shifted coordinate \eqref{eq:axi_recenter}
\bseq\label{eq:var_ss}
\beq\label{ss-var:a}
\bal
 s  & = - r^{-1}  \log(T-t) , \quad y = \frac{\xc}{(T-t)^{1/r}}  ,
\quad \xi = |y| =  \frac{ |\xc|}{(T-t)^{1/r}}   , 
  \\
\eal
\eeq

Throughout the paper we will use $y \in \R^2$ to denote the self-similar vectorial space coordinate, and $\xi = |y|$ to denote the self-similar radial variable. 
The self-similar velocity components and rescaled sound speed are then given by 
\begin{equation}
\bigl( \uu, \sigma\bigr)(\xc,t) 
= \tfrac{1}{r}  (T-t)^{\frac{1}{r}-1} ( \UU, \Sigma )(y ,s).
 \label{ss-var:b}
\end{equation}
\eseq

We introduce the  polar coordinate for $(y_1, y_2) \in \R^2$ 
\beq\label{eq:polar}
\xi = |y|,\quad \th = \arctan(y_2 / y_1), \quad  \ee_R = (\cos \th, \mathrm{sin} \th), \quad \ee_{\th} = (-\mw{sin} \th, \cos \th). 
\eeq

Throughout the paper we will assume that the parameter $r$ appearing in \eqref{eq:var_ss} satisfies the following inequalities\footnote{The definition of $r_{\mathsf{eye}}$ is directly adapted from ~\cite{merle2022implosion1} upon letting $d=2$ and $\gamma = 1 + 2\alpha $; \eqref{r-range} is the same as \cite[equation (1.15)]{merle2022implosion1}.}
\begin{subequations}  \label{r-range}
\begin{align}
1<r<r_{\mathsf{eye}}(\alpha),
\qquad
&r_{\mathsf{eye}}(\alpha) = \tfrac{1 + 2\alpha}{1+\alpha \sqrt{2}}, 
& \alpha > \tfrac{1}{2}, \label{r-rangea} \\
\tfrac{1 + 2\alpha}{1+\alpha \sqrt{2}}< r<r_{\mathsf{eye}}(\alpha), 
\qquad
&
r_{\mathsf{eye}}(\alpha) =1+ \tfrac{\alpha}{(\sqrt{\alpha}+1)^2},
&
0 < \alpha < \tfrac{1}{2} \label{r-rangeb}.
\end{align}
\end{subequations}


Using the self-similar coordinates and variables, we can rewrite the system \eqref{eq:euler_2D} as follows 
\bseq\label{eq:euler_ss}
\beq\label{eq:euler_ssa}
\bal
& \pa_s \UU + (r-1) \UU + (y + \UU) \cdot \na \UU + \alpha \S  \na \S  = 0,  \\
& \pa_s \S + (r-1)\S + (y + \UU) \cdot \na \S + \alpha \S \cdot \div(\UU)   = \cE_d,
\eal
\eeq
where  $\cE_d$ denotes the lower order term and is defined below 
\beq\label{eq:euler_ssb}
 \cE_d = - \al (d-2) \f{ e^{-s} U_1 \S }{  R_c +  e^{-s} y_1 } .
\eeq
\eseq

The initial time $t=0$ in \eqref{ss-var:a} corresponds to $s_{in}$  at the self-similar time 
\begin{equation}\label{eq:s_in}
\sin \teq - \tfrac{1}{r} \log T. 
\end{equation}

\subsection{Self-similar blowup profiles}\label{sec:profile}

Recall the polar coordinate $\ee_R, \ee_{\th}$ defined in \eqref{eq:polar}. We summarize the results of \cite[Theorem 1.2 and Corollary 1.3]{merle2022implosion1}, which establish the existence of self-similar profiles in $\R^2$, i.e. smooth solution 
$\bar \UU = \bar U \ee_R, \bar \S $ to \eqref{eq:euler_ss} in $\R^2$ with $\pa_s \UU, \pa_s \S = 0$ and without the term $\cE_d$.

\begin{theorem}[\bf Existence of smooth similarity profiles~\cite{merle2022implosion1}]
\label{thm:merle:implosion}
There exists a (possibility empty) exceptional countable sequence $\Gamma = \{ \g_n\}_{n \geq 1}$ whose accumulation points can only be $\{1, 2, \infty \} $ such that the following holds. For each $\g \notin \Gamma$, 
there exists a discrete sequence of blow up speeds $\{ r_k \}_{k\geq 1}$  in the range \eqref{r-range}, accumulating at $r_{\mathsf{eye}}(\frac{\gamma-1}{2})$, such that the system \eqref{eq:euler_ss} admits a $C^\infty$-smooth radially symmetric stationary solution $\bar \UU(y) = \bar U(\xi) \ee_R,  \bar \S(\xi)$.
\end{theorem}
The requirements on $\gamma$ and $r_k$ stated above are the same as in~\cite[Theorem 1.2]{merle2022implosion1}, with $d=2$. 
Throughout the paper we will fix a non-exceptional adiabatic exponent $\gamma \not \in \Gamma$, a similarity exponent $r \in \{r_k\}_{k\geq 1}$, a similarity profile $(\bar U,0,\bar{\S})$, the blowup time $T$, and $\sin$ \eqref{eq:s_in} in our analysis. We treat any constants related to the profile $(\bar U,0,\bar{\S})$, to $\alpha = \frac{\gamma-1}{2}$, to $r \in(1,r_{\mathsf{eye}}(\alpha))$, and to $T, \sin$ as absolute constants. 

In the next lemma we collect some useful properties of the profiles $(\bar{U}, \bar{\S})$ constructed in~\cite{merle2022implosion1}.

\begin{lemma}[\bf Properties of the similarity profiles]
\label{lem:profile}
The radial vector field $\bar U(\xi) \ee_R$  and the radially symmetric function $\bar{\Sigma}(\xi) > 0$ are smooth. For every integer $k \geq 0$ we have the following decay\footnote{Here and throughout the paper we denote by $\langle \cdot \rangle$ the quantity $\sqrt{1+|\cdot|^2}$.} as $\xi \to \infty$:
\begin{subequations}
\label{eq:profile:properties}
\begin{equation}
|\partial^k_\xi  \bar{U}(\xi)| \les_k  \la \xi \ra^{1-r-k} , \quad
|\partial^k_\xi  \bar{\S}(\xi)| \les_k  \la \xi \ra^{1-r-k} .
\label{eq:dec_U}   
\end{equation}
There exists $\kp>0$ and $\xi_1 > \xi_s$\footnote{The value $\xi_s$ corresponds to the point $P_2$ in the phase portrait of~\cite{merle2022implosion1}.} such that 
\begin{align}
1 + \partial_\xi \bar{U}(\xi) - \alpha |\pa_{\xi} \bar{\S}(\xi)| &> \kp,  \quad \xi \in [0, \xi_1] \label{eq:rep1}, \\
\xi + \bar U(\xi) - \al \bar{\S}(\xi)&> 0, \quad \xi > \xi_s, \label{eq:rep2} \\
1 + \xi^{-1} \bar U(\xi)  & > \kp, \quad \xi \geq 0 ,  \label{eq:rep22} \\
1 + \f{\bar U}{\xi} - |\alpha \pa_{\xi} \S| & > \kp , 
 \quad \xi \in [0, \xi_1] \label{eq:rep_ag}. 
\end{align}
Lastly, since $r$ lies in the range \eqref{r-range}, we have
\begin{equation}
{\lim}_{\xi \to 0^+} \xi^{-1} \bar{U}(\xi)  = \partial_\xi \bar{U}(0) =- \tfrac{r-1}{2 \alpha} > -\tfrac{r}{2}. 
\label{eq:rep3}
\end{equation}
\end{subequations}
\end{lemma}

Properties~\eqref{eq:profile:properties} except for \eqref{eq:rep_ag} have been proved in \cite{chen2024Euler} based on the estimates established in \cite{merle2022implosion1}. In \cite{cao2023non}, the authors established that the new repulsive condition \eqref{eq:rep_ag} and the repulsive condition \eqref{eq:rep1} for all $\xi \in [0, \infty)$ are sufficient for the stability analysis of the non-radial perturbation of the radial implosion. Here, we only need the angular repulsive condition \eqref{eq:rep_ag} for $\xi \in [0, \xi_s]$, which has essentially been proved in \cite{merle2022implosion1} and follows from the sign of the phase portrait in \cite{merle2022implosion1}. By continuity, we can extend it to a region $[0, \xi_1]$ slightly larger than $[0, \xi_s]$. For completeness, we present the proof of \eqref{eq:rep_ag} to Appendix \ref{sec:angle_rep }.

\begin{remark}[No exterior repulsive condition]\label{rem:no:repulsive}
The repulsive condition~\eqref{eq:rep1} may not hold true in the entire exterior domain $\xi > \xi_s$ for the 2D compressible Euler equations, which is related to \cite[Lemma 1.7]{merle2022implosion1}. This issue has been discussed in detail in \cite[Remark 2.3]{chen2024Euler}, and we refer there for more discussion.
\end{remark}

\subsection{Linearization around the profile}\label{sec:ss_lin}

We consider a solution to the Euler equations \eqref{eq:euler_ss} in the form $ (U, \S) =  ( \bar \UU,  \bar{\S} ) + (\td \UU,  \td \S) $. Our goal is to analyze the evolution of the perturbation
\bseq\label{eq:lin}
\begin{align}
\pa_s \td \UU & = \cL_U( \td \UU,  \td \S) + {\cN}_{U}, \label{eq:lin_U}  \\
  \pa_s \td \S & = \cL_{\S}( \td \UU, \td \S) + \cN_{\S} , \label{eq:lin_S} 
\end{align}
where the linearized operators are defined as 
\begin{align}
 \cL_U &=  -  (r-1) \td \UU -  (y +\bar{ \UU}) \cdot \na \td \UU  -  \td \UU \cdot \na \bar{\UU} 
 - \alpha \td \S \na \bar{\S}  - \alpha \bar{\S} \na \td \S, \label{eq:linc}  \\
 \cL_{\S} &= -  (r-1) \td \S - (y +\bar{ \UU}) \cdot \na \td \S
 -  \td \UU  \cdot \na \bar{\S}     -  \alpha \div(\bar \UU)
 \td \S - \alpha \div(\td \UU) \bar{\S}  , \label{eq:lind}
\end{align}
\eseq
and the nonlinear terms are defined as
\beq\label{eq:non}
\bal
 {\cN}_{  U} &=  -  \td \UU \cdot \na \td \UU - \alpha \td \S \na \td \S  ,  \\
 {\cN}_{ \S} & =  - \td \UU \cdot \na \td \S -  \alpha  \div(\td \UU) \td \S 
 + \cE_d .
\eal
\eeq

Note that we do not linearize the lower order term $\cE_d$ \eqref{eq:euler_ssb}. 
The linearized operators $\cL =(\cL_U, \cL_S)$ are the same as \cite{chen2024Euler}. Yet, we analyze $\cL$ for non-radial perturbation, while \cite{chen2024Euler} analyzes radial perturbation. 


\subsection{Key observations and ideas of the proof}\label{sec:intro_idea}


Recall the 2D axisymmetry and the difficulties discussed before Section \ref{sec:intro_axi}. Since the 2D axisymmetry used in \cite{chen2024Euler} no longer holds in our setting, we face new difficulties in estimating the vorticity, the trajectory, and constructing nontrivial initial vorticity. To overcome these difficulties, in addition to lifting the compressible Euler equations from $\R^2$ to $\R^d, d \geq 3$ using the setup in Section \ref{sec:intro_axi}, we have three important observations listed in steps (a),(b),(c) below.

\vspace{0.1in}
\noindent \textbf{(a) Vorticity estimates.}
Firstly, using \eqref{eq:euler_2D}, we can derive the equation of a specific vorticity $\Om =\f{\om}{\rho}$ as 
\beq\label{eq:intro_vor}
 \pa_t ( \f{\om}{\rho} ) + \uu \cdot \na (\f{\om}{\rho}) =  (d-2) \f{u_1}{R_c + \xc_1} \cdot (\f{\om}{\rho} ),
 \eeq
where $\om$ is defined in \eqref{eq:vor_id}. See the derivation in Section \ref{sec:traj}. 
We observe that if $\uu$ blows up according to the rate 
in \eqref{eq:blow_asym} and it is supported uniformly away from $r = R_c + \xc_1 = 0$, e.g. $\supp(\uu) < R_c/2 $, then the forcing only amplifies $\Om$ in a bounded manner. Concretely, denote by $\Phi(x, t)$ the characteristics associated with $\uu$ starting from $\xc  = x$. Under this assumption on $\uu$, solving $\Om$ along the characteristics
, we obtain
\beq\label{eq:vor_tran}
   \Om( \Phi(x, t), t) = \Om(x, 0) e^{ R(x, t)}, 
\eeq
where $R(x, t)$ satisfies 
\[
 R(x, t) = \int_0^t  (\f{ u_1}{ R_c + \xc_1} )(\Phi(x, s), s) d s,
 \quad  |R(x, t)| \les \int_0^t  ( T-s)^{1/r-1} d s \les T^{1/r}
\]
uniformly in $t< T$. Thus, we obtain that $\Om(\Phi(x, t), t)$ and $\Om( x, 0)$ are comparable. 

Thus, one would construct vorticity blowup by (1) establishing implosion of $\rho( t)$, (2)  constructing nontrivial initial data for $\Om$, and (3) controlling the trajectory and confining $x$ to $\{ y: \Om(y, 0)\neq 0 \}$ 
and $\Phi(x, t) \in \{ y: \rho( y, t) \to \infty, \mathrm{\ as  \ } t \to T \}$.

\vspace{0.1in}
\noindent \textbf{(b) Non-trivial initial vorticity.}  
To prove implosion, we develop a finite co-dimension stability estimates of the profile with  non-radial perturbation. We refer to Section \ref{sec:idea_NR} for the estimates. 
In general, one needs to choose initial data in some finite co-dimension set defined implicitly, which does not guarantee that the initial vorticity is non-trivial at some point. See \cite[Section 2.7, Lemma 2.4]{chen2024Euler} for more discussions. 

We observe that the error term $\cE_d$ \eqref{eq:euler_ssb} (or $\e_d$ \eqref{eq:euler_2D}) is of lower order, and the linearized equations \eqref{eq:lin_U}, \eqref{eq:lin_S} with $\cN = 0$ is the same as that in \cite{chen2024Euler}. In such equation, it has been proved in 
\cite{chen2024Euler} that the swirl velocity with radial symmetry enjoys \textit{full stability}. 
Based on the above observations, we construct the solution to \eqref{eq:euler_ss} using the following ansatz 
\[
 (\UU, \S)(s) = (\bar \UU, \bar \S) + \td \VV^{\mw{R}} + \td \VV^{ \mw{NR}},
 \quad \td \VV^{\mw{R}} =   e^{\cL( s-\sin)}(\AA, 0) ,\quad 
 \td \VV^{\mw{NR}} =  (\td \UU^{\mw{NR}}, \td \S^{\mw{NR}}), \quad \AA = A(\xi) \ee_{\th},
\]
where $\cL = (\cL_U, \cL_{\S})$ is the linear operator defined in \eqref{eq:linc}, \eqref{eq:lind}, and \textit{R, NR} are short for radial, non-radial, respectively. At the linear level, both 
$ \td \VV^{\mw{NR}}, \td \VV^{\mw{R}} $ solves the linear equations \eqref{eq:lin_U}, \eqref{eq:lin_S} with $\cN = 0$. We will use the full stability of $\td \VV^{\mw{R}}$ to construct initial data with  $\td \VV^{\mw{R}} $ much larger than $\td \VV^{\mw{NR}}$: 
\beq\label{eq:intro_size}
|| \td \VV^{\mw{NR}}(\sin) ||_Y \ll || \td \VV^{\mw{R}}(\sin) ||_Y,
\quad 
|\na \times \td \UU^{\mw{NR}}(\sin) | \ll |\na \times \AA| \neq 0,
\eeq
in some norm $Y$. 
Since $\na \times \bar \UU = 0$, using the triangle inequality, we construct non-trivial initial vorticity.

Due to the potentially unstable modes, we cannot prescribe $\td \VV^{\mw{NR}}(\sin)$ explicitly, and it also depends on $\td \VV^{\mw{R}}(s)$ for all $s \geq \sin$. 
We show that $  \td \VV^{\mw{NR}}(\sin)$ only weakly depends on 
$\td \VV^{\mw{R}}(\sin)$, and thus we can achieve \eqref{eq:intro_size}.  We refer to Sections \ref{sec:decom_init} and \ref{sec:non_stab} for a more detailed decomposition and estimates.

\vspace{0.1in}
\noindent  \textbf{(c) Estimate of the trajectory.}   
Based on \eqref{eq:vor_tran}, we first obtain $|\Om( 0, t) | \gtr |\Om( \Phi^{-1}( 0, t), 0) |$. Using the self-similar variable and the outgoing property \eqref{eq:rep22} related to the self-similar velocity, we show that the self-similar coordinate of $\Phi^{-1}(0, t)$ is within the unit ball $B(0, 1)$ for all $t < T$. Then by constructing $|\Om( y, 0)| \geq c > 0$ for all $ |y|\leq 1$ using step (b), and combining the above steps, we obtain
\[
 || \om(t) ||_{L^{\infty}} \gtr \min_{|y| \leq 1} |\Om(y, 0)| \cdot |\rho(0, t)| 
 \geq c |\rho(0, t)| .
\]
By taking $t \to T$, we prove vorticity blowup.

\vspace{0.1in}
\noindent  \textbf{Control $\cE_d$ and support.}  
To control the lower order term $\cE_d$ \eqref{eq:euler_ss} in the self-similar variable, we bound the support of the solution $\UU$ so that in the support of $\cE_d$, the denominator of $\cE_d$ does not vanish. Note that the estimates are nontrivial since the solution is not propagated along some characteristics, in which case one can just control the support by estimating the ODE for characteristics, see e.g., \cite{ChenHou2023a,chen2019finite2}.
Instead, we control the support by estimating a suitable time-dependent energy in section \ref{sec:supp}. See Section \ref{sec:supp} for more details. 

\subsubsection{A new proof of implosion with non-radial perturbation}\label{sec:idea_NR}

In this section, we develop a new proof of implosion for the compressible Euler equations with non-radial perturbation $(\td \UU, \td \S)$ based on the framework developed in \cite{chen2024Euler}. Note that we cannot adopt the stability proof in \cite{cao2023non} for non-radial perturbation to our setting since it requires the exterior repulsive property of the profile \eqref{eq:rep1}, but may not hold for the similarity profile of 2D Euler equations (see Remark~\ref{rem:no:repulsive}). 
Compared to the analysis of the radial perturbation, which reduces the dimension of the problem significantly due to the symmetry,  the analysis of the non-radial setting is much more 
challenging,  which has been discussed in detail in \cite[Section 1.2]{cao2023non}.   

Note that the proof of Lemma \ref{lem:profile} does not use $d=2$ specially, and the self-similar profiles constructed in \cite{merle2022implosion1} satisfy the same properties as those in Lemma \ref{lem:profile} (with different constants). The stability proof of Euler equations in $\R^d$ with non-radial perturbation would be the same. We outline the stability proof below.


\subsubsection*{Step 1: Weighted energy estimates}

Similar to \cite{chen2024Euler}, in the linear stability analysis of Section~\ref{sec:coer_Hm} we perform weighted $H^{2m}$ estimates on the energy $E_{2m} = \int (|\D^m  \td \UU|^2 + |\D^m \td \S|^2) \td \vp_{2m}$ with sufficiently large $m$ and a carefully designed weight $\td \vp_{2m}$, to obtain coercive estimates for the ``top order'' terms. The interior repulsive property \eqref{eq:rep1} allows us to obtain a damping term proportional to $m$ in the region $|y| = \xi  \leq \xi_1 $,  slightly beyond the sonic point $\xi_s < \xi_1$. For $|y| = \xi > \xi_s$, we use the outgoing property related to the transport term \eqref{eq:rep2}. In the weighted $H^{2m}$ estimates of the non-radial 
perturbation, we need to control mixed  derivatives of $(\td \UU, \td \S)$, which, however, cannot be bounded directly by the energy norm of $\D^{m} f$.

We have two key ideas to overcome this difficulty in the weighted $H^{2m}$ estimates.

\textbf{(1) Almost equivalence between $\D $ and $\na^2$.}
By performing integration by parts, 
we show that for any function $f$ smooth enough with fast decay, the derivatives $\D f $ and $\na^2f $, on average, are essentially the same: 
\beq\label{eq:intro_IBP1}
\int |\D^m f|^2 \td \vp_{2m} - \int |\na^{2 i} \D^{m-i} f |^2 \td \vp_{2m}
 = \cE_{i, m}
\eeq
for any $ i \in \{1,..,m\}$, where $\cE_{i, m}$ is some lower order term. By polarization, we can generalize the above estimates to inner products between two functions $f, g$. See generalizations of \eqref{eq:intro_IBP1} in Lemma \ref{lem:norm_equiv_coe}.

\textbf{(2) Commute $\na$ and $\pa_{\xi}$.}
One of the difficult terms in the weighted $H^{2m}$ estimate of the energy $E_{2m}$ is
\[
   J = \int I \cdot \D^m G \td \vp_{2m}, \quad I  =- 2 m  \xi \pa_{\xi} ( \f{\bar F }{\xi} ) \pa_{\xi \xi} \D^{m-1} G , \quad \bar F = \xi + \bar U, 
\]
for $G = \td \UU, \td \S$. The term $J$ is a main term in deriving the dissipative term for $E_{2m}$ (see $\cD_U, \cD_{\S}$ in \eqref{eq:lin_Hk}), and its estimate 
is trivial in the radial setting in \cite{chen2024Euler}
\footnote{	
For radial perturbation $G$, we get
$I= - 2 m  \xi \pa_{\xi} ( \f{\bar F }{\xi} )\D^{m} G + l.o.t.$ and yield the estimate of $J$ trivially. See \cite[Section 3.2.1]{chen2024Euler}. 
}.
Denoting $G_m = \D^{m-1} G, \  B(y) = \td \vp_{2m} \xi \pa_{\xi} ( \f{\bar F }{\xi} )$, we write 
$I = - B \pa_{\xi}^2 G_{m}$. To estimate $J$, performing integration by parts in $\pa_{\xi}$ and $\na$, formally, we get 
\[
J = -2m \int  B  \pa_{\xi \xi} G_{m} \D G_{m} d y 
= - 2m  \int B  |\pa_{\xi} \na G_{m}|^2 d y 
- 2m  \int B [ \pa_{\xi}, \na] \pa_{\xi} G_{m} \cdot \na G_{m} d y + l.o.t.
\]
The commutator $[\pa_{\xi}, \na]$ is similar to $\f{1}{\xi} \na$. To compensate the singularity $\f{1}{\xi}$, we require $|B(y)| \les_m \min(1, |y|)$, which holds in our setting. Then we  treat the second term on the right side (RS)
as a lower order term when compared to $|\na^{2} G_m|^2$. We generalize these estimates in Lemma \ref{lem:non_IBP}. The first term on RS and other main terms along with the angular repulsive condition \eqref{eq:rep_ag} lead to a good damping term in the energy estimates. 

The above two ideas allow us to commute the derivatives $\pa_{\xi}$ and $\na$, 
$\na^2$ and $\D$ up to some lower order terms. The upshot of the linear stability analysis is that we obtain weighted $H^{2m}$ estimates of the kind
\begin{equation}
\label{eq:intro_coer}
 \la \cL (\UU, \S), (\UU, \S) \ra_{\cX^m} 
 \leq - \lam \| (\UU, \S)\|_{\cX^m}^2 + \int_{|y| \leq R_4}  \cR_m d y ,
\end{equation}
for some inner product $\cX^m, m \geq m_0$ with $m_0$ sufficiently large, a suitably chosen $R_4>0$ and a coercive parameter $\lam > 0$ (uniform in $m \geq m_0$), up to some lower order terms $\cR_m$, which is made explicit in~\eqref{eq:coer_est}. 

Since we obtain the same type of coercive estimates \eqref{eq:intro_coer} as that in \cite{chen2024Euler}, in the next Steps 2-4, we can essentially follow the framework in \cite{chen2024Euler} with minor changes.


\subsubsection*{Step 2: Compact perturbation.} To handle the lower order remainder term $\cR_m$, 
we use the Riesz representation theory to construct a linear operator $\cK_m$ from some bilinear 
form similar to 
\[
\la \cK_m f , g \ra_{\cX^m} =  \int \chi  f \cdot  g d y,
\] 
for some cutoff function $\chi$ supported near $0$. Then we obtain that $\cL - \bar C \cK_m$ is dissipative for large $\bar C>0$.

\subsubsection*{Step 3: Construction of semigroup}

To handle the compact perturbation, we estimate the growth bound for the semigroup $e^{\cL t}$ and $e^{ (\cL - \cK_m) t}$ based on the dissipative estimates of $\cL -\cK_m$. We use the local well-posedness result of the symmetric hyperbolic system to construct the semigroup $e^{\cL t}$ and $e^{ (\cL - \cK_m) t}$.

\subsubsection*{Step 4: Smoothness of unstable direction}

To construct a smooth initial data, in Section~\ref{sec:smooth_unstab} we show that the potentially unstable directions of the linearized operators are spanned by smooth functions using an abstract functional lemma.


\subsubsection*{Step 5: Nonlinear estimates} %
To close the  nonlinear estimates, in Section~\ref{sec:non} we decompose the perturbation $\td \VV = (\td \UU, \td \S)$ into a stable part $\td \VV_1$ and an unstable part $\td \VV_2$, following the splitting method of~\cite{ChenHou2023a}, and estimate $\td \VV_2$ using the semigroup $e^{\cL s}$. We construct a global in self-similar time solution using a fixed point argument following \cite{chen2024Euler}, which implicitly determines the 
initial data in the correct finite co-dimension set. 
In addition, we need to control $\cE_d$ and the support of the solution. See the discussion before Section \ref{sec:idea_NR}.





\subsection{Comparison with existing works}\label{sec:compare}

In \cite{chen2024Euler}, the authors established the first vorticity blowup result in  compressible Euler equations, where the equations are studied in $\R^2$. A natural idea to generalize the result in \cite{chen2024Euler} to $\R^d, d \geq 3$ is to consider the trivial extension: $\uu(x , y )  
= \uu^{(2D)}(x), \s( x, y) =\s^{(2D)}(x)$ for $x \in \R^2, y \in \R^{d-2}$, or the extension with 
a suitable truncation in the far-field. A construction based on the extension can either have a solution with infinite-energy, or potentially destroy the implosion mechanism, which has many potentially unstable directions \cite{merle2022implosion2,biasi2021self,buckmaster2022smooth}.

In \cite{cao2023non}, the authors developed a stability proof of implosion with non-radial perturbation. For the dissipative estimates, the authors performed $H^{2m}$ estimates on the full mixed derivatives $\na^{2m} \UU, \na^{2m} \S$ to control the non-radial perturbation. 
We perform weighted estimates on $\D^m \UU, \D^m \S$ directly, which is different from \cite{cao2023non}. To obtain the dissipative estimates, we need to use the extra repulsive condition \eqref{eq:rep_ag}, which was first used in \cite{cao2023non}. Yet, we only require \eqref{eq:rep_ag} for $\xi \in [0, \xi_s]$, as opposed to $\xi \in [0, R)$ with $R \gg \xi_s$ required in \cite{cao2023non}. Establishing \eqref{eq:rep_ag} for $\xi \in [0, \xi_s]$ (and for all the profiles in $\R^d$ constructed in \cite{merle2022implosion1}) is much easier than for the whole range $ \xi \in \R_+$.
For the rest of steps, we follow the framework in \cite{chen2024Euler}. 
In particular, we do not need to localize the linearized operator nor solve a specific PDE for maximality of the linearized operator like \cite{cao2023non}.

There appears to be essential difficulties in using the non-radial imploding solutions built in \cite{cao2023non} to construct vorticity blowup for the compressible Euler equations, since one needs to choose initial data in some finite co-dimension set defined implicitly,  which does not guarantee that the initial vorticity is non-trivial at some point. We refer to \cite[Lemma 2.4 and Section 2.7]{chen2024Euler} for more discussions.

In the weighted energy estimates in Step 1 in Section \ref{sec:idea_NR}, we extract some stability properties of the linearized operators by designing weights depending on the profiles. 
Such an idea has been used successfully in a series of works by the author and collaborators~\cite{chen2019finite,chen2021HL,chen2020singularity,ChenHou2023a}. 
Moreover, by designing weights depending on the decay rates of the profile, in the stability analysis, we can obtain tight spatial decay estimates of the solutions using embedding inequalities. See Lemma \ref{lem:prod}. This method has also been used recently in \cite{chen2024stability}.

\subsection{Notation}
We use $A\les B$ to mean that there exists a constant $C = C(\gamma, r, \bar U,\bar \Sigma,T, \sin)\geq 1$ such that $A \leq C B$. We use $A \asymp B$ to mean that both $A \les B$ and $B \les A$. 
We use $A \les_a B$ to mean that $A \leq C_a B$ for a constant $C_a$ depending on $a$ and 
use $A \asymp_a B$ similarly.

We use $A\teq B$ to say that quantity $A$ is {\em defined} to be equal to quantity/expression $B$. 

We will use $\eta, \eta_s, \lam, \lam_1$ to denote constants related to the decay rates; see~\eqref{eq:decay_stab}, \eqref{eq:decay_unstab}, and~\eqref{eq:decay_para}.
We use $\kp_1, \kp_2$ to denote exponents for the weights; see~\eqref{norm:Xkb} and~\eqref{eq:kappa3}. We use $\vp_1, \vp_m,\vp_g, \vp_A$ to denote various weights; see~Lemma \ref{cor:wg}, \eqref{eq:wg_vp}, \eqref{norm:Xkb}, Theorem \ref{thm:coerA0}, and use $E_{k}$ for energy estimates~\eqref{eq:non_E}. We use $\cX^m, \cX^m_A, \cW^m , Y_1, Y_2 $ to denote various functional spaces; see~\eqref{norm:Xk}, \eqref{norm:Xk_A}, \eqref{norm:Wk}, and~\eqref{norm:fix}. 
We will use $c, C$ to denote absolute constants that can vary from line to line, and use $\mu, \nu, \e_m, \bar C, \mu_m, \d, \d_Y, \lam$ to denote some absolute constants that do not change from line to line. 

In Sections~\ref{sec:lin}, \ref{sec:decay}, we use $(\UU, \S)$ to denote the perturbation $(\td \UU, \td \S)$ mentioned earlier in~\eqref{eq:lin}, \eqref{eq:non}, since for {\em linear stability analysis} there is no ambiguity. In contrast, in the nonlinear stability of Section~\ref{sec:non}, $(\UU, \S)$ is reserved for the full velocity and rescaled sound speed, consistently with \eqref{eq:lin}, \eqref{eq:non} and~\eqref{eq:euler_ss}.

\section{Linear stability analysis}\label{sec:lin}

In this section, we perform weighted $H^{2m}$ estimates on \eqref{eq:lin} and study the semigroups associated with the linear operators in \eqref{eq:lin}. We design the weights for the energy estimates in Section \ref{sec:wg}, and perform weighted Sobolev estimates in Section \ref{sec:coer_Hm}. 


\subsection{Choice of weights}\label{sec:wg}

In this section, we design the weight $\vp_{2 m}$ for weighted $H^{2 m}$ estimates. We will generalize the construction of weights in \cite[Lemma 3.1]{chen2024Euler}, which we recall below.
\begin{lemma}[Lemma 3.1 \cite{chen2024Euler}]\label{lem:wg0}
There exists a radially symmetric weight $\wh \vp_1(y)$ and an absolute constant $\mu_2 > 0$ such that
\bseq 
\begin{align}
& \wh \vp_1 \asymp \la y \ra, \quad |\na \wh \vp_1| \les 1 , 
\quad  \xi = |y| , \label{eq:wg_asym}  \\
& 
D_\mw{wg, l}( \wh \vp_1)(y)  \leq - \mu_2 \la  y \ra^{-1} ,
 \label{eq:repul_wg}
\end{align}
for all $y \in \R^2 $ and $l \in \{ 0, 1\}$, where $D_\mw{wg, l}$ is defined as 
\beq\label{eq:repul_wg_D}
D_\mw{wg, l}( \wh \vp_1)  \teq \f{  (\xi + \bar U)\pa_{\xi}  \wh \vp_1 }{ \wh \vp_1} 
+ l \al \bar \S  \B| \f{ \pa_{\xi} \wh \vp_1 }{\wh  \vp_1} \B|  - 
(1 + \pa_{\xi} \bar U - l  \al |\pa_{\xi} \bar \S|  ) .
\eeq
\eseq
\end{lemma}

Based on the above result, we construct the weights $\vp_1, \vp_m$ with the following properties.

\begin{lemma}\label{cor:wg}
There exists a radially symmetric weight $\vp_1(y)$ and an absolute constant $\mu_1 > 0$ such that 
\bseq\label{eq:repul_wg_ag}
\begin{align}
&  \vp_1 \asymp \la y \ra, \quad |\na  \vp_1| \les 1 , 
\quad  \xi = |y| , \label{eq:repul_wg_aga}  \\
&  D_\mw{wg, l}(  \vp_1)(y)  \leq - \mu_1 \la  y \ra^{-1} ,
\label{eq:repul_wg_agb} \\
& D_\mw{ag, l}(\vp_1)(y)  \leq - \mu_1 \la  y \ra^{-1} ,
\label{eq:repul_wg_agc} 
\end{align}
for any $y$ and  $l = 0, 1$,  where $D_\mw{wg,l}$ is defined in \eqref{eq:repul_wg_D} and $D_\mw{ag, l}$ is defined as 
\beq\label{eq:repul_wg_agd} 
D_\mw{ag, l} \teq \f{  (\xi + \bar U)\pa_{\xi} \vp_1 }{\vp_1} 
+ l \al \bar \S  \B| \f{ \pa_{\xi} \vp_1 }{\vp_1} \B|  - 
(1 +  \f{ \bar U}{\xi} - l  \al |\pa_{\xi} \bar \S|  )  .
\eeq
\eseq
For $m \geq 1$ and $\vp_1$ satisfying \eqref{eq:repul_wg_ag}, we define 
\beq\label{eq:wg_vp}
\vp_m(y) =  \vp_1(y)^m.
\eeq
\end{lemma}

\begin{proof}

We will introduce several parameters and determine them in the following order 
\beq\label{eq:wg_para_ord}
p , q \rightsquigarrow c,c_2, c_3 \rightsquigarrow  \b \rightsquigarrow  \mu_1 ,
\eeq
with parameters appearing later being allowed to depend on the previous ones.

Let $\wh \vp_1$ be the weight constructed in Lemma \ref{lem:wg0}. From the definitions \eqref{eq:repul_wg_D} and \eqref{eq:repul_wg_agd}, we have 
\beq\label{eq:rela_Damp}
D_\mw{ag,l } = D_\mw{wg,l} + f,  \quad f = \f{\bar U}{\xi} - \pa_{\xi} \bar U.
\eeq
Thus we can use the estimate of $D_\mw{wg,l}$ from Lemma \ref{lem:wg0} to estimate 
$D_\mw{ag,l}$. Using \eqref{eq:dwdx},\eqref{eq:dec_U}, we get 
\[
 f(y) \geq 0, \mathrm{ \ for \ }  |y| \leq \xi_s, \quad |f| \leq \f{\mu_2}{2} \la y \ra^{-1}, 
 \mathrm{ \ for  \  } |y| \geq q, 
\]
for $q > \max(p, 2 \xi_s)$ large enough. Using \eqref{eq:rela_Damp}, \eqref{eq:repul_wg} for $D_\mw{wg,l}$, and the above estimate for $f$, we obtain 
\beq\label{eq:rep_wg_pf1}
D_\mw{ag, l}(\wh \vp_1)(y) \leq - \f{\mu_2}{2} \la y \ra^{-1}, \quad |y| \leq p, \quad |y| \geq q
\eeq
for some $p > \xi_s$. Below, we will further estimate $D_\mw{ag, l}$ in the region $|y| \in [p, q]$.

With $p, q$ chosen above, we can construct a radially symmetric function $g(y)$ satisfying 
\bseq\label{eq:wg_g}
\begin{align}
 g(y) &= 1, \quad  |y| \leq \xi_s, \quad g(y)  = \f{1}{2},\quad  |y| \geq 2 q, \label{eq:wg_ga} \\
   \pa_{\xi} g(y) & \leq 0, \ \forall \xi \in (0,\infty),  \quad \pa_{\xi} g \leq -c , \ \xi \in [p, q], \label{eq:wg_gb}
\end{align}
\eseq
for some $c = c(p, q)> 0$. For some $\b > 0$ to be chosen, we construct $\vp_1$ as follows 
\beq\label{eq:wg_vp1}
 \vp_1 = \wh  \vp_1 g^{\b}  .
\eeq
Since $\f{\pa_{\xi} \vp}{\vp} = \f{\pa_{\xi} \wh \vp_1}{\wh \vp_1} + \b \f{ \pa_{\xi} g}{g}$, using the definitions $D_\mw{wg, l}$  \eqref{eq:repul_wg_D}, $D_\mw{ ag, l}$ \eqref{eq:repul_wg_agd} and the triangle inequality yields
\bseq\label{eq:rep_wg_pf2}
\beq
D_{\al}( \wh \vp_1 g^{\b}) \leq D_{\al}( \wh \vp_1) + \b \f{ (\xi + \bar U) \pa_{\xi} g} {g}
+ \b \B| \al \bar \S \f{ \pa_{\xi} g} {g} \B|
=  D_{\al}( \wh \vp_1) + \b \f{ (\xi + \bar U - \al \bar \S) \pa_{\xi} g} {g} ,
\eeq
for $D_{\al} = D_\mw{wg, l}$ or $D_\mw{ ag, l}$. Using $\xi + \bar U - \al \bar \S \geq 0$ for $\xi \geq \xi_s$  \eqref{eq:rep2} and the estimates of $\pa_{\xi} g$ \eqref{eq:wg_gb}, we get
\beq
(\xi + \bar U - \al \bar \S) \pa_{\xi} g \leq 0, \ \forall y, \quad  D_{\al}( \wh \vp_1 g^{\b}) \leq  D_{\al}( \wh \vp_1 ).
\eeq
\eseq

Using \eqref{eq:repul_wg} for $D_\mw{wg, l}$, \eqref{eq:rep_wg_pf1} for $D_\mw{ag, l}$, and \eqref{eq:rep_wg_pf2}, we prove 
\bseq\label{eq:rep_wg_pf3}
\begin{align}
& D_\mw{wg,l}(\vp_1) = 
D_\mw{wg,l}( \wh \vp_1 g^{\b})  \leq 
D_\mw{wg,l}( \wh \vp_1)  \leq - \mu_2 \la y \ra^{-1},
\label{eq:rep_wg_pf3a}
 \\
& D_\mw{ag,l}(\vp_1) \leq 
 D_\mw{ag,l}(\wh \vp_1)\leq - \f{1}{2}\mu_2 \la y \ra^{-1}, \mathrm{\quad for \ } |y| \notin [p,  q],  \label{eq:rep_wg_pf3b}
\end{align}
and obtain \eqref{eq:repul_wg_agb} by requiring $\mu_1 < \mu_2$. 

We further estimate $D_\mw{ag,l}(\vp_1)$ for $|y|\in [p, q]$. 
Using \eqref{eq:rep2}, we get 
\beq\label{eq:repl_wg_mid}
      (\xi + \bar U - \al \bar \S) \pa_{\xi} g \leq - c_2<0,
 \eeq
on $[p, q]$ for some $c_2= c_2(p, q)>0$. Using \eqref{eq:rep_wg_pf2},  \eqref{eq:rep_wg_pf3a}, \eqref{eq:repl_wg_mid}, and  \eqref{eq:rela_Damp}, 
for $|y| \in [p, q]$, we obtain 
\[
D_\mw{ag,l }( \vp_1) \leq D_\mw{ag, l}( \wh \vp_1) 
-  \one_{[p, q]}(|y|) \cdot  c_2 \b 
=  D_\mw{wg, l}( \wh \vp_1)  + f -  \one_{[p, q]}(|y|)(  c_2 \b).
\]
Since $ |\bar U / \xi |, |\pa_{\xi} \bar U|,  |f| \leq c_3$ for all $\xi$ and some absolute constant $c_3$ \eqref{eq:dec_U}, using \eqref{eq:rep_wg_pf3a} and then choosing $\b = \b(c_2, c_3)$ large enough, for $|y| \in [p, q]$, we prove 
\beq\label{eq:rep_wg_pf3c}
D_\mw{ag,l }( \vp_1) 
\leq -\mu_2 \la y \ra^{-1} + 2 c_3 - c_2 \b 
\leq -\mu_2 \la y \ra^{-1}.
\eeq
\eseq

Combining \eqref{eq:rep_wg_pf3b}, \eqref{eq:rep_wg_pf3c} and choosing $\mu_1 = \f{1}{2} \mu_2 $, we prove \eqref{eq:repul_wg_agc}. 

Since $g \asymp 1, |\na g| \les \la y \ra^{-1}$ 
from \eqref{eq:wg_g} and $ \wh \vp_1$ satisfies \eqref{eq:wg_asym}, we obtain that $\vp_1$ satisfies \eqref{eq:repul_wg_aga}. 
\end{proof}

\subsection{Weighted $H^{2m}$ coercive estimates}\label{sec:coer_Hm}

Throughout this section, since we are only concerned with {\em linear} stability analysis, we will use $(\UU, \S)$ to denote the perturbation $(\td \UU, \td \S)$; that is, we drop the $\td \cdot $ in \eqref{eq:lin}. Recall the weights $\vp_m$ defined in Lemma \ref{cor:wg} and the linearized operator from \eqref{eq:lin}. We have the following coercive estimates.
\begin{theorem}\label{thm:coer_est}
Denote $\cL = (\cL_U, \cL_{\S})$ with $\cL_U, \cL_{\S}$ defined in \eqref{eq:lin}. 
There exists $m_0 \geq 6,  R_4, \bar C>0$ large enough and $\lam \in (0, \f{1}{2} )$  small enough \footnote{The parameters $m_0,R_4, \bar C$ and $\lambda$  depend only on the weight $\vp_1$ from Lemma~\ref{cor:wg}, on $\gamma>1$, $r>0$, and on the profiles $(\bar \UU,\bar \Sigma)$.}, such that the following statements hold true. For any $m \geq m_0$, there exists $\e_m = \e_m( m_0, R_4,\bar C, \lam)$ such that 
\begin{align}
 \la ( \cL(\UU, \S), (\UU, \S) \ra_{\cX^m} 
& \leq - \lam || (\UU, \S)||_{\cX^m}^2 
+ \bar C \int_{ |y |\leq R_4} (|\UU|^2 + |\S|^2 ) \vp_0^2 \vp_g d y ,
\label{eq:coer_est}
\end{align}
for any $(\UU, \S) \in \{ (\UU, \S) \in \cX^m, \cL(\UU, \S) \in \cX^m \} $. Here, the Hilbert spaces $\cX^m$ \footnote{It is also convenient to denote $\cX^{\infty} = \cap_{m \geq 0} \cX^{m}$.} 
are defined as the completion of the space of $C^\infty_c$ scalar/vector functions, with respect to the norm induced by the inner products
\bseq\label{norm:Xk}
\beq
 \la f , g \ra_{\cX^m} \teq \int \e_m \D^m f \cdot \D^m g  \vp_{2m}^ 2 \vp_g + f \cdot g  \vp_g  d y, \ m \geq 1, \quad  \la f , g \ra_{\cX^0} = \int f \cdot g \vp_g . 
\eeq
The inner products $\la\cdot,\cdot\ra_{\cX^m}$ and the associated norms 
are defined in terms of the constants $\e_m$, the weight $\vp_{2m} = \vp_1^{2m}$ defined in Lemma \ref{cor:wg}, and $\vp_g$ defined below 
\beq\label{norm:Xkb}
 \vp_g(y) = \la y \ra^{-\kp_1- 2}, \quad \kp_1 =  \f{1}{4} . 
\eeq
\eseq
\end{theorem}

Note that the linearized operator $\cL$ in \eqref{eq:coer_est} is the same as that in 
\cite[Theorem 3.2]{chen2024Euler}, 
but we do not assume radial symmetry for the perturbation. We emphasize that $\lam, \bar C $ in \eqref{eq:coer_est} are independent of $m$. 

We have the following simple \textit{nestedness} property of the spaces $\cX^m$, which follows from Lemmas \ref{lem:interp_wg}, \ref{lem:norm_equiv} with $\d_1 = 1, \d_2 = -2 - \kp_1$. 
\begin{lemma}\label{lem:Xm_chain}
For $n > m$, we have $||f ||_{\cX^m} \les_n ||f ||_{\cX^n}$ and $\cX^n \subset \cX^n$.
\end{lemma}

\subsubsection{Functional inequalities for non-radial perturbation}
Before proving Theorem \ref{thm:coer_est}, we establish a few 
functional inequalities in Lemmas \ref{lem:norm_equiv_coe}, \ref{lem:non_IBP} to handle the non-radial perturbation. It allows us to control mixed derivative terms, e.g. $\na^2 \D^{m-1} F, \pa_{\xi} \na \D^{m-1} F$, using  the energy based on $\D^m \UU, \D^m \S$.

In the following lemma, we show that $\D^i$ and $\na^{2i}$ can be exchanged up to some lower order term. 
\begin{lemma}\label{lem:norm_equiv_coe}
Let $\d_1 \in (0, 1], \d_2 > 0$, and define $\b_j \teq 2 j \d_1 + \d_2,  \la x \ra = (1 + |x|^2)^{1/2} $. Suppose that the weights $\psi_i, i\geq 0$ satisfies the pointwise bounds  $\psi_{2n} \asymp_n \la x \ra^{\b_{2n}},  |\na \psi_{2n}| \les_n \la x \ra^{\b_{2n}-1}$,
and a function $B(y)$ satisfies $|B(y) | \les C(b) , |\na B| \les C(b) \la y \ra^{-1} $ for some constant $C(b) > 0$ depending on a parameter $b$. Then for any $\e >0$ and $ i, n$ with $0 \leq i \leq n$, there exists a constant $C(\e, n, \d)>0$ such that 
\begin{align}
    \B| \int B(y) & (| \D^n f|^2 - |\na^{2 i} \D^{n-i} f|^2 ) \psi_{2n}  \B|    \leq \e \int |\D^n f | \psi_{2n} + C(\e, n, \d, b) \int f^2 \psi_0 , \label{eq:norm_dif_ff}  \\
    \B| \int B(y) & (  \D^n f \cdot \D^n g  - \na^{2i} \D^{n-i} f \cdot \na^{ 2 i } \D^{n-i} g ) \psi_{2n}  \B| \notag \\
  & \leq \e \int ( |\D^n f |^2 + |\D^n g|^2) \psi_{2n} + C(\e, n, \d, b) \int ( f^2 + g^2) \psi_0  .  \label{eq:norm_dif_fg} 
\end{align}
for any functions $f, g$ sufficiently smooth with suitable decay at infinity, where the k-tensor $\na^k f$ has entry $ \pa_1^{\al_1} \pa_2^{\al_2} .. \pa_{d}^{\al_d} f$ with $\sum \al_i = k$ 
and $\na^k f \cdot \na^k g = \sum_{|\al| = k} \pa^{\al}f \pa^{\al} g$ denotes the standard inner product.
\end{lemma}

In Lemma \ref{lem:non_IBP}, if the coefficient $B(y)$ vanishes near $0$, which compensates 
a singularity $\f{1}{\xi}$ from the commutator $[\pa_{\xi}, \pa_i]$, we can perform integration by parts for derivatives $\pa_{\xi}, \pa_i$ up to some lower order terms.
\begin{lemma}\label{lem:non_IBP}

Let $\d_2 \in \R$, and define $\b_j = 2 j + \d_2, \xi = |y|$. Suppose that the weights $\psi_j$ satisfy the pointwise bounds $\psi_{j} \asymp_j \la y \ra^{\b_{j}}, |\na \psi_{j} | \les \la y \ra^{\b_{j} - 1}$, and a function $B(y)$ satisfies $|B(y) | \les C(b) \min(1, |y|), |\na B| \les C(b) \la y \ra^{-1} $ for some constant $C(b) > 0$ depending on a parameter $b$. For any $\e > 0, n \geq 0$ and any $i, j \in \{1,2,.., n\}$, we have 
\bseq
\begin{align}
\B| \int B(y)  \pa_{\xi}^2  \D^{n-1} f \pa_{jj} \D^{n-1} g  \psi_{2n} - 
 \int B(y) \pa_{\xi} \pa_j \D^{n-1} f \pa_{\xi} \pa_j \D^{n-1} g \psi_{2n} \B| & \leq 
\cE_{\e, n, \d,b}(f, g) , \label{eq:non_IBP_RR} \\  
 \B|\int B(y) \pa_{\xi} \pa_i f \pa_{jj } g \psi_{2n} - \int B(y) \pa_i \pa_j f \pa_{\xi} \pa_{j} g \psi_{2n} \B| & \leq \cE_{\e, n, \d,b}(f, g), 
\label{eq:non_IBP_Ria} \\
\B|\int B(y)  \pa_{\xi} \pa_i f \pa_{jj } g \psi_{2n} - \int B(y)  \pa_{\xi} \pa_j f \pa_{ij} g \psi_{2n} \B|
& \leq \cE_{\e, n, \d,b}(f, g) ,  \label{eq:non_IBP_Rib} 
\end{align}
\eseq
 for any functions $f, g$ sufficiently smooth with suitable decay at infinity, where $\cE$ is defined as follows
\beq\label{eq:non_IBP_err}
\cE_{\e, n, \d, b}(f, g) = \e \int ( |\D^n f |^2 + |\D^n g|^2) \psi_{2n} + C(\e, n, \d, b) \int ( f^2 + g^2) \psi_0   .
\eeq
 with some constant $C(\e, n, \d, b)>0$. Note that the repeated indices on $j$ in \eqref{eq:non_IBP} \textit{does not} denote summation over $j$.
\end{lemma}

We defer the proofs of the above Lemmas to Appendix \ref{sec:norm_dif}.

\subsubsection{Proof of Theorem \ref{thm:coer_est}}\label{sec:thm_coer_pf}

With Lemmas \ref{lem:norm_equiv_coe} and \ref{lem:non_IBP}, we are now ready to prove Theorem~
\ref{thm:coer_est}.

\begin{proof}

Firstly, we introduce the following notations. Denote 
\bseq\label{eq:nota_coer}
\beq\label{eq:nota_coera}
  \xi =   |y|,\quad  D_{\xi} = \xi \pa_{\xi}, \quad  \bar F = \xi + \bar U(\xi),  \quad  \bar \FF = y + \bar \UU = \bar F e_R . \\
\eeq
We introduce $\G_m, \G_{m, (\xi)}$ for the top order terms 
\beq\label{eq:nota_coer_topE}
 \G_m =  |\na^2 \D^{m-1} \UU|^2 + |\na^2 \D^{m-1} \S|^2 ,
\quad \G_{m, (\xi)} = |\pa_{\xi} \na \D^{m-1} \UU|^2 + | \pa_{\xi} \na \D^{m-1} \S|^2 ,
\eeq
and introduce $\cE_{m, \e}$ to track lower order terms
\beq\label{eq:nota_coer_err}
 \cE_{m, \e} = \e \int (|\D^m \UU|^2 + |\D^m \S|^2) \vp_{2m}^2 \vp_g  + C(m , \e) \int ( |\UU|^2 + |\S|^2 ) \vp_g , 
\eeq
where $C(m, \e)$ depending on $\e, m$ can change from line to line.
\eseq

Applying $\D^m$ to the linearized operators $\cL_U, \cL_{\S}$ defined in \eqref{eq:linc}, \eqref{eq:lind}, and using Lemma \ref{lem:leib} to extract the leading order parts from the terms containing $\na \UU, \na \S$,  we get 
\beq\label{eq:lin_Hk}
\bal
\D^m \cL_U & = \underbrace{- (y + \bar \UU) \cdot \na \D^m \UU
- \al \bar \S \na \D^m \S }_{\cT_{U}}
\underbrace{- \D^m  \UU \cdot \na \bar \UU 
- \al \D^m \S  \na \bar \S}_{ \cS_U}
+ \cR_{U, m} \\
& \quad   \underbrace{ - (r -1) \D^m \UU
- 2 m \f{\bar F}{\xi} \D^m \UU - 2 m D_{\xi} (\f{\bar F}{ \xi } ) \pa_{\xi \xi} \D^{m-1} \UU
-  2 m \al \pa_{\xi} \bar \S \pa_{\xi} \na \D^{m-1} \S  }_{ \cD_{U} } , \\
\D^m \cL_{\S} & = \underbrace{  - (y + \bar \UU) \cdot \na \D^m \S
- \al \div(\D^m \UU) \bar \S }_{\cT_{\S}} \underbrace{ - \D^m \UU \cdot \na \bar \S 
- \al \div(\bar \UU) \D^m \S }_{\cS_{\S}} + \cR_{\S, m} \\
& \quad  \underbrace{- (r-1) \D^m \S 
- 2 m \f{ \bar F}{ \xi } \D^m \S 
- 2 m D_{\xi} ( \f{\bar F}{\xi } ) \pa_{\xi \xi} \D^{m-1} \S
- 2 m \al \pa_{\xi} \bar \S  \pa_{\xi} \div( \D^{m-1} \UU )   }_{ \cD_{\S}}, \\
\eal
\eeq
where $\cR_{U, m}, \cR_{\S, m}$ denote the lower order terms. We have used the notations $\cT, \cD, \cS$ to single out {\em transport}, {\em dissipative}, and {\em stretching} terms.

Using Lemma \ref{lem:leib} and the decay estimates \eqref{eq:dec_U}, we obtain 
\beq\label{eq:lin_lower}
\bal
 |\cR_{U, m}| 
& \les_m \sum_{ 1 \leq i \leq 2 m-1}  | \na^{ 2 m + 1- i}  \bar \UU | \cdot | \na^i \UU |  
+ |\na^{2m + 1 - i} \bar \S| \cdot |\na^i \S|  \\
& \les_m \sum_{1 \leq i \leq 2m - 1} \la y \ra^{- r - 2m + i} 
( | \na^i \UU |   +  |\na^i \S|  ) , \\
 |\cR_{\S, m}| 
& \les_m \sum_{ 1 \leq i \leq 2 m-1}  | \na^{ 2 m + 1- i}  \bar \UU | \cdot | \na^i \S |  
+ |\na^{2m + 1 - i} \bar \S| \cdot |\na^i \UU|  \\
& \les_m \sum_{1 \leq i \leq 2m - 1} \la y \ra^{- r - 2m + i} 
( | \na^i \UU |   +  |\na^i \S|  ) . \\
 \eal
\eeq

Next, to bound the left hand side of \eqref{eq:coer_est}, we perform weighted $H^{2m}$ estimates with weight $\vp_{2m}^2 \vp_g$. To this end, we estimate the term 
\beq\label{eq:lin_Hk_EE}
 I_{\cL} 
  = 
  \int ( \D^m \cL_U \cdot \D^m \UU 
 + \D^m \cL_{\S} \cdot \D^m \S   ) 
 \vp_{2m}^2 \vp_g . 
\eeq
Our goal is to derive the main terms as integrals of $ |\na^2 \D^{m-1} G|^2, |\pa_{\xi} \na \D^{m-1} G|^2$ for  $G = \UU$ or $\S $.


\vspace{0.1in}
\paragraph{\bf{Estimate of $\cT_U, \cT_{\S}$}}

Using the identity
\beq\label{eq:lin_cross1}
 \na \D^m \S \cdot \D^m \UU + \div(\D^m \UU) \D^m \S = \na \cdot ( \D^m \UU \cdot \D^m \S)
\eeq
and integration by parts, we obtain 
\[
\bal
I_{\cT} & = -\int \B( (y + \bar \UU) \cdot \na \D^m \UU \cdot \D^m \UU  
+ (y + \bar \UU) \cdot \na \D^m \S \cdot \D^m \S  \\
& \qquad + \al  \bar \S \cdot \na \D^m \S \cdot \D^m \UU
+ \al \bar \S \div( \D^m \UU) \D^m \S 
\B) \vp_{2m}^2 \vp_g \\ 
& = \int \f{1}{2}  \f{ \na \cdot ( (y + \bar \UU) \vp_{2m}^2 \vp_g) }{ \vp_{2m}^2 \vp_g }
( |\D^m \UU|^2 + |\D^m \S|^2  ) \vp_{2m}^2 \vp_g 
+ \f{ \na (\al \bar \S \vp_{2m}^2 \vp_g)}{ \vp_{2m}^2 \vp_g}\cdot \D^m \UU  \D^m \S \vp_{2m}^2 \vp_g  .
\eal
\]

Below, we further estimate the coefficients. Denote 
\beq\label{eq:coer_est_J}
J_1(y)  \teq \f{1}{2}  \f{ \na \cdot ( (y + \bar \UU) \vp_{2m}^2 \vp_g) }{ \vp_{2m}^2 \vp_g },
\quad 
J_2(y) = \f{ \na( \al \bar \S \vp_{2m}^2 \vp_g)}{ \vp_{2m}^2 \vp_g} .
\eeq

Recall $\vp_g = \la y \ra^{-2-\kp_1}, \kp_1 = \f{1}{4}$ from \eqref{norm:Xk}. Using the decay estimates in \eqref{eq:dec_U}, the outgoing property $\xi + \bar U>0$ \eqref{eq:rep2}, and $\kp_1>0, r > 1$, and denoting 
\footnote{The fact that $r<2$ follows since for $d=2$ (the case of this paper) the inequality $r<r_{\sf eye}(\alpha)$ implies $r<2$, see~\eqref{r-range}.}
\beq\label{eq:kappa3}
\kp_2 = \min(2, r) = r > 1,
\eeq
we obtain
\beq\label{eq:coer_est_vpg}
 \f{\pa_{\xi} \vp_g}{\vp_g} = (-2 - \kp_1) \f{\xi}{1 + \xi^2} , \quad 
 (\xi + \bar U)  \f{\pa_{\xi} \vp_g}{\vp_g}
 \leq \xi  \f{\pa_{\xi} \vp_g}{\vp_g} + C \la \xi \ra^{-r}
 \leq  - 2 + C \la \xi \ra^{-\kp_2 }.
\eeq
Recall $\vp_{2m} = \vp_1^{2m}$ from Lemma \ref{cor:wg}. Applying 
\eqref{eq:coer_est_vpg}, decay estimates in \eqref{eq:dec_U}, and $\na f = \pa_{\xi} f \ee_R$ for a radially symmetric function, we estimate $J_i$  in \eqref{eq:coer_est_J}

\[
\bal
 J_1(y) & 
= \f{1}{2} \B( 2 + \div(\bar \UU)
+ 4m (\xi + \bar U) \f{\pa_{\xi} \vp_1}{\vp_1}
 + (\xi + \bar U) \f{\pa_{\xi} \vp_g}{\vp_g} \B)  \\
  & \leq 1 + C \la \xi \ra^{-r}
  + 2m (\xi + \bar U) \f{ \pa_{\xi} \vp_1}{\vp_1}
  +  \f{1}{2} \cdot (-2) 
  =  C \la \xi \ra^{-r}+ 2m (\xi + \bar U) \f{ \pa_{\xi} \vp_1}{\vp_1} , \\
 |J_2(y)|  
&  \leq \al |\na \bar \S| + \al \bar \S \B( 4m \f{ |\na \vp_1|}{\vp_1}
 + \f{ |\na \vp_g| }{ \vp_g } \B) 
 \leq 4m \al \bar \S  \f{ |\pa_{\xi}\vp_1|}{\vp_1}
 + C \la \xi \ra^{-r} .
  \eal
\]
Using Lemma \ref{cor:wg} for $\vp_1$,  $\vp_g$ from \eqref{norm:Xk}, and the decay estimates \eqref{eq:dec_U}, we have
\[
|J_i| \les_m 1, \quad  |\na J_i(y)| \les_m \la y \ra^{-1} , \quad i = 1, 2.
\]

Applying Lemma \ref{lem:norm_equiv_coe} with $B = J_i, \psi_{2m} = \vp_{2m}^2 \vp_g,\b_j = 2 j - ( 2 + \kp_1)/2 , b = m$, we obtain 
\[
I_{\cT} \leq \int \B( J_1(y) (|\na^2 \D^{m-1} \UU|^2 + |\na^{2} \D^{m-1} \S|^2 ) 
+ J_2(y) \cdot \na^2 \D^{m-1} \UU   \na^{2} \D^{m-1} \S \B) \vp_{2m}^2 \vp_g  
+ \cE_{m, \e} ,
\]
where $J_2(y) \cdot \na^2 \D^{m-1} \UU  \cdot \na^{2} \D^{m-1} \S 
= \sum_{i=1}^d (J_{2}(y))_{, i} \na^2 \D^{m-1} \UU_{, i} \cdot \na^{2} \D^{m-1} \S $.

Combining the above estimates, applying $|a b| \leq \f{1}{2}(a^2 + b^2) $ to $\D^m \UU \D^m \S$, and using the notation $\G_m$ for $\UU, \S$ \eqref{eq:nota_coer}, we get
\beq\label{eq:lin_est1}
\bal
I_{\cT} & \leq \int \f{1}{2} \B(  \f{ \na \cdot ( (y + \bar \UU) \vp_{2m}^2 \vp_g) }{ \vp_{2m}^2 \vp_g }
 +  \f{ | \na (\al \bar \S \vp_{2m}^2 \vp_g) | }{ \vp_{2m}^2 \vp_g}
\B) ( | \na^2 \D^{m-1} \UU|^2 + |\na^2 \D^{m - 1} \S|^2  ) \vp_{2m}^2 \vp_g  \\
& \leq \int \B(2m \B( (\xi + \bar U) \f{\pa_{\xi} \vp_1}{\vp_1} 
+ \al \bar \S | \f{\pa_{\xi} \vp_1}{\vp_1} | \B) + C \la \xi \ra^{- \kp_2 } \B) 
\G_m\vp_{2m}^2 \vp_g .
\eal
\eeq

\vspace{0.1in}
\paragraph{\bf{Estimate of $\cD_U, \cD_{\S}$ }}

Recall $\cD_U, \cD_{\S}$ from \eqref{eq:lin_Hk} and $\bar F = \xi + \bar U$ \eqref{eq:nota_coer}. Using \eqref{eq:dec_U}, we obtain 
\beq\label{eq:coer_est_F}
\bga
|\bar F / \xi| \les 1,  \quad  
| \na (\bar F / \xi ) | \les \la y \ra^{-1}, \quad 
  | \xi \pa_{\xi} (\bar F / \xi )| \les \min(|y|, 1),
  \quad  |\na   ( \xi \pa_{\xi} (\bar F / \xi )) | \les \la y \ra^{-1}, \\
  |\pa_{\xi} \bar \S|  \les \min( |y|, 1), \quad 
  \quad |\na \pa_{\xi} \bar \S| \les \la y \ra^{-1}. 
  \ega
  \eeq

Next, we apply Lemmas \ref{lem:norm_equiv_coe}, \ref{lem:non_IBP} for integration by parts to each term in 
\[
I_{\cD} = \int ( \cD_U \D^m \UU + \cD_{\S} \D^m \S  )  \vp_{2m}^2 \vp_g ,
\]
and rewrite it as integrals involving $\na^2 \D^{m-1} G,\pa_{\xi} \na \D^{m-1} G, G = \UU, \S $. 
Applying Lemma \ref{lem:norm_equiv_coe} with $i =1, 
\psi_{2m} = \vp_{2m}^2 \vp_g,  B = - (r-1) - 2m \f{\bar F}{\xi}, b = m$, and the estimate of $B$ in \eqref{eq:coer_est_F}, we obtain 
\beq\label{eq:coer_est_ID1}
\bal
 & \int ( -(r-1) - 2 m \f{F}{ \xi } )\B( |\D^m \UU|^2 + |\D^m \S|^2 \B) \vp_{2m}^2 \vp_g \\
 & \quad \leq  \int ( -(r-1) - 2 m \f{F}{ \xi } )\B( | \na^2 \D^{m-1} \UU|^2 + |\na^2 \D^{m-1} \S|^2 \B)  \vp_{2m}^2 \vp_g + \cE_{\e, m} .\\
 \eal
 \eeq
 Applying Lemma \ref{lem:non_IBP} with $\psi_{2m}^2 = \vp_{2m}^2 \vp_g, \b_j = 2 j - \f{2 + \kp_1}{2}$, $B = 2m D_\xi(\bar F/ \xi ) $ and estimates of $B$ in \eqref{eq:coer_est_F} ($B$ vanishes near $y=0$), $b = m$, we get
 \beq\label{eq:coer_est_ID2}
 \bal
& \int ( - 2 m D_\xi (\f{F}{ \xi} )  ( \pa_{\xi\xi} \D^{m-1} \UU \cdot \D^m \UU 
+ \pa_{ \xi\xi }  \D^{m-1} \S \cdot \D^m \S )  \vp_{2m}^2 \vp_g \\
& \quad \leq \int ( - 2 m D_\xi (\f{F}{ \xi} )  ( | \pa_{ \xi} \na \D^{m-1} \UU|^2 
+ | \pa_{\xi} \na  \D^{m-1} \S|^2  )  \vp_{2m}^2 \vp_g + \cE_{\e, m}  .\\
\eal
\eeq

Similarly, we apply Lemma \ref{lem:non_IBP} \eqref{eq:non_IBP_Ria}, \eqref{eq:non_IBP_Rib} with $ B = 2 m \pa_{\xi} \bar \S$ and estimates of $B$ in \eqref{eq:coer_est_F} to obtain 
\beq\label{eq:lin_cross2}
\bal
& I_{\cD, 1} = \int (-  2m \al \pa_\xi \bar \S \pa_\xi \pa_i \D^{m-1} \S ) \D^{m} \UU_i  \vp_{2m}^2 \vp_g
\leq \int (-  2m \al \pa_\xi \bar \S \pa_i \pa_j \D^{m-1} \S )  \pa_\xi \pa_j \D^{m-1} \UU_i \vp_{2m}^2 \vp_g + \cE_{\e, m} ,\\
& I_{\cD, 2} = \int( -2m \al \pa_\xi \bar \S \pa_\xi \div(\D^{m-1} \UU)) \D^m \S \vp_{2m}^2 \vp_g 
\leq \int( -2m \al \pa_\xi \bar \S \pa_\xi \pa_j (\D^{m-1} \UU_i)) \pa_i \pa_j \D^{m-1} \S \vp_{2m}^2 \vp_g  + \cE_{\e, m}.
 \eal
\eeq
Here, if an index appears twice, it means summation over such an index. 

Recall the notations $\G_{m} = |\na^2 \D^{m-1} \UU|^2 + |\na^2 \D^{m-1} \S|^2, \G_{m, (\xi)}
= |\pa_\xi \na \D^{m-1} \UU|^2 + | \pa_\xi \na \D^{m-1} \S|^2$ from \eqref{eq:nota_coer}.
Note that $|\na f| \geq |\ee_R \cdot \na f| = |\pa_\xi f|$. Applying 
\[
\sum_{1\leq i, j \leq d} 
|\pa_i \pa_j \D^{m-1} \S  \cdot \pa_\xi \pa_j \D^{m-1} U_i|
\leq \f{1}{2} (|\na^2 \D^{m-1} \S|^2 + |\pa_\xi \na \D^{m-1} \UU|^2) 
\leq \f{1}{2} \G_m 
\]
to $I_{\cD,1}, I_{\cD,2}$ and then combining \eqref{eq:coer_est_ID1}, \eqref{eq:coer_est_ID2}, 
and the above estimates for $I_{\cD, i}$, we obtain 
\beq\label{eq:lin_est2D}
I_{\cD} 
\leq  \int \B( (- (r-1) - 2 m \f{F}{\xi} ) \G_m 
 - 2 m D_\xi (\f{F}{\xi}) \G_{m, (\xi) }
 + 2 m \al |\pa_\xi \bar \S| \G_m \B) \vp_{2m}^2 \vp_g + \cE_{\e, m} .
\eeq



\vspace{0.1in}
\paragraph{\bf{Estimate of $\ \cS_U, \cS_{\S}$}}

Recall $ \cS_U, \cS_{\S}$ from \eqref{eq:lin_Hk}. Using the decay estimates \eqref{eq:dec_U}, we get 
\[
|\cS_U|, |\cS_{\S}| \les \la \xi \ra^{-r} (|\D^m \UU| + |\D^m \S|) ,
\]
and obtain 
\beq\label{eq:lin_est2}
\bal
I_S &= \int \B(   \cS_U  \cdot   \D^m \UU +  
   \cS_{\S}   \D^m \S \B) \vp_{2m}^2 \vp_g  \leq  C \int \la \xi \ra^{-r} ( |\D^m \UU|^2 + |\D^m \S|^2) \vp_{2m}^2 \vp_g .
\eal
\eeq

\vspace{0.1in}
\paragraph{\bf{Estimate of $\cR_U, \cR_{\S}$}}

Recall the remaining terms $ \cR_U, \cR_{\S}$ from \eqref{eq:lin_Hk} with estimates \eqref{eq:lin_lower}. Using $\vp_{2m} \asymp \la y \ra^{2m}$ from Lemma \ref{cor:wg} 
and recalling $\vp_g$ from \eqref{norm:Xk}, we obtain
\[
\vp_{2m}^2 \vp_g \la y \ra^{-r-2m + i} \asymp \la \xi \ra^{2m + i - r - \kp_1} 
= \la \xi \ra^{2m + i + 2 \d_2},  \quad \d_2 =\f{ -2 - \kp_1 - r}{2} .
\]

Applying interpolation in Lemma \ref{lem:interp_wg} and Lemma \ref{lem:norm_equiv} with $\d_1 = 1$
and $\d_2$ given above, for $1 \leq i \leq 2m-1$ and any $\e > 0$, we obtain 
\[
\bal
  \int & \la y \ra^{-r - 2m + i} |\na^i F| |\D^m G | \vp_{2m}^2 \vp_g  \\
  & \leq    \e  || \D^m G \cdot \la y \ra^{2m + \d_1} ||^2_{L^2}  
 + C_{m ,\e } || \na^i F \la y \ra^{i + \d_1} ||^2_{L^2} \\
 & \leq 
 \e  || \D^m G \cdot \la y \ra^{2m + \d_1} ||_{L^2}^2
 + \e || \na^{2 m} F \cdot \la  y\ra^{2m + \d_1} ||_{L^2}^2
  + C_{m, \e} ||  F \cdot \la y \ra^{ \d_1} ||_{L^2}^2 \\
   & \leq   2 \e ( || \D^m G \cdot \la y \ra^{2m + \d_1} ||_{L^2}^2 
   + || \D^m F \cdot \la y \ra^{2m + \d_1} ||_{L^2}^2 )
  + C_{m, \e} ||  F \cdot \la y \ra^{ \d_1} ||_{L^2}^2 .
\eal
\] 
Since $\e > 0$ is arbitrary, applying the above estimates to each term in \eqref{eq:lin_lower}, 
and then using $ \la y \ra^{2( 2m + \d_1)} \les_m \la y \ra^{-r} \vp_{2m}^2 \vp_g, \vp_0 = 1$, we get 
\beq\label{eq:lin_est3}
\bal
I_{\cR} &= \int (\cR_{U, m}  \cdot \D^m \UU + \cR_{\S, m} \cdot \D^m \S) \vp_{2m}^2 \vp_g \\
& \leq \int \la \xi \ra^{-r} \B(  ( | \D^{m} \UU|^2 + | \D^{m} \S|^2 ) \vp_{2m}^2 \vp_g
+ C_{m}  (| \UU |^2 + \S^2)  \vp_g \B)  \\
 & \leq \int 4 \la \xi \ra^{-r} \B( \G_m \vp_{2m}^2 \vp_g
 + C_{m}  (| \UU |^2 + \S^2)  \vp_g \B)  ,
\eal
\eeq
where we have used the notation $\G_m$ from \eqref{eq:nota_coer} and $|\D f| \leq 4 |\na^2 f|$ in the last inequality.

\vspace{0.1in}
\paragraph{\bf{Summary of $H^{2m}$ estimates}}

Plugging \eqref{eq:lin_est1}, \eqref{eq:lin_est2}, \eqref{eq:lin_est3} in \eqref{eq:lin_Hk_EE}, 
we establish 
\beq\label{eq:lin_coerHk0}
\bal
& \int ( \D^m \cL_U \cdot \D^m \UU 
 + \D^m \cL_{\S} \cdot \D^m \S   ) 
 \vp_{2m}^2 \vp_g = I_{\cL} = I_{\cD} + I_S  + I_{\cD} +  I_{\cR}  \\
 &   \leq \int \B\{  - (r-1) \G_m
 + 2m \B( (\xi + \bar U) \f{\pa_{\xi} \vp_1}{\vp_1} 
+ \al \bar \S | \f{\pa_{\xi} \vp_1}{\vp_1} | 
-  \f{ \bar F}{\xi} + \al |\pa_\xi \bar \S| \B) \G_m  \\
& \quad   - 2 m D_\xi (\f{\bar F}{ \xi})  \G_{m, (\xi)}  + C \la \xi \ra^{- \kp_2 } \G_m 
+ \e \G_m \B\} \vp_{2m}^2 \vp_g 
+ C_{m, \e} (|\UU|^2 + |\S|^2 ) \vp_g  , \\
\eal
\eeq
where $\kp_2$ is defined in \eqref{eq:kappa3}. We need to further handle the term $\G_{m, (\xi)}$. We focus on the main terms involving $m$ in \eqref{eq:lin_coerHk0}, which are given by 
$2m H_1$ with $H_1$ defined as
\[
H_1 =  (J  -  \f{ \bar F}{\xi} + \al |\pa_\xi \bar \S| ) \G_m 
-  D_\xi (\f{\bar F}{ \xi})  \G_{m, (\xi) } ,
\quad J(y) = (\xi + \bar U) \f{\pa_{\xi} \vp_1}{\vp_1} 
+ \al \bar \S | \f{\pa_{\xi} \vp_1}{\vp_1} | .
\]

From the definition of $\bar F = \f{\bar U}{\xi}$ \eqref{eq:nota_coera} 
and $D_\xi = \xi \pa_\xi$ \eqref{eq:nota_coer}, we have 
\bseq\label{eq:coer_est_H1}
\beq
 \f{\bar F}{\xi} = 1 + \f{\bar U}{\xi}, 
 \quad  
D_\xi (\f{\bar F}{ \xi} ) +  \f{\bar F}{ \xi}   = \pa_\xi \bar F 
= 1 + \pa_\xi \bar U .
\eeq

Using  $ \G_m \geq \G_{m, (\xi) }$ (since $|\na f| \geq |\pa_\xi f|$) and \eqref{eq:repul_wg_agb}, \eqref{eq:repul_wg_agc} with $l = 1$, we can estimate $H_1(y)$ as follows 
\beq
\bal
 H_1 & = (J-  \f{ \bar F}{ \xi} + \al |\pa_\xi \bar \S| ) (\G_m - \G_{m, (\xi)} + \G_{m, (\xi) })
-  D_\xi (\f{\bar F}{\xi})  \G_{m, (\xi) } 
 \\
&= ( J - \f{\bar F}{\xi} + \al |\pa_\xi \bar \S| ) ) (\G_m -\G_{m, (\xi) })
+ ( J - \pa_\xi \bar F + \al |\pa_\xi \bar \S| ) \G_{m, (\xi) }  \\
& \leq - \mu_1 \la y \ra^{-1} (\G_m -\G_{m, (\xi) }) 
- \mu_1  \la y \ra^{-1} \G_{m, (\xi) } \leq -\mu_1 \la y \ra^{-1} \G_m.
\eal
\eeq
\eseq

Plugging the above estimates to \eqref{eq:lin_coerHk0}, we obtain
\[
I_{\cL} \leq \int ( -(r-1) -  2m \mu_1 \la y \ra^{-1} + 
a_1 \la \xi \ra^{- \kp_2 }  + \e) \G_m \vp_{2m}^2 \vp_g + C_{m, \e } (|\UU|^2 + |\S|^2 ) \vp_g
\]
for some absolute constant $a_1 > 0$ independent of $m, \e$.
We choose $\e = \f{r-1}{2}$. Since $\kp_2 > 1$ (see \eqref{eq:kappa3}),  there exists $m_0$ sufficiently large, e.g $m_0 =\lfloor \f{a_1}{2\mu_1} \rfloor +1$, such that for any $ m \geq m_0$, we get
\[
- (r-1) - 2m \mu_1 \la \xi \ra^{-1} + a_1 \la \xi \ra^{- \kp_2}  + \f{r-1}{2}
\leq - \f{1}{2}(r-1) + (a_1 -  2m_0 \mu_1) \la y \ra^{-1} 
\leq - \f{1}{2} (r-1).
\]

Combining the above two estimates and \eqref{eq:lin_Hk_EE}, and then use \eqref{eq:norm_dif_ff} in Lemma \ref{lem:norm_equiv_coe} with $i = 1, B(y) =1, n = m$ to bound the energy of $\G_m$ \eqref{eq:nota_coer_topE}, 
we prove  
\beq\label{eq:lin_coerHk}
\bal
\int ( \D^m \cL_U \cdot \D^m \UU  &+ \D^m \cL_{\S} \cdot \D^m \S   ) \vp_{2m}^2 \vp_g 
 = I_{\cL}  
  \leq 
 - \f{r-1}{2} \int  \G_m \vp_{2m}^2 \vp_g + C_{m } (|\UU|^2 + |\S|^2 ) \vp_g   \\
&  \leq - \f{r-1}{4} \int ( |\D^m \UU|^2 + |\D^m \S|^2 ) \vp_{2m}^2 \vp_g + C_{m } (|\UU|^2 + |\S|^2 ) \vp_g .
 \eal
\eeq

\vspace{0.1in}
\paragraph{\bf{Weighted $L^2$ estimates} }
For $m = 0$, we do not have the lower order terms $\cR_{U, 0}, \cR_{\S, 0}$ \eqref{eq:lin_lower} and do not need to estimate $I_{\cR}$ \eqref{eq:lin_est3}. Moreover, we do not need to apply Lemmas \ref{lem:norm_equiv_coe}, \ref{lem:non_IBP} for various integration by parts estimates, which lead to error terms $\cE_{\e, m}$ \eqref{eq:nota_coer_err}. For example, for $m=0$, we do not need to perform the estimate \eqref{eq:coer_est_ID2} using Lemma \ref{lem:non_IBP}. Thus, combining \eqref{eq:lin_est1}, \eqref{eq:lin_est2D}, \eqref{eq:lin_est2}, we obtain 
\beq\label{eq:lin_coerL2}
 \int ( \cL_U \cdot \UU + \cL_{\S} \cdot \S )    \vp_g 
\leq \int  (- (r-1) + \bar C \la \xi \ra^{-\kp_2}) ( | \UU|^2 + |\S|^2 ) \vp_g 
\eeq
for some absolute constant $\bar C > 0$ independent of $m$.

Since $r > 1, \kp_2 > 1$ (see \eqref{eq:kappa3}), there exists $R_4$ sufficiently large  such that for $ |y| \geq R_4$, we get 
\[
- (r-1) + \bar C \la \xi \ra^{-\kp_2} \leq 
- (r-1) + \bar C \la R_4 \ra^{-\kp_2} \leq - \f{r-1}{2}.
\]

Therefore, combining the above two estimates, we arrive at 
\[ 
 \int ( \cL_U \cdot \UU + \cL_{\S} \cdot \S )    \vp_g  \leq  \int ( \bar C\one_{|y| \leq R_4}  - \f{r-1}{4} )  ( | \UU|^2 + |\S|^2 ) \vp_g .
\]

\paragraph{\bf{Choosing the $\e_m$} } In order to conclude the proof of~\eqref{eq:coer_est},  we combine \eqref{eq:lin_coerHk} and \eqref{eq:lin_coerL2}. Choosing $\e_m$ sufficiently small, e.g.~$\e_m = \f{r-1}{ 8 C_m}$, where $C_m$ is as in \eqref{eq:lin_coerHk}, 
multiplying \eqref{eq:lin_coerHk} with $\e_m$ and then adding to~\eqref{eq:lin_coerL2}, we deduce \eqref{eq:coer_est} for $\lambda \in (0, \f{1}{2})$ sufficiently small which is independent of $m$, e.g.~$\lambda = \frac{r-1}{8}$ ($< \f{1}{2}$ due to \eqref{r-range}).
\end{proof}

\subsection{Compact perturbation and semigroup} 

In this section, we construct the compact perturbation $\cK_m$ to $\cL$ and the semigroups generated by $\cL, \cL - \cK_m$ following \cite{chen2024Euler}. 
We remark that the arguments in \cite{chen2024Euler} for these steps do not rely on any symmetry assumptions on the solution $(\UU, \S)$ to the linearized equation \eqref{eq:lin}.



\subsubsection{Compact perturbation via Riesz representation}

In this section, using the estimates established in Theorem~\ref{thm:coer_est}, we construct a compact operator $\cK_m$ such that $\cL - \cK_m$ is dissipative in $\cX^m$. We fix $m_0\geq 6$. 
The following result was first proved in \cite[Proposition 3.4]{chen2024Euler} for
radially symmetric functions $(\UU, \S)$.


\begin{proposition}\label{prop:compact}
For any $m \geq m_0 $, there exists a bounded linear operator $\cK_m \colon \cX^0 \to \cX^m$ with:
\begin{itemize}
\item[(a)] 
for any $f \in \cX^{0}$ we have
\[
{\rm supp}(\cK_m f) \subset B(0, 4 R_4),
\]
where $R_4$ is chosen in Theorem~\ref{thm:coer_est} (in particular, it is independent of $m$);
\item[(b)] the operator $\cK_m$ is compact from $ \cX^{m} \to \cX^m$;  
\item[(c)] the enhanced smoothing property  $\cK_m: \cX^{0} \to \cX^{m+ 3}$ holds;
\item[(d)] the operator $\cL - \cK_m$ is dissipative on $\cX^m$ and we have the estimate
\begin{equation}
 \la (\cL - \cK_m) f ,f \ra_{\cX^m} \leq - \lam \| f \|_{\cX^m}^2
 \label{eq:dissip}
\end{equation}
for all $f \in \{ (\UU,\Sigma) \in \cX^m \colon \cL(\UU,\Sigma) \in \cX^m\}$, $\cL = (\cL_U,\cL_S)$, and where $\lambda >0$ is the  parameter from~\eqref{eq:coer_est} (in particular, it is independent of $m$).
\end{itemize}
\end{proposition}

The proof in \cite{chen2024Euler} is based on the coercive estimates \eqref{eq:coer_est}, 
the Riesz representation theory and the Rellich-Kondrachov compact embedding theorem, and does not require any assumption on the symmetry of $(\UU, \S)$. Thus, the proof is the same.

\subsubsection{Construction of the semigroup}\label{sec:semi}
In this section, for any $m \geq m_0$, we construct strongly continuous semigroups generated by $\cL, \cL - \cK_m$, where $\cK_m$ is constructed in Proposition \ref{prop:compact}.

We define the domains of $\cL, \cD_m $ as follows 
\beq\label{eq:domain}
D_m(\cD_m) = D_m(\cL) \teq \{  (\UU, \S) \in \cX^m,  \cL (\UU, \S) \in \cX^m \}. 
\eeq

Note that the linearized operators $\cL = (\cL_U, \cL_{\S})$ \eqref{eq:lin} are the same as that in \cite{chen2024Euler}. We have the following results, which were first proved in \cite[Section 3.4]{chen2024Euler} for radially symmetric functions. 



\begin{proposition}\label{prop:semi}
Suppose that $m \geq m_0$ and $\cK_m$ is the compact operator constructed in Proposition \ref{prop:compact}. The operators $\cL, \cD_m = \cL - \cK_m : D_m \subset \cX^m \to \cX^m$ generate strongly continuous semigroup 
\[
e^{\cL s} : \cX^m \to  \cX^m ,
\quad 
e^{\cD_m s} : \cX^m \to  \cX^m . 
\]
Moreover, we have the following estimates of the semigroup 
\beq\label{eq:dissp_semi}
 ||  e^{s \cD_m} ||_{\cX^m} \leq e^{-\lam s} , \quad || e^{s \cL} ||_{\cX^m} \leq e^{C_m s}
\eeq
for some $C_m>0$, and the following spectral property of $\cD_m$
\beq\label{eq:spec}
 \{ z: \Re(z) > -\lam \} \subset \rho_{\mw{res}}(\cD_m) ,
\eeq
where $\rho_{\mw{res}}(\cA)$ denotes the resolvent set of $\cA$. 
\end{proposition}

The construction of the semigroup $e^{s \cL }$ in \cite{chen2024Euler} is based on solving the linear PDE of $(\UU, \S)$ \eqref{eq:lin}, which is a symmetric hyperbolic system, and proving its well-posedness: existence, applying the coercive estimates \eqref{eq:coer_est} to obtain uniqueness of the solution and continuous dependence on the initial data. To further construct the semigroup $e^{s  (\cL -\cL_m) }$ based on $e^{s \cL }$, one used the Bounded Perturbation Theorem \cite[Theorem 1.3, Chapter III]{ElNa2000}. 
These properties do not rely on any symmetry assumptions of $( \UU, \S)$. 
In the current  setting without radially symmetry assumptions, the linear PDE with linear operators $\cL = (\cL_U, \cL_{\S})$ defined in \eqref{eq:lin} is still a symmetric hyperbolic system. Moreover, we have the same type of coercive estimate \eqref{eq:coer_est} as that in \cite{chen2024Euler}. Thus, the proof is the same as that in \cite{chen2024Euler}. After we construct the semigroup, the decay estimate of $e^{ \cD_m s}$ follows from the dissipative estimates \eqref{eq:dissip} and implies the estimates of resolvent set \eqref{eq:spec}. Note that the estimates \eqref{eq:dissp_semi}, \eqref{eq:spec} apply to all  $(\cD_m, \cX^m)$ with $m \geq m_0$ with $\lam$ independent of $m$.

\section{Decay estimates of the semigroup}\label{sec:decay}

In this section, we first study two special classes of stable perturbation, and then perform decay estimates for general perturbation. The first class of perturbation studied in Section \ref{sec:pertb_A} satisfies radial symmetry and builds on \cite{chen2024Euler}. It is used crucially to construct non-trivial initial vorticity. The second class studied in Section \ref{prop:far_decay} is supported away from $0$ and is used to construct initial data with appropriate far-field behavior.

\subsection{A stable perturbation with radial symmetry}\label{sec:pertb_A}

Recall from \eqref{eq:polar} the definitions of polar coordinates $\xi, \th, \ee_R, \ee_{\th}$ for $y \in \R^2$. We decompose the velocity in \eqref{eq:lin} into the radial part $\UU^R$ and angular part $\AA$ as 
\[
 \UU = \UU^R + \AA , \quad  \UU^R = (\UU \cdot \ee_R ) \ee_R = U^R(\xi, \th) \ee_R , \quad \AA = 
 (\UU \cdot \ee_{\th} ) \ee_{\th} =  A(\xi, \th) \ee_{\th}. 
\]

We have the identities in the polar coordinates for any sufficient smooth functions $f^R, f^{\th}, v, g$
\[
\bal
& \div( \AA) =  \f{\pa_{\th}}{\xi } A ,
\quad  \div(\UU) = \div(\UU^R) + \div(\AA) ,  \\
& \ee_R \cdot \na g = \pa_R g, \quad \ee_{\th} \cdot \na g = \f{1}{\xi} \pa_{\th} g, 
\quad \AA \cdot \na f(R) =A \pa_{\th} f(R) = 0, 
\\
&  v \ee_R \cdot \na ( f^R \ee_R + f^{\th} \ee_{\th} )
=  v \pa_R  ( f^R \ee_R + f^{\th} \ee_{\th} )
= (v \pa_R f^R ) \ee_R + (v \pa_R f^{\th}) \ee_{\th},
\eal
\]
where $\div$ is the divergence in $\R^2$. Using the above identities and projecting the linear equations 
\beq\label{eq:linlin}
\pa_s (\UU, \S ) = \cL (\UU , \S), \quad \cL = (\cL_U, \cL_{\S}),
\eeq
i.e. \eqref{eq:lin} with $\cN_U= 0, \cN_{\S }= 0$, into $\ee_R, \ee_{\th}$ directions, 
where $\cL= (\cL_U, \cL_{\S})$ are defined in \eqref{eq:linc}, \eqref{eq:lind},  we can derive the linear equations of $U^R, A, \S$ 
\bseq\label{eq:lin_axi}
\begin{align}
 \pa_s U^R + ( y + \bar \UU) \cdot \na U^R   & = -(r-1) U^R
-  U^R \pa_\xi \bar U
  - \al ( \S  \pa_\xi \bar \S  + \bar \S \pa_\xi \S)  , \\
\pa_s A + ( y + \bar \UU) \cdot \na A  & = -(r-1) A   - \f{\bar U A}{\xi}
-  \f{ \al \bar \S}{ \xi} \pa_{\th} \S , \label{eq:lin_axib} \\
 \pa_s \S + (y + \bar \UU) \cdot \na \S 
& = -(r-1) \S - U^R \pa_{\xi}  \bar\S- \al \S \div(\bar \UU)
-  \al \bar \S \B( \div( \UU^R ) + \f{1}{\xi} \pa_{\th} A \B) .
 \end{align}
\eseq
There is a special class of solution to \eqref{eq:lin_axi} 
\[
U^R , \  \S \equiv 0 , \quad  A(\xi, \th, s) = A_*(\xi, s),
\]
for some radially symmetric function $A_*$, i.e. $A, A_*$ are independent of $\th$, with $A_*(0, s) = 0$, which is preserved by \eqref{eq:lin_axi}. It has been observed crucially in \cite{chen2024Euler} that the linearized equation of $A$ 
\eqref{eq:lin_axib} is decoupled from those of $(U^R, \S)$ for radially symmetric solution. For this special class of solution, we can rewrite \eqref{eq:lin_axib} in the vector form by multiplying \eqref{eq:lin_axib} with $\ee_{\th}$, which is the equation of $\AA = A_*(\xi, s) \ee_{\th}$
\beq\label{eq:lin_A}
 \pa_t \AA = \cL_A \AA, \quad \cL_A \AA = - (y + \bar \UU) \cdot \AA - (r-1) \AA - \AA \cdot \na \bar \UU.
\eeq


The following coercive estimates for $\cL_A$ and radially symmetric $\AA$ was 
established in \cite[Theorem 3.2]{chen2024Euler}. 

\begin{theorem}[Theorem 3.2 \cite{chen2019finite2}]\label{thm:coerA0}
Let $\vp_g$ be the weight defined in \eqref{norm:Xkb}. There exists $\lam_A > 0$ and a weight $\vp_A$ satisfying $\vp_A(y) \asymp |y|^{-\b_1} \la y \ra^{\b_1 - 2 -\kp_1}$ for some $\b_1 \in (3, 4)$, and $\vp_g \les \vp_A, \vp_g \asymp \vp_A$ for $|y| \geq 1$, such that for any $ \AA = A(\xi) \ee_{\th} $ with $\AA, \cL_A \AA \in L^2(\vp_A)$, we have
\[
 \int \cL_A \AA  \cdot \AA \vp_A \leq - \lam_A  \int |\AA|^2 \vp_A  .
\]

\end{theorem}

In \cite[Theorem 3.2]{chen2024Euler}, weighted $H^{2m}$ coercive estimates for $\AA$ are also established. 
We do not use such estimates since we use weights $\vp_{2m}$ different from those in \cite{chen2024Euler}. Instead, we have the following results.  


\begin{theorem}\label{thm:coerA}
There exists $m_0 \geq 6, \lam >0$ such that the following statements hold true. For any $m \geq m_0$, there exists $\e_m =\e_m(m_0, \lam)$ such that for any $ \AA = A(\xi) \ee_{\th} $ with $\AA \in \cX_A^m, \cL_A \AA \in \cX_A^m$, we have
\beq\label{eq:coer_estA}
 \la  \cL_A \AA, \AA  \ra_{\cX_A^m} 
 \leq - \lam || \cA ||_{\cX_A^m}^2 ,
\eeq
 Here, the Hilbert spaces $\cX_A^m$ is defined by the inner products 
\beq\label{norm:Xk_A}
\bal 
 \la f , g \ra_{\cX_A^m} \teq \int \e_m \D^m f \cdot \D^m g \wh \vp_{2 m}^ 2 \vp_g + f \cdot g  \vp_A  d y, \ m \geq 1, \quad  \la f , g \ra_{\cX_A^0} = \int f \cdot g \vp_A , \\
\eal
\eeq
where $\vp_A$ is constructed in Theorem \ref{thm:coerA0}, 
$\kp_1$ and weight $ \vp_g$ are chosen in \eqref{norm:Xk}, and $ \vp_{2m}$ in Lemma \ref{cor:wg}.
\end{theorem}

It is clear from the proofs of Theorems \ref{thm:coer_est} and \ref{thm:coerA} that the estimates 
in both theorems hold true for smaller $\e_m$. By choosing $\lam$ smaller and $m_0$ larger in both theorems, we can assume that the parameters $\e_m, \lam, m_0$  are the same in both Theorems. 

\begin{proof}[Proof of Theorem \ref{thm:coerA}]



The weighted $H^{2m}$ estimates with weights $\vp_{2m}^2 \vp_g$ are similar to those for $\cL$ in Section \ref{sec:thm_coer_pf}. In analogy to the derivation in \eqref{eq:lin_Hk}, we apply $\Delta^m$ to $\cL_A$ defined in \eqref{eq:lin_A} and the resulting expression into a transport, damping, stretching, and remainder term; to avoid redundancy we do not spell out these details. Analogously to the bounds \eqref{eq:lin_est1}, 
\eqref{eq:lin_est2D}, \eqref{eq:lin_est2}, \eqref{eq:lin_est3}, and recalling~\eqref{eq:kappa3}, we obtain 
\begin{equation*}
\int \D^m \cL_A (\AA) \cdot \D^m \AA \vp_{2m}^2 \vp_g d y
  \leq \int 
  D_{A, 2m}
   \vp_{2m}^2 \vp_g
+ C_{m, \e} \int |\AA|^2 \vp_g ,
\end{equation*}
where $D_{A, 2m}$ is defined as 
\[
\bal
D_{A, 2m}  &= ( - (r-1) + \bar C_A \la \xi \ra^{-\kappa_3} + \e ) \G_{A, m}
+ 2 m H_{A, 1},  \\
H_{A, 1} &=  \B( (\xi + \bar U) \f{\pa_{\xi} \vp_1}{\vp_1} -   \f{\bar F}{\xi}  \B) \G_{A, m}- D_{\xi} ( \f{\bar F}{\xi} ) \G_{A,m,(\xi)} , \\
\G_{A, m} &= |\na^2 \D^{m-1} \AA|^2 ,\quad \G_{A, m, (\xi)} = | \pa_{\xi} \na \D^{m-1} \AA|^2 .
\eal
\]


Compared to \eqref{eq:lin_est1}, \eqref{eq:lin_est2}, in the above estimates, we do not need to estimate cross terms \eqref{eq:lin_cross1}, \eqref{eq:lin_cross2}, which contribute to the bounds $\al \bar \S | \f{\pa_{\xi} \vp_1}{\vp_1} |$ in \eqref{eq:lin_est1} and $2 m \al |\pa_{\xi} \bar \S|$ in \eqref{eq:lin_est2D}. In fact, $\cL_A$ does not involve $\bar \S$ (see \eqref{eq:lin_A}). Using Lemma \ref{cor:wg} \eqref{eq:repul_wg_ag} with $ l  = 0$ and following the derivations in \eqref{eq:coer_est_H1}, we obtain
\[ 
H_{A,1} 
\leq 
 \B( (\xi + \bar U) \f{\pa_{\xi} \vp_1}{\vp_1} - \f{\bar F}{\xi}  \B) (\G_{A, m} -\G_{A, m, (\xi)}) +  \B( (\xi + \bar U) \f{\pa_{\xi} \vp_1}{\vp_1} - \f{\bar F}{\xi} - D_{\xi} \f{\bar F}{\xi} \B) \G_{A, m,(\xi)}
\leq - \mu_1 \la y\ra^{-1} \G_{A, m}.
\]

Choosing $\e = \f{r-1}{2}$ and $m_0$ large enough and using $\G_{A, m} \geq \f{1}{4} |\D^m A|^2$, 
for any $ m \geq m_0$, we establish 
\[
 D_{A, 2m} \leq \B( - \f{r-1}{2} + ( \bar C_A -2 m\mu_1) \la \xi \ra^{-1} \B) \G_{A, m}
 \leq  - \f{r-1}{2} \G_{A, m} \leq -\f{r-1}{8} |\D^m \AA|^2. 
\]
The remaining steps follow the argument at the end of the proof of Theorem \ref{thm:coer_est}
and by combining the above estimates with Theorem \ref{thm:coerA0}. We omit them and conclude the proof.
\end{proof}


Theorem \ref{thm:coerA} imply the following stability estimate for this special class of perturbation. 
\begin{prop}\label{prop:semiA}
Let $\cX^m$ and $\cX^m_A$ be the norms defined in \eqref{norm:Xk}, \eqref{norm:Xk_A}, respectively. For $ m \geq  m_0  $, we have 
\[
|| e^{\cL s} (\AA , 0) ||_{\cX^m} \les_m || e^{\cL s} (\AA , 0) ||_{\cX^m_A}
\les_m e^{-\lam s} || \AA ||_{\cX^m}. 
\]
\end{prop}

Before we prove Proposition \ref{prop:semiA}, we have the following equivalence between the norms $\cX_A^m$ and $\cX^m$.

\begin{lemma}\label{norm:equiv}
For $m \geq 3$ and any $f \in \cX^m$ with $f(0)=0$, we have
$
 || f ||_{\cX^m} \les_m  || f ||_{\cX_A^m}  \les_m || f ||_{\cX^m}$.
\end{lemma}

\begin{proof}

From the definitions of $\cX_{A}^m, \cX^m$ \eqref{norm:Xk}, \eqref{norm:Xk_A}, we know that 
$\cX_{A}^m$ is stronger than $\cX^m$. For  $ f \in \cX^m, m \geq 3$, using the embedding in Lemma \ref{lem:prod}, we obtain $f \in C^1(B(0, 1))$. 
Since $f(0) = 0$, we get $|f(y)| \les |y|, |y| \leq 1$, which along with $f \in \cX^m$
and $\vp_A(y) \asymp \vp_g(y)$ for $|y| \geq 1$ (see Theorem \ref{thm:coerA0}) yields
\[
 || f \vp_A^{1/2} ||_{L^2} \les_m || f ||_{\cX^m} + || f \vp_g^{1/2} ||_{L^2} 
 \les_m || f ||_{\cX^m}. 
\]
Since $\cX_A^m$ and $\cX^m$ \eqref{norm:Xk}, \eqref{norm:Xk_A} only differ by the weighted $L^2$ part, we prove the equivalence.
\end{proof}

Now, we are ready to prove Proposition \ref{prop:semiA}. 

\begin{proof}[Proof of Proposition \ref{prop:semiA}]
For initial data $(\UU, \S) = (\AA, 0), \AA = A(\xi ) \ee_{\th} \in \cX^m$, from Proposition \ref{prop:semi} and the above discussions of the symmetry, we obtain the solution $\AA(s) = e^{\cL s} (\AA, 0) \in \cX^m$. Since $\AA \in C^1$ by embedding in Lemma \ref{lem:prod} and is axisymmetric, we get $\AA(0) = 0$. Using the dissipative estimates in Theorem \ref{thm:coerA} and the equivalence in Lemma \ref{lem:norm_equiv}, we establish 
\[
 || e^{\cL s }(\AA, 0) ||_{\cX^m} 
 \les_m  || e^{\cL s }(\AA, 0) ||_{\cX_A^m} 
 \les_m e^{-\lam s} || (\AA, 0) ||_{\cX_A^m}  \les_m e^{-\lam s} || (\AA, 0) ||_{\cX^m}  . 
\]
We conclude the proof.
\end{proof}

\subsection{Perturbation supported in the far-field}\label{sec:pertb_far}

We have the following decay estimates for perturbation supported in the far-field, which was first developed in \cite{chen2024Euler} for perturbation with radial symmetry. 


\begin{proposition}\label{prop:far_decay}
Let $R_4$ be defined in Proposition \ref{prop:compact}. Consider the linearized equations \eqref{eq:linlin} from initial data $\VV_0 = (\UU_0, \S_0)$ with $\supp(\UU_0), \supp(\S_0) \subset B( 0, R)^c$ for some $R > 4 R_4> \xi_s $. For any $m \geq m_0$, the solution $\VV(s)$ satisfies 
\[
\supp(\VV(s)) \subset B( 0, 4 R_4)^c, \quad  || e^{\cL s} \VV_0||_{\cX^m} \leq e^{-\lam s} || \VV_0 ||_{\cX^m}.
\]
\end{proposition}

The proof is the same as that in \cite{chen2024Euler} for radially symmetric perturbation. 
For completeness, we attach the proof in Appendix \ref{sec:add_pf}.

\subsection{Decay estimates of $e^{\cL s}$ }\label{sec:semi_grow}

Based on Proposition \ref{prop:semi}, using operator theories from Corollary 2.11, Chapter IV \cite{engel2000one} for the growth bound, and Theorem 2.1, Chapter XV, Part IV (Page 326) \cite{GoGoKa2013} for the spectral projection, we can perform the following decomposition of $\cX^m$ and the spectral $\s(\cL)$. Note that the functional framework does not require any symmetry assumption of $(\UU, \S)$. The arguments presented below is the same as those in \cite[Section 3.5]{chen2024Euler}. We refer more discussions to \cite{chen2024Euler}. Below, we summarize the results.

Recall the parameter $\lam$ from Theorem \ref{thm:coer_est}, Proposition \ref{prop:compact},  Theorem \ref{thm:coerA}, and \eqref{eq:spec}. Due to \eqref{eq:spec},  the set 
\beq\label{eq:sig_eta}
\s_{\eta} \teq \s( \cL) \cap  \{ z : \Re(z) > - \eta \}, \mathrm{\quad where \ } 
 \eta = \f{4}{5} \lam < \lam ,
\eeq
only consists of finite many eigenvalues of $\cL$ with finite multiplicity. We can decompose $X = \cX^m$ into the stable part $\cX_s^m$ and unstable part $\cX_u^m$
\beq\label{eq:dec_X}
  \cX^m = \cX^m_{u} \oplus \cX^m_s, \quad \s( \cL|_{\cX^m_s}) = \s(\cL) \backslash \s_{\eta}
\subset \{ z: \Re(z) \leq -\eta \}, 
  \quad \s( \cL|_{X_u}) = \s_{\eta}.
\eeq
The space $\cX^m_u$ has finite dimension and can be decomposed as follows 
\beq\label{eq:dec_Xu}
\cX^m_u = \bigoplus_{ z \in \s_{\eta}} \ker( (z- \cL)^{\mu_z}), \quad \mu_z < \infty, \quad |\s_{\eta}| < +\infty.
\eeq

We have the following estimates of the semigroup in these two spaces. For $f$ in the stable part, we have
\bseq\label{eq:decay_stab} 
\beq
 || e^{\cL s } f ||_{\cX^m}   \leq C_m e^{ - \eta_s s } || f ||_{\cX^m}, \quad \forall f \in \cX^m_s, 
 \eeq
 where $\eta_s$ is defined as 
 \beq
 \quad  \eta_s = \f{3}{5} \lam <  \f{4}{5} \lam =  \eta .
 \eeq
\eseq
For $f$ in the unstable parts, we have
\beq\label{eq:decay_unstab}
  || e^{-\cL s} f ||_{\cX^m}  \leq  C_m e^{\eta s} || f ||_{\cX^m}.
   \quad \forall f \in \cX^m_u , \quad s > 0 .
\eeq

\subsubsection{Smoothness of unstable directions}\label{sec:smooth_unstab}

The property that the unstable part $\cX_u^m$ \eqref{eq:dec_Xu} is spanned by smooth functions follow from the following Lemma proved in \cite[Lemma 3.9]{chen2024Euler}.

\begin{lemma}\label{lem:smooth}
Let $\{ X^i\}_{i \geq 0}$ be a sequence of Banach spaces with $X^{i+1} \subset X^{i}$ for all $i \geq 0$.
Assume that for any $i \geq i_0$ we can decompose the linear operator 
$\cA \colon D(\cA) \subset X^i \to X^i$ as 
 $\cA = \cD_i + \cK_i$, where the linear operators $\cD_i$ and $\cK_i$ satisfy
\beq\label{eq:smooth_ass}
\cD_i : D(\cA) \subset X^i \to X^i , \quad  \cK_i :  X^{i-1}  \to  X^i,
\quad \{ z \in {\mathbb C}\colon \Re(z) > - \lam \} \subset \rho_{\mw{res}}(\cD_i).
\eeq
Here, $\rho_{\mw{res}}(\cdot)$ denotes the resolvent set of an operator and $\lambda>0$ is independent of $i\geq i_0$. Fix $n\geq 0$ and $z \in {\mathbb C}$ with $\Re(z) > -\lam$. Assume that the functions $ f_0, \ldots , f_n  \in  X^{i_0}$ satisfy 
\[
(z - \cA) f_0 = 0,
\quad (z - \cA) f_{i} = f_{i-1}, \quad \mbox{ for } \quad 1 \leq i \leq n.
 \]
Then, we have $ f_0, \ldots , f_n  \in X^{\infty} \teq \cap_{i \geq 0} X^i$. 
\end{lemma}

To describe the regularity of the elements of $\ker( (z- \cL)^{\mu_z})$ \eqref{eq:dec_Xu}, we apply Lemma \ref{lem:smooth} with $(\cA,  \{ X^i \}_{i\geq 0} )$ = 
$(\cL, \cX^i)$ and the decomposition $\cL = \cD_m + \cK_m$ for any $m \geq m_0 \geq 6$, where $\cK_m$ is constructed in Proposition \ref{prop:compact}. 
Using Lemma \ref{lem:Xm_chain}, Proposition \ref{prop:compact}, and \eqref{eq:spec}, 
we verify the assumption on $X^i$ and \eqref{eq:smooth_ass}  in Lemma \ref{lem:smooth}. We fix $z \in \s(\cL) \cap \{ z : \Re(z) > -\eta \}$, where we recall $-\eta > -\lam$, and fix $g \in \ker( (z- \cL)^{\mu_z} ) \subset \cX_u^{m}$ for some $\mu_z < \infty$ constructed in \eqref{eq:dec_X}, \eqref{eq:dec_Xu}. We apply Lemma \ref{lem:smooth} with  $f_i = (z- \cL)^{ \mu_z- i -1} g, 0 \leq i \leq \mu_z-1$ to obtain $f_i \in \cX^{\infty}= \cap_{m \geq m_0} \cX^m$. Since $\cX^{\infty}  = \cap_{m \geq m_0} \cX^m \subset C^{\infty}$ from Lemma \ref{lem:prod}, we conclude from \eqref{eq:dec_Xu} that $\cX_u^{m}$ are spanned by smooth functions.

\section{Nonlinear stability and vorticity blowup}\label{sec:non}

The goal of this section is to prove Theorem~\ref{thm:blowup}, by constructing global solutions $(\UU,\S)$ to \eqref{eq:euler_ss} and estimating the vorticity along the trajectory. To construct the global solutions, we follow the finite rank perturbation of the linearized equations developed in \cite{ChenHou2023a} and the framework based on a fixed point argument in \cite{chen2024Euler}. Throughout this section we fix $m\geq m_0$, where $m_0$ is sufficiently large, as in Theorem~\ref{thm:coer_est}.

\subsection{Decomposition of the solution}\label{sec:decom_init}

 We will use $\VV = (\UU,  \S )$ to denote the nonlinear solution to~\eqref{eq:euler_ss}. As in~\eqref{eq:lin}, \eqref{eq:non}, we denote the perturbation to the stationary profile as 
\[
\td \VV \teq (\td \UU,\td \S) \teq (\UU - \bar \UU, \S - \bar \S ),
\]
and recall that this perturbation solves \eqref{eq:lin}--\eqref{eq:non}. We shall further decompose the perturbation\footnote{We emphasize that the $\td \cdot_1$ or $\td \cdot_2$ sub-index denote different parts of the perturbation, they  {\em do not represent} Cartesian coordinates.} as $\td \VV = \td \VV_1 + \td \VV_2$, with $\td \UU = \td \UU_1 + \td \UU_2$, $\td \S = \td \S_1 + \td \S_2$, so that
\beq\label{eq:init}
\bal
\VV= (\UU,  \S ) 
= \td \VV_1 + \td \VV_2 + \bar \VV, \quad 
\td \VV_i = (\td \UU_i, \td \S_i ), \quad 
\bar \VV =  (\bar \UU, \bar \S ),
 \eal
 \eeq
 and $\td \VV_i$ solves the following equations 
\begin{subequations}
\label{eq:non_V}
\begin{align}
 & \pa_s \td \VV_1   = \cL (\td \VV_1) 
 - \cK_m(\td \VV_1) + \cN(\td \VV) ,
 \label{eq:non_Va} \\
 & \pa_s  \td \VV_2  = \cL ( \td \VV_2) 
 + \cK_m( \td \VV_1), 
 \label{eq:non_Vb}
\end{align}
where $\cK_m$ is constructed in Proposition \ref{prop:compact}, and the linear operators 
$\cL = (\cL_U, \cL_{\S})$ are defined in \eqref{eq:lin}, and nonlinear operators 
$\cN = (\cN_U, \cN_{\S})$ in \eqref{eq:non}. The part $\td \VV_2$ is used to capture the unstable parts. 
\end{subequations}

Let $\Pi_s, \Pi_u$ be the projections of $\cX^m$ onto $\cX^m_s, \cX^m_u$, respectively.   
To handle the unstable part and obtain initial data with nontrivial vorticity and compact support, we construct $\td \VV_2$ as follows. 
\bseq\label{eq:V2_form2}
\begin{align}
 \td \VV_2(s)  & \teq e^{\cL (s-\sin)}(  \AA, 0) + 
\td \VV_{2, s}(s) - \td  \VV_{2, u}(s) + e^{\cL (s-\sin)}( \td \VV_{2, u}(\sin ) (1 - \chi(y / (8 R_4) )) ), \label{eq:V2_form2a}  \\
  \td \VV_{2, s}(s) & = \int_{\sin}^s e^{\cL(s -\tau)} \Pi_s \cK_m( \td \UU_1,\td  \S_1)( \tau) d \tau,   \\
   \td \VV_{2, u}(s) & = \int_s^{\infty} e^{\cL(s - \tau)} \Pi_s \cK_m( \td \UU_1, \td \S_1)(\tau) d \tau ,\label{eq:V2_form2c} 
\end{align}
\eseq
where $\AA = A(\xi) e_{\th} $ is the initial swirl velocity to be chosen,  
$(\AA, 0)$ denotes $(\UU, \S)$ with $\UU = \AA$ and $\S = 0$, and $\chi$ is some smooth cutoff function with $\chi(y) = 1, |y| \leq 1, \chi(y) = 0, |y| \geq 2$. 

For our later fixed point argument, we codify the term in $\td \VV_2$ \eqref{eq:V2_form2a} depending on $(\td \UU_1, \td \S_1)$ as a map $\cT_2$
\beq\label{eq:T2}
 \cT_2( \td U_1, \td \S_1) \teq 
 \td \VV_{2, s}(s) - \td  \VV_{2, u}(s) + e^{\cL (s-\sin) }( \td \VV_{2, u}( \sin) (1 - \chi(y / (8 R_4) )) ).
 \eeq

There are a few parts in $\td \VV_2$ \eqref{eq:V2_form2}. The main idea is that  we apply the 
semigroup $e^{\cL s}$ forward-in-time on the stable piece of $\td \VV_{2}$, and backward-in-time on the unstable part, which is a usual technique for constructing unstable manifolds. We use $e^{\cL (s -\sin)}(  \AA, 0)$ to construct nontrivial initial vorticity, and $e^{\cL (s-\sin) }( \td \VV_{2, u}(\sin) (1 - \chi(y / (8 R_4) )) )$ to obtain initial data with compact support. Indeed, by definition of $\td \VV_2$ and $\chi$, we have
\beq\label{eq:V2_init}
    \td \VV_2( \sin )=  (\AA, 0) - \td \VV_{2, u}( \sin ) \chi(y / (8 R_4)) ,
    \quad  \supp( \td \VV_{2, u}( \sin ) \chi(y / (8 R_4))) \subset( B(0, 16 R_4)).
\eeq
We will choose  $\AA$ with compact support to obtain that $\td \VV_2( \sin) $ has a compact support.

At the linear level of the decomposition \eqref{eq:non_V}, $\td \VV_2$ is decoupled from  \eqref{eq:non_Va} for $\td \VV_1$. Such a decomposition and representation \eqref{eq:V2_form2} allows us to first perform energy estimates to obtain stability estimates of $\td \VV_1$, and then obtain stability estimates of $\td \VV_2$ using the decay estimates of $e^{\cL s}$ in Section \ref{sec:decay}. We construct global solution to the nonlinear equations \eqref{eq:non_V} by treating the nonlinear terms perturbatively and using a fixed point argument. This method has been first introduced in \cite{ChenHou2023a} and developed in \cite{chen2024Euler}.

\subsection{Functional setting}

We introduce some space $\cW^{m+1}$ for closing the nonlinear estimates. Our goal is to perform both weighted $H^{2m}, H^{2m+2}$ estimates on \eqref{eq:non_V} using the 
\textit{same compact operator} $\cK_m$ and the
\textit{same projections} $\Pi_s, \Pi_u$ \eqref{eq:V2_form2} appearing in \eqref{eq:V2_form2}. 
In particular, we \textit{do not} change $\cK_m$ to $\cK_{m+1}$ for weighted $H^{2m+2}$ estimates. 

Fix an arbitrary $ m \geq m_0$. For some $\mu_{m+1} > 0$ to be chosen,  using Theorem \ref{thm:coer_est}, items (c) and (d) in Proposition \ref{prop:compact}, and $||f ||_{\cX^0} \les || f ||_{\cX^m}$ from Lemma \ref{lem:Xm_chain}, we obtain 
\[
\bal
   \la  (\cL & - \cK_m) f , f \ra_{\cX^m}   + \mu_{m+1} \la  (\cL - \cK_m) f ,  f \ra_{\cX^{m+1}}  \\
   & \leq - \lam || f ||_{\cX^m}^2 + \mu_{m+1} ( \la \cL f , f \ra_{\cX^{m+1}}   
+ || \cK_m f ||_{\cX^{m+1}} || f ||_{\cX^{m+1}} ) \\
& \leq -\lam || f ||_{\cX^m}^2 +  \mu_{m+1} ( - \lam || f ||_{\cX^{m+1} }^2
+ \bar C || f ||_{\cX^0}^2 
+  C_m  || f ||_{\cX^{0} }|| f ||_{\cX^{m+1}} ) \\
& \leq (-\lam + \mu_{m+1} \td C_m  ) || f ||_{\cX^m}^2 - \f{9}{10} \mu_{m+1} \lam || f ||_{\cX^{m+1} }^2
\eal
\]
for some $\td C_m > 0$, where we have used $C_m ||f ||_{\cX^0} || f||_{\cX^{m+1}}
\leq \f{1}{10} \lam  || f||_{\cX^{m+1}}^2 +  C_m^{\pr} || f||_{\cX^m}^2$ in the last inequality. Choosing $ \mu_{m+1} $ small enough in terms of $m$, we can obtain the coercive estimates 
\[
 \la  (\cL  - \cK_m) f , f \ra_{\cX^m}   + \mu_{m+1} \la  (\cL - \cK_m) f ,  f \ra_{\cX^{m+1}}
 \leq - \f{9}{10} \lam (  || f ||_{\cX^m}^2 + \mu_{m+1} || f ||_{\cX^{m+1}}^2 ).
\]
 In light of the above estimates, we can define the Hilbert space $\cW^{m+1}$ as 
\beq\label{norm:Wk}
\la f, g \ra_{\cW^{m+1}} \teq 
 \la f , g \ra_{\cX^m}
+  \mu_{m+1}  \la f, g \ra_{\cX^{m+1}} 
 , \quad || f ||_{\cW^{m+1}}^2 = \la f , f \ra_{\cW^{m+1}},
\eeq
and can rewrite the coercive estimate in this space
\beq\label{eq:coer_Wk}
  \la (\cL - \cK_m ) f , f \ra_{\cW^{m+1}} \leq - \lam_1 || f||^2_{ \cW^{ m+1 }}, \quad 
  \lam_1 = \f{9}{10 } \lam.
\eeq

It is clear that the norms $\cW^{m+1}$ and $\cX^{m+1}$ \eqref{norm:Xk} are equivalent.

\subsection{Nonlinear stability and proof of Theorem \ref{thm:blowup}}\label{sec:non_stab}

We recall from \eqref{eq:sig_eta}, \eqref{eq:decay_stab}, \eqref{eq:coer_Wk},  \eqref{eq:decay_stab}  the decay rates $\eta_s, \eta, \lam_1$
\beq\label{eq:decay_para}
(\eta_s, \eta, \lam_1) = ( \f{3}{5}, \f{4}{5}, \f{9}{10}) \lam,\quad  \eta_s < \eta < \lam_1 < \lam.
\eeq

Our goal is to prove the following theorem for nonlinear stability.

\begin{theorem}[Nonlinear stability]\label{thm:non}
Fix $m \geq m_0$. Let $R_4$ be chosen in Theorem \ref{thm:coer_est} and 
$\d_Y = \d^{2/3}, \d_A = \d^{5/9}$. There exists a sufficiently small $\d_0$, a large $R_{c, 0} > (16 R_{5})^2 + 1$ and a constant $\bar C_m > 0$ such that for any $\d < \d_0$ and $R_c > \max( R_{c, 0}, \d^{-4})$, the following holds true. For any initial data $\td \VV_{1,in} =  (\td \UU_1(\sin), \td \S_1(\sin) )$ and $\AA= A(\xi) \ee_{\th}$ smooth enough 
\footnote{We require the $\cX^{m+2}$-regularity of $\td \VV_{1,in}, \AA$, a space stronger than $\cW^{m+1}$ and $\cX^{m+1}$, in order to obtain the local-in-time existence of a $\cX^{m+2}$-solution; in turn, this allows us to justify a few estimates, e.g.~\eqref{eq:coer_Wk} for $(\td \UU_1, \td \S_1)$ which requires $\cL(\td \UU_1, \td \S_1) \in \cX^{m+1}$. Note that this regularity requirement is only \textit{qualitative}, and we only use Theorem~\ref{thm:non} with an $C^{\infty}$ initial perturbation (see~\eqref{eq:thm_init3}) in order to prove Theorem~\ref{thm:blowup}. The only {\em quantitative} assumption on the initial data is given by~\eqref{eq:thm_ass}. Here, our reasoning and argument to justify the estimates, e.g., \eqref{eq:coer_Wk}, are the same as those in \cite[Sections 4.3, 4.4]{chen2024Euler}, and we refer to it for more details.} 
to ensure $\td \VV_{1,in}, \AA \in  \cX^{m+2}$ and small enough to ensure
\beq\label{eq:thm_ass}
\bga
 ||(\td \UU_1(\sin), \td \S_1(\sin))||_{\cW^{m+1}}  < \d, \quad ||\AA||_{\cX^{m+2}} < \d_A, \\
 \supp( \td \UU_1(\sin) + \AA + \bar \UU )  \cup  \supp(\td \S_1(\sin) + \bar \S- \S_{\infty}) \subset B(0, R_c^{1/2}) ,
\ega
\eeq
for some constant $\S_{\infty} > 0$, there exists a global solution 
$\td \VV_1$ to \eqref{eq:non_V} from initial data $\td \VV_{1, in}$ and a global solution $\td \VV_2$ to \eqref{eq:V2_form2} given by \eqref{eq:V2_form2}, satisfying the bounds
\bseq\label{eq:non_decay}
\begin{align}
   ||( \td \UU_1(s), \td \S_1(s))||_{\cW^{m+1}}  & < 2 \d e^{- \lam_1 (s-\sin) } , \\
 || ( \td \UU_2(s), \td \S_2(s))  ||_{\cX^{m+2}} & \les_m \d_A e^{-\eta_s (s-\sin)}, \\
   || \td \VV_{2, u}(s) ||_{\cX^{m+2}} & \les_m \d e^{- \eta_s (s-\sin) }, \label{eq:non_decayc} \\ 
  e^{- s} \supp( \UU(s) ) \cup \supp( \S(s) - e^{ - (r-1) (s-\sin)} \S_{\infty} ) &  <
\bar C_m R_c^{1/2} <  R_c / 2 , \label{eq:non_decay_supp}
\end{align}
\eseq
for all $s \geq s_{in}$. We emphasize that we cannot prescribe the initial data $\td \VV_{2, in}
= (\td \UU_2(\sin), \td \S_2(\sin) )$; rather it is constructed by \eqref{eq:V2_form2} (together with $\td \VV_1$) to lie in a finite co-dimension subspace of $\cW^{m+2}$. 
\end{theorem}

The support condition in \eqref{eq:thm_ass} is equivalent to that for $(\UU, \S - \S_{\infty})$
due to \eqref{eq:V2_init} and $R_c^{1/2} > 16 R_4$.
Due to \eqref{eq:non_decay_supp}, \eqref{eq:axi_recenter}, \eqref{eq:var_ss}, 
in the physical space, 
$\supp(u^r, u^z)$ remains uniformly away from the axis $r= 0$.

\begin{remark}\label{rem:data}
The initial data $(\uu_0,\s_0)$ in Theorem \ref{thm:blowup} is obtained from Theorem~\ref{thm:non} and the decomposition~\eqref{eq:init} at time $s=\sin$. In light of Theorem~\ref{thm:non}, \eqref{eq:init}, and 
\eqref{eq:V2_init},  we identify the space $Z_2$ mentioned in Remark~\ref{rem:intro_init_data} with an open ball in the weighted Sobolev space $\cX^{m+2}$ defined in~\eqref{norm:Xk} and with radial symmetry $\FF = F(\xi) \ee_{\th}$. On the other hand, the space $Z_1$ mentioned in Remark~\ref{rem:intro_init_data} consists of functions as the sum of an element $(\td \UU_1(\sin), \td \S_1(\sin))$ which lies in open ball in the weighted Sobolev space $\cW^{m+1}$ (see definition~\eqref{norm:Wk}) and the element $-\td \VV_{2, u}(\sin) \chi(y / 8 R_4)$ constructed in \eqref{eq:V2_init} and \eqref{eq:V2_form2a}, which lies in a finite-dimensional subspace of $\cX^{m+2}$. Note that we do not impose symmetry assumption on $Z_2$.
\end{remark}

In the remaining sections, we will first estimate the vorticity along the trajectory in Section \ref{sec:traj}, and then use it and Theorem \ref{thm:non} to prove the main Theorem \ref{thm:blowup} in Section \ref{sec:thm1_pf}. In Section \ref{sec:non_idea}, we outline the proof of Theorem \ref{thm:non} and the organization of Sections \ref{sec:supp}-\ref{sec:contra}, which are devoted to prove Theorem \ref{thm:non}.

\subsubsection{Estimate of the trajectory}\label{sec:traj}

We estimate the vorticity along the trajectory. 
Using \eqref{eq:euler_2D}, we obtain the equations for the vorticity $
\om =\pa_2 u_1 - \pa_1 u_2$ \eqref{eq:vor_id} and density $\rho = (\al \s)^{1/\al}$ 
\[
\bal
 & \pa_t \om + \uu \cdot \na \om & = -\div(\uu) \om, \quad \pa_t \rho + \uu \cdot \na \rho & = - \div(\uu) \rho - (d-2) \f{u_1 \rho}{ R_c + \xc_1} . \\
  \eal
\]

Using these two equations, we derive the evolution of the specific vorticity $\f{\om}{\rho}$
\beq\label{eq:vor}
 \pa_t \f{\om}{\rho} + \uu \cdot \na (\f{\om}{\rho}) =  (d-2) \f{u_1}{R_c + \xc_1} \f{\om}{\rho}.
\eeq
We will show that the right hand side is of lower order. Then $ \f{\om}{\rho}$ is almost conserved along the trajectory. 

We have the following estimate of the trajectory in the self-similar variable.

\begin{lemma}[Escape to spatial infinity]\label{lem:traj}
Let $\kp$ be the parameter from \eqref{eq:rep2} and $\sin$ be the initial time defined in \eqref{eq:s_in}. Consider the characteristics associated with the full velocity in the self-similar equation
\[
  \f{d}{d s} X(x, s)  = X(x, s) + \UU(X(x, s), s), 
  \quad X(s_{in}) = x. 
\]
If $ | \UU(s) - \bar \UU | \leq \f{1}{10} \kp \la y \ra $ for all time $ s \geq \sin$, then for any $|x| \geq 1$ and $s \geq 0$, we have 
\[
 |X(x, s) | \geq e^{ \kp  (s-\sin) / 2 } |x| \geq e^{\kp (s - s_{in})/ 2 }.
\]
\end{lemma}

We remark that for a flow $\UU(s, y)$ close to $\bar \UU$, since $y = 0$ may not be a stationary point of the flow $\UU(y, s)$, $X(x, s)$ may not escape to $\infty$ if 
the initial location $x$ is sufficiently small.

\begin{proof}
We derive the evolution of $|X|^2$
\[
 \f{1}{2}\f{d}{d s}  |X|^2 = X \cdot (X + \UU(X, s))
 = |X|^2 + X \cdot \bar \UU + X \cdot ( \UU - \bar \UU)
\]
Since $\bar \UU(x) = \bar U(|x|) \ee_R$, using the assumption on $\UU - \bar \UU$, \eqref{eq:rep22}, 
and $  \la X \ra \leq 3 |X|  $ for $|X| \geq \f{1}{2}$,  we get 
\[
 \f{1}{2}\f{d}{d s}  |X|^2 \geq |X|^2(1 + \f{ \bar U(|X|)}{|X|} - \f{\kp}{10 |X|} \la X \ra ) 
 \geq  |X|^2( \kp- \f{3\kp}{10})  \geq \f{1}{2}\kp |X|^2 > 0,
\]
if $|X| \geq \f{1}{2}$. Since the initial location satisfies $|x| \geq 1$, solving the above ODE inequality, we prove the desired estimate. 
\end{proof}

\subsubsection{Proof of Theorem \ref{thm:blowup}}\label{sec:thm1_pf}

Now, we are ready to prove Theorem \ref{thm:blowup}. We will first construct initial data with nontrivial vorticity near $0$. Then we estimate the trajectory and prove that the vorticity blows up. 

\begin{proof}%

\vspace{0.1in}
\textbf{Step 1: Initial data.}
Firstly, we choose initial perturbation $\td \VV_{1, in} = (\td \UU_1(\sin), \td \S_1(\sin) ), \AA$ as in Theorem \ref{thm:non}. Moreover, this initial perturbation may be chosen to ensure that 
\bseq\label{eq:thm_ass2}
\begin{gather}
   || \AA||_{\cX^{m+ 2}} \in [ \f{1}{2}\d_A ,\d_A ], 
\quad  \AA( y) = A( |y|) \ee_{\th} \in C_c^{\infty}  , 
\label{eq:thm_init2a} \\
   \B| \pa_{\xi} A( \xi ) + \f{A(\xi )}{ \xi } \B|  \gtr_m \d_A , \quad \mathrm{for \ } | \xi | \leq 1 , \label{eq:thm_init2} \\
\UU_{in} \in C_c^{\infty}  , \   \S_{in} \in C^{\infty}, 
\quad 
   \S_{in}= \bar \S + \td \S_1 + \td \S_2 > c>0, \quad 
   \S_{in} = 1, \ |y| \geq 2 C_{in}, 
  \label{eq:thm_init1}
\end{gather}
\eseq
for some $c > 0$ and large enough $C_{in} \geq 1$, where $\xi = |y|$. For example, we can choose 
\bseq\label{eq:thm_init3}
\begin{align}
 \td \UU_1(y, \sin) &=  ( \chi(y / C_{in} ) -1)  \bar \UU(y),  \\
  \td \S_1(y, \sin) &=  (1 - \chi(y/C_{in} )) (1 - \bar \S(y))  \label{eq:thm_init3b}, \\
 A(\xi) & = c_m \d_A \cdot \xi \chi(\xi) ,  \label{eq:thm_init3c}
 \end{align}
 where $ \chi : \R^d \to [0, 1]$ is a radially symmetric and smooth cutoff function with $\chi(y) = 1$ for $|y| \leq 1$ and  $\chi(y) = 0$ for $|y| \geq 2$, and $c_m > 0$ is some parameter to be chosen.
\eseq

We show that for $C_{in}$ large enough, the data given in \eqref{eq:thm_init3} satisfies assumptions \eqref{eq:thm_ass} and \eqref{eq:thm_ass2}. Let $F(y) = \td \UU_1(y, \sin)$ or $\td \S_1(y, \sin)$.  From Lemma \ref{lem:profile} and Lemma \ref{cor:wg} \eqref{eq:repul_wg_aga}, \eqref{eq:wg_vp} for $\vp_{2i}$, we have
 $ \vp_{2i} |\na^{2i} F| \les_i \one_{|y| \geq C_{in}} \la y \ra^{ 1  - r  } \in L^2(\vp_g)$,
where $\vp_g$ is defined in \eqref{norm:Xk}. 
By the dominated convergence theorem, we obtain $|| F||_{\cX^i} \to 0$ for $i \leq m+3$ as $C_{in} \to \infty$. In particular, we can ensure $||  (\td \UU_1, \td \S_1) ||_{\cW^{m+1}} <\f{1}{2} \d$ by choosing $C_{in} > \bar C_{in} > 16 R_4 + 1$ with $\bar C_{in}$ sufficiently large that only depends on $(\bar \UU, \bar \S)$. From the formulas of  $\td \VV_2(\sin)$ \eqref{eq:V2_init} 
and of $\AA$ \eqref{eq:thm_init3c}, $\td \VV_{2}(\sin)$ has compact support in $ B(0, 16 R_4)$. 

For $A(\xi)$ given in \eqref{eq:thm_init3}, we get $\pa_{\xi} A + \f{A}{\xi} = 2 c_m \d_A$ for $|\xi| \leq 1$ and $\xi \chi(\xi) \in \cX^{m+3}$. Thus, we can choose $c_m $ depending only on $m$ to verify the size conditions of $\AA$ in \eqref{eq:thm_init2} and \eqref{eq:thm_init2a}. 

Since $\td \VV_2(\sin)$ \eqref{eq:V2_form2} and $\AA(\sin)$ \eqref{eq:thm_init3c} have compact support, $\td \UU_1(\sin) \in C_c^{\infty},\td \S_{1}(\sin)\in C^{\infty}$ \eqref{eq:thm_init3}, and $\cX_u^m$ is spanned by smooth functions (see Section \ref{sec:smooth_unstab}), we get $\UU_{in}, \AA(\sin) \in C_c^{\infty}, \S_{in} \in C^{\infty}$. Thus, we verify all the conditions of $\UU, \AA$ in \eqref{eq:thm_ass2} and $\S_{in} \in C^{\infty}$ in \eqref{eq:thm_init1}.

Next, we verify the estimates in \eqref{eq:thm_init1}. Recall $\S_{in} =\bar \S + \td \S_1(\sin) + \td \S_2( \sin )$ from \eqref{eq:init}. From \eqref{eq:non_decayc} and \eqref{eq:Xm_Linf} in Lemma \ref{lem:prod}, we get $| \td \S_{2,in}(y) | \les C(R_4) \d, |y| \leq 16 R_4 $. 
Choosing $\d$ small enough, we obtain
\[
\S_{in}= \bar \S + \td \S_{2}(\sin) > \f{1}{2} \bar \S > 0, \ \mathrm{for} \ |y| \leq 16 R_4,
\quad 
 \S_{in} = \bar \S + \td \S_1(\sin)= \bar \S \chi( \f{y}{ C_{in}} ) + (1 -\chi( \f{y}{ C_{in}} )) \geq c > 0 
\]
for  $ |y| > 16 R_4$. It follows $\min_{y\in \R^d} \S_{in}(y) > c > 0$ for some $c>0$ and $\S_{in}(y) = 1$ for $|y| \geq 2 C_{in}$.

After choosing the initial data $\td \VV_{1,in}, \AA(\sin)$, by choosing $R_c > ( \max( R_{c,0}, \d^{-4},  2 C_{in}  , 1 6 R_4 ) + 1)^2$ with $R_{c, 0}$ determined in Theorem \ref{thm:non}, we verify the support condition in \eqref{eq:thm_ass} with $\S_{\infty} = 1$.  Thus, applying Theorem \ref{thm:non}, we can construct a global solution to \eqref{eq:non_V} with the decay estimates in Theorem \ref{thm:non}. For a fixed $y$, the asymptotic convergence in \eqref{eq:blow_asym} follows from the ansatz~\eqref{eq:var_ss} and the exponential decay of the perturbation $(\td \UU(\cdot,s), \td \S(\cdot,s) $ as the self-similar time $s\to \infty$ (equivalently, as $t\to T$), established in \eqref{eq:non_decayc}, and the embedding \eqref{eq:Xm_Linf} in Lemma \ref{lem:prod}.

For later estimate of the vorticity, applying 
\eqref{eq:non_decay_supp} for the support of $U_1$, the estimate \eqref{eq:non_decay} and Lemma \ref{lem:prod} to $U_1$ with $\kp_1 = \f{1}{4}$ \eqref{norm:Xk}, we obtain
\beq\label{eq:blow_pf1}
\bal
  \B|\f{ U_1}{ R_c + e^{-s} y_1} \B|
 \les_m \f{ |\supp(\UU)|^{\kp_1/2}  || \UU ||_{\cX^m}}{R_c}
 \les_m \f{ R_c^{\kp_1/2} e^{  \f{\kp_1 s }{2}  }}{R_c}
  \les_m R_c^{ \f{\kp_1}{2}-1} e^{ \f{\kp_1 s  }{2} } \les_m e^{ \f{\kp_1 s  }{2} }, \ 
  \kp_1 = \f{1}{4} .
  \eal
\eeq

\vspace{0.1in}
\textbf{Step 2. Nontrivial vorticity.}
We compute the initial vorticity associated with $\AA = A(\xi) \ee_{\th}$ using \eqref{eq:vor_id}
A direct computation yields 
\[
\na \times \AA = \pa_1 \AA_{2} - \pa_2 \AA_{ 1}
 =  \xi \pa_{\xi} ( \f{A(\xi)}{\xi}) + 2 \f{A(\xi)}{\xi} 
 = \pa_\xi A(\xi) + \f{A(\xi)}{\xi}.
\]

Recall $\d_A = \d^{5/9}$ from \eqref{eq:para_fix}. Since $\na \times \bar \UU = 0$, using the formula \eqref{eq:V2_init} for the initial data and the estimates \eqref{eq:thm_init2} for $A$, \eqref{eq:thm_ass} for $\td \UU_1$, \eqref{eq:non_decayc} for $\td \VV_{2, u}$,
by choosing $\d$ sufficiently small, we obtain 
\beq\label{eq:vor_init}
\bal
|\om_0 | & =  | \na \times (\td \UU_1 + \td \UU_2 + \bar \UU)(\sin)  | 
 \geq |\na \times \AA| - 
 |\na \times \td  \UU_1(\sin) | - |\na \times \td \VV_{2, u}(\sin) | \\
 & \geq C_{m,1} \d_A - C_{m,2} \d 
  \geq C_{m,3}  \d_A,
 \eal
\eeq
for any $|y| \leq 1$ and some absolute constant $C_{m, 3} > 0$. Thus, the initial vorticity 
is bounded away from $0$ uniformly in $B(0, 1)$.

To estimate the evolution of the vorticity, we use the equation \eqref{eq:vor}. Applying the rescaling relations \eqref{eq:var_ss} and $(T-t)^{1/r} = e^{-s}$ to \eqref{eq:blow_pf1}, we obtain 
\[
  \B| \f{u_1(t)}{ R_c + \xc_1} \B|
\les  (T-t)^{1/r - 1}   |\f{ U_1}{ R_c + e^{-s} y_1} |
\les_m (T-t)^{1/r - 1} 	e^{  \f{\kp_1 s }{2}  }
\les_m 
(T-t)^{ \f{ 1- \f{\kp_1}{2} }{r}   -1 } .
\]
Since $\kp_1 = \f{1}{4}$ \eqref{norm:Xk}, we get $\f{1-\f{\kp_1}{2} }{r} > 0$ and 
\[
  \lim_{t \to T^-} \int_0^t   \B\| \f{u_1(\tau)}{ R_c + \xc_1} \B\|_{L^{\infty}} d \tau 
  \les_m  \int_0^T   (T-t)^{  \f{ 1- \f{\kp_1}{2} }{r} -1} d \tau  \les_m 1.
\]

Solving \eqref{eq:vor} along the trajectory in the original equation \eqref{eq:euler_2D}
\beq\label{eq:traj_phy}
 \f{d}{dt} \Phi(x, t) = \uu( \Phi(x, t) , t), \quad \Phi(x, 0) = x,
\eeq
from any initial location $x $, we obtain 
\beq\label{eq:vor_equi}
    | \f{\om}{\rho}( x, 0) | \les_m  | \f{\om}{\rho}( \Phi(x, t), t) | \les_m    | \f{\om}{\rho}( x, 0) |, 
   \eeq
with absolute constants uniformly in $x$ and $t < T $.  Recall the relations \eqref{eq:sig}, \eqref{eq:var_ss} among the initial data for $\s, \rho, \S_{in}$. Combining  $\rho_0(z) \asymp 1$ 
for $|z| \leq 1$, which is deduced from \eqref{eq:thm_init1}, the lower bound of $\om_0$ in $B(0, 1)$ \eqref{eq:vor_init}, and the relation \eqref{eq:vor_equi}, for any $ |x| \leq 1$, we establish 
\beq\label{eq:blow_pf2}
 | \f{\om}{\rho}( \Phi(x, t), t) |  \gtr_m   | \f{\om}{\rho}( x, 0) |
\gtr_m |\om(x, 0 )| \gtr_m \d_A , \quad 
 | \f{\om}{\rho}( \Phi(x, t), t) | \les_m  | \f{\om}{\rho}( x, 0) | \les_m 1 .
\eeq



\vspace{0.1in}
\textbf{Step 3. Estimate of the trajectory.}
 We further control the trajectory $\Phi(x, t)$ and use \eqref{eq:blow_pf2} to show that $|| \om(t)||_{\infty}$ blows up. Let $\Phi(x, t)$ be the trajectory associated with the velocity $\uu$ \eqref{eq:traj_phy}. According to \eqref{eq:var_ss}, define the corresponding self-similar trajectory to be
\beq\label{eq:blow_pf3}
X(x, s) \teq \frac{\Phi(x, t)}{(T-t)^{1/r}} = \Phi(x, t) e^s.
\eeq
It follows from $\frac{d}{dt} \Phi( x ,t) = \uu(\Phi(x ,t),t)$ and the self-similar transform \eqref{eq:var_ss} that $X(x, s)$ solves the ODE
\[
 \pa_s X(x , s) = X(x, s) + \UU( X(x, s), s), 
 \quad X(x, s_{in}) = \Phi(x, 0) e^{s_{in}}
 =  e^{s_{in}} x .
\]

Now, for $t < T$, since $\Phi(\cdot, t)$ is a bijection 
\footnote{
Let $\td \Phi(\cdot, t)$ be the flow map associated with $\td \uu$ to the axisymmetric Euler equations \eqref{eq:euler_axi} without shifting the center \eqref{eq:axi_recenter}. Since $u^r(0, z) = 0$ due to the axisymmetry, $\td \Phi$ is a bijection from the axis $ \{ (r, z): r= 0\}$ to itself and from $ \{ (r, z) \in (0, \infty) \times \R \}$ to itself. After we recenter the equations as in \eqref{eq:euler_2D} using \eqref{eq:axi_recenter}, the associated flow map $\Phi(\cdot, t)$ is a bijection from $ \{ (\xc_1, \xc_2) : \xc_1 > - R_c \}$ to itself.
}, 
we can define $z(t) = (\Phi(t))^{-1}(0) $, i.e. $\Phi(z(t), t) = 0$. If $|z(t)| \geq 1$, applying Lemma \ref{lem:traj} and \eqref{eq:blow_pf3}, we obtain
\[
 0= |\Phi(z(t), t) | = |X(z(t), s(t)) e^{-s}| >  e^{\kp_1 / 2 (s - s_{in}) }  e^{-s}> 0,
\]
which is a contradiction. Hence, we get $|z(t)| < 1$. Using the relation \eqref{eq:var_ss} and the estimate on the solution \eqref{eq:non_decay}: 
$|\bar \S(0) - \S(0,s)| \leq \f{1}{2} \bar \S(0)$, which can be made by further choosing $\d$ smaller,  we deduce $\S(0, s) \asymp \bar \S(0)$ uniformly for all $s \geq \sin$ and
\[
\s( 0, t) = r^{-1} (T-t)^{1/r-1} \S( 0, s ) 
\asymp (T-t)^{1/r-1}.
\]

Using $|z(t)| < 1$, the definition $\rho = (\al \s)^{1/\al}$ in \eqref{eq:sig},  and \eqref{eq:blow_pf2}, we obtain 
\[
\bal
|\om(t,0)| &= |\f{\om(0, t)}{ \rho(0, t)} | \rho(0, t)
\asymp_m |\f{\om( z(t), 0)}{\rho(z(t), 0)}| \rho(0, t)
\gtr_m \d_A \rho(0, t) \gtr_m  \d_A (T-t)^{ \f{ (1/r-1) }{ \al}}, \\
|\om(t,0)| & \les_m \rho(0, t) \les_m (T-t)^{ \f{ (1/r-1) }{ \al}}. 
\eal
\]

Since $t$ can be any time less than $T$, taking $t \to T^-$, we prove \eqref{eq:vor_blowup} and that $|| \om(t) ||_{\infty}$ blows up.

The support estimate \eqref{eq:thm_supp} follows from \eqref{eq:non_decay_supp} with $\S_{\infty} = 1, \s_{\infty} = \f{1}{r} e^{(r-1) \sin} $, the transform \eqref{ss-var:b}, and the relation 
\eqref{eq:axi_recenter} between \eqref{eq:euler_axi} and \eqref{eq:euler_2D}. We conclude the proof of Theorem \ref{thm:blowup}.
\end{proof}

\subsection{Ideas of the proof of Theorem \ref{thm:non}}\label{sec:non_idea}
It remains to prove Theorem \ref{thm:non}. Since the formula for $\td \VV_2$  \eqref{eq:V2_form2} involves the future of the solution 
$(\td \UU_1, \td \S_1)$, and since $\td \VV_2$ enters the evolution \eqref{eq:non_Va} for $\td \VV_1$ through the nonlinear term, we cannot solve \eqref{eq:non_Va} for the perturbation  $\td \VV_1$  directly. Instead, we reformulate \eqref{eq:non_Va} as a fixed point problem. The argument and estimates are similar to those in \cite{chen2024Euler}.

There is an additional difficulty from the lower order term $\cE_d$ \eqref{eq:euler_ssb}. To show that $\cE_d$ is indeed a lower order term, we need to control the support of $\cE_d$ so that the support of $\cE_d$ remains uniformly away from the symmetry axis $r = 0$ and the denominator $\f{1}{R_c + e^{-s} y_1}$ does not become singular. We remark that the estimates are nontrivial since the solution is not propagated along some characteristics, in which case one can just control the support by estimating the ODE for characteristics, see e.g., \cite{ChenHou2023a,chen2019finite2}.
Instead, we estimate the growth of the support of $\td \VV$ by designing a suitable time-dependent energy in section \ref{sec:supp}.

In the rest of this section, we fix the initial data $\td V_1(\sin) =(\td \UU_1, \td \S_1) |_{s = \sin} \in \cX^{m+2}, \AA \in \cX^{m+2}$ and $\S_{\infty}$ satisfying the assumptions of Theorem \ref{thm:non}.

We define the spaces $Y_1$ and $Y_2$, which capture the decay of the solutions in time 
\begin{subequations}
\label{norm:fix}
\begin{align}
 \| (\td \UU_1,\td \S_1)\|_{ Y_1}  
 & \teq \sup_{s\geq \sin} e^{ \lam_1 (s-\sin)}   
  \| (\td \UU_1(s), \td \S_1(s) )\|_{ \cW^{m+1}} , 
 \label{norm:fix:a} \\
 \| (\td \UU_1,\td \S_1)\|_{Y_2}  
 & \teq \sup_{ s \geq \sin} e^{ \lam_1 (s-\sin) }   
\| (\td \UU_1(s), \td \S_1(s) )\|_{ \cX^{m}} .
\label{norm:fix:b}
\end{align}
\end{subequations}
Showing $(\td \UU_1,\td \S_1) \in Y_i$ for $i\in\{1,2\}$ implies that suitable norms of $\td \UU_1(s), \td \S_1(s)$ decay with a rate $e^{-\lam_1 s}$ as $s\to \infty$; recall cf.~\eqref{eq:decay_para} that $\lambda_1 = \frac{9}{10} \lambda$.

Next, we construct the map $\cT$, whose fixed point will be the solution $\td \VV_1$ to \eqref{eq:non_Va}. To distinguish various notations, we will use variables with a "hat", e.g. $(\hat \UU_1, \hat \S_1)$, for the input of the map $\cT$ or $\cT_2$ \eqref{eq:T2}, and use variables with a tilde, e.g.  $(\td \UU_1, \td \S_1)$, for the output of the map. Below, we construct the map $\cT$ in two steps.

Firstly, given $( \wh \UU_1, \wh \S_1 ) \in Y_2$, following \eqref{eq:V2_form2a}, we define 
\beq\label{eq:non_fix_V2}
  \td \VV_2 = e^{\cL(s-\sin)} \AA +   \cT_2( \wh \UU_1, \wh \S_1),
  \quad \mathrm{for \ all  \ } s \geq \sin.
\eeq

Secondly, we solve $\td \VV_1$ from \footnote{
	The local existence of solution $\td \VV_1$ follows the same ideas in \cite[Section 4.4]{chen2024Euler}. We refer to \cite{chen2024Euler} for more discussion.
} 
\beq\label{eq:non_fix}
\bal
  \pa_s \td \VV_1  & = \cL  \td \VV_1 - \cK_m(\td \VV_1) + \cN(  \td \VV_1 + \td \VV_2
 ), \\
 \td \VV_1 |_{s = \sin} & = (\td \UU_1(\sin), \td \S_1(\sin ) ) .
 \eal
\eeq

The above two steps defines a map $\cT$, which takes input $( \wh \UU_1, \wh \S_1)$, and outputs the solution to \eqref{eq:non_fix}:
\beq\label{eq:non_fix_V1}
(\td \UU_1, \td \S_1) = \td \VV_1 \teq \cT( \wh \UU_1, \wh \S_1). 
\eeq

We choose the following parameters for Theorem \ref{thm:non} and the size of the norm 
\beq\label{eq:para_fix}
 R_4^2, \ \d^{-4} \ll R_c, \quad 
 \d_Y = \d^{2/3},\quad \d_A = \d^{5/9},  \quad \d \ll_m 1 . 
\eeq

The proof of Theorem \ref{thm:non} consists of two propositions and Lemma \ref{lem:supp}. In Proposition \ref{prop:onto}, we show that $\cT$ maps a ball with norm $Y_2$ and radius $\d_Y$ to itself, and to a more regular space $Y_1$. In Proposition \ref{prop:contra}, we show that $\cT$ is a contraction for the topology $Y_2$.

\begin{proposition}\label{prop:onto}
Recall that the initial data $\td \VV_{in}, \AA$ satisfies the assumptions \eqref{eq:thm_ass} with parameters $R_c, \d$. Denote \footnote{
In general, the solution $(\UU, \S)$ defined in 
Proposition \ref{prop:onto} does not agree with the solution $(\UU, \S)$ to \eqref{eq:non_V} with decomposition \eqref{eq:init}. Yet, for a fixed point to the map $\cT$, these two solutions are the same. We abuse the notation for these reasons.
\label{foot:prop_onto}
} 
$(\UU, \S)= (\bar \UU, \bar \S) + (\cT + \cT_{2})( \wh \UU_1, \wh \S_1)+  e^{\cL (s-\sin)} ( \AA, 0)$. There exists $\d_0 \ll_m 1$, $R_{c,0} > (16 R_4+1)^2$ and a constant $\bar C_m > 0$ such that for any  $\d, R_c$ with $0<\d < \d_0, R_c > R_{c, 0}$ and any $(\wh \UU_1, \wh \S_1) \in Y_2$
 with  $|| (\wh \UU_1, \wh \S_1)||_{Y_2} < \d_Y$, we have 
 \begin{gather}
\label{eq:prop_onto_est}
  || \cT( \wh \UU_1, \wh \S_1)(s) ||_{Y_1} < 2 \d, \     || \cT( \wh \UU_1, \wh \S_1)(s) ||_{Y_2} < \d_Y,
 \  || \cT_2( \wh \UU_1, \wh \S_1)(s) ||_{\cX^{m+3}} \les_m e^{-\eta_s (s-\sin)} \d_Y, \\
\label{eq:prop_supp}
 e^{- s } \supp( \UU ) \cup \supp(\S - e^{ - (r-1) (s-\sin)} \S_{\infty} ) < 
\bar C_m R_c^{1/2} < R_c / 4.
\end{gather}
\end{proposition}


\begin{proposition}\label{prop:contra}
There exists a positive $ \d_0\ll_m 1$ such that for any $\d < \d_0$ and any pairs $( \wh \UU_{1,a}, \wh \S_{1,a})$, $( \wh \UU_{1,b}, \wh \S_{1,b}) \in Y_2$ with $\| ( \wh \UU_{1,a}, \wh \S_{1,a})\|_{Y_2} < \d_Y$ and $\| ( \wh \UU_{1,b}, \wh \S_{1,b})\|_{Y_2} < \d_Y$, we have 
\[
 \|\cT(\wh \UU_{1, a}, \wh \S_{1,a}) - \cT( \wh \UU_{1,b}, \wh \S_{1,b}) \|_{Y_2} < \tfrac{1}{2} 
  \|( \wh \UU_{1, a}, \wh \S_{1,a}) - ( \wh \UU_{1,b}, \wh \S_{1,b}) \|_{Y_2}  .
\]
\end{proposition}

Note that the equation of $\td \VV_1$ \eqref{eq:non_Va} involves $ \na \td \VV_2$ via the nonlinear terms, which loses one derivative. To compensate it, we use the smoothing effect of $\cK_m$ and the estimates from $(\wh \UU_1, \wh \S_1)$ to $\td \VV_2$ \eqref{eq:V2_form2} in Lemma \ref{lem:T2}. Due to the transport term, e.g. $\UU \cdot \na \UU$, we choose $Y_2$ less regular than $Y_1$ to estimate the difference $\cT(\wh \UU_{1,a}, \wh \S_{1,a}) - \cT( \wh \UU_{1,b}, \wh \S_{1,b} ) $.

\begin{proof}[Proof of Theorem \ref{thm:non}]
The above two Propositions allows us to apply the Banach fix-pointed theorem to the map $\cT$ in  $B_{Y_2} = \{ (\td \UU, \td \S):  ||(\td \UU, \td \S)||_{Y_2} < \d_Y \}$ and construct the fixed point $\td \VV_1 \in B_{Y_2}$ with $\td \VV_1 = \cT (\td \VV_1)$. By defining $\td \VV_2 = \cT_2 (\td \VV_1)$, we obtain a global solution $\td \VV_1, \td \VV_2$ to \eqref{eq:non_V}. Moreover, using the estimate $ || (\td \UU_1, \td \S_1) ||_{Y_1} < 2 \d$ in Proposition \ref{prop:onto}, the definitions of $Y_i$ in \eqref{norm:fix} and $\td \VV_{2,u} $ in \eqref{eq:V2_form2c}, Lemma \ref{lem:T2} and Proposition \ref{prop:semiA} for the decay estimates of $\cT_2$ and $e^{\cL s} \AA$, we derive
\[
\bal
  ||(\td \UU_1, \td  \S_1) ||_{\cW^{m+1}}  < 2 \d e^{-\lam_1 (s -\sin)},
 \  || \td \VV_{2,u} ||_{\cX^{m+2}} \les_m \d e^{-\eta_s (s- \sin)}, \ 
    ||(\td \UU_2, \td  \S_2) ||_{\cW^{m+2}} \les_m \d_A e^{ -\eta_s( s - \sin)},
\eal
\]
which prove \eqref{eq:non_decay} except for \eqref{eq:non_decay_supp}. Since $(\td \UU_1, \td \S_1)$ is the fixed point for $\cT$, in light of \eqref{eq:non_fix_V1} and \eqref{eq:non_fix_V2}, the solution $(\UU, \S)$ defined in Proposition \ref{prop:onto} agree with the full solution $(\UU, \S)$ \eqref{eq:init}. By choosing $R_{c, 0}$ to be the same $R_{c, 0}$ as in Lemma \ref{lem:supp} and using \eqref{eq:prop_supp} in Proposition \ref{prop:onto}, we prove the estimate \eqref{eq:non_decay_supp}.
\end{proof}


\vspace{0.1in}
\noindent \textbf{Organization.} In the remaining sections, we prove Propositions \ref{prop:onto} and \ref{prop:contra}. We first estimate the evolution of the support of the solution $(\UU, \S)$ in Section \ref{sec:supp}. In Section \ref{sec:T2}, we estimate the linear map $\cT_2$ for $\td \VV_2$ \eqref{eq:V2_form2}. 
In Section \ref{sec:fix_pt}, we prove Proposition \ref{prop:onto}. In Section \ref{sec:contra}, we prove Proposition \ref{prop:contra}.

\subsection{Estimate of the support}\label{sec:supp}

By abusing the notation, we denote by $( \UU, \S)$ the full solution to \eqref{eq:non_fix} 
\footnote{We remark that this full solution $(\UU,\S)$ is different from the solution $(\UU, \S)$ to \eqref{eq:euler_ssa}, or equivalently \eqref{eq:non_V} with the decomposition \eqref{eq:init} since we do not have $(\td \UU_1, \td \S_1) = (\wh \UU_1, \wh \S_1)$  in general.
}
and the support
\begin{align}
  \UU  & = \td \UU_1 + \td  \UU_2 + \bar \UU,
 \quad  \S =\td  \S_1 + \td  \S_2 + \bar \S  , \label{eq:solu_full} \\
S(s) & \teq  \inf_{R \geq 0} \B( \supp(  \UU(t) ) \cup \supp( \S(t) -e^{-(r-1) (s - \sin) }  \S_{\infty} )  \B) \subset B(0, R)  \label{eq:supp}. 
\end{align}
The definition of $(\UU, \S)$ is the same as that in Proposition \ref{prop:onto}, but different from \eqref{eq:init} in general. See Footnote \ref{foot:prop_onto} for more discussions.

We have the following bootstrap estimates for the growth of the support. 

\begin{lemma}[Bootstrap estimate for the support]\label{lem:supp}
Let $S(s), R_4$ be defined in \eqref{eq:supp} and Theorem \ref{thm:coer_est}, $\td \VV_i$ be the solution to \eqref{eq:non_fix}, \eqref{eq:non_fix_V2}, and $(\UU, \S)$ be defined in \eqref{eq:solu_full}. There exists $R_{c,0} > (16 R_4)^2 + 1$ large enough and a constant $\bar C_m > 0$, such that for any $R_c > R_{c, 0}$, the following statement holds true. Suppose that for some $s_1 > \sin$ and $\S_{\infty} >0$, we have
\bseq\label{eq:supp_ass}
\begin{gather}
\supp(  \UU(\sin) ) \cup \supp(  \S(\sin)
 -  \S_{\infty} )
 \subset B(0, R_c^{1/2}),
 \label{eq:supp_assa}  \\
   \sup_{s \in [\sin, s_1)} ||( \td \UU_1, \td \S_1)(s) ||_{\cX^m} 
+ || ( \td \UU_2, \td \S_2)(s) ||_{\cX^m} 
    < 1,
  \label{eq:supp_assb}
 \end{gather}
\eseq
Then we have 
\beq\label{eq:lem_supp}
 e^{-s} S(s)  \leq \bar C_m R_c^{1/2} < R_c / 4, \quad s \in [ \sin , s_1).
\eeq
\end{lemma}

We perform weighted $L^2$ estimate and show that the $L^2$ energy restricted to $|y| > 2 a(s)$ is $0$ by choosing $a(s)$ carefully, which further controls the support size.

\begin{proof}
Let $\chi_S : \R^d \to [0, 1]$ be a radially symmetric smooth cutoff function with $\chi_S(y) = 0, |y| \leq 1, \chi_S(y) = 1, |y| \geq 2$. Denote 
\bseq\label{eq:supp_nota}
\begin{gather}
 \mr{\S}(s, y) =  \S(s, y) - b(s),\quad  b(s) =  \S_{\infty} e^{-(r-1) (s - \sin)}, 
\quad \cK_{\D} = \cK_m(\wh \UU_1, \wh \S_1)  - \cK_m( \td  \UU_1,  \td \S_1) , \label{eq:supp_notaa} \\
 M = R_c^{1/2} , \quad \chi_a(t, y) = \chi_S(y / a(t)) , \label{eq:supp_notab}
\end{gather}
\eseq
where $a(s)$ is increasing with $a(\sin) > M$ and will be chosen later.

\vspace{0.1in}
\textbf{Step 1. Equations of $(\UU, \mr \S)$.} By definition of \eqref{eq:non_fix} and \eqref{eq:V2_form2}, we have 
\beq\label{eq:non_fix_V2_eq}
\bal
 \pa_s \td \VV_2 & = \cL \td  \VV_2 + \cK_m(\wh \UU_1, \wh \S_1) ,
 \quad  \td \VV_2( \sin )=  \AA - \td \VV_{2, u}( \sin) \chi(y / (8 R_4)) \subset( B(0, M)). 
 \eal
\eeq

The above equation for $\td \VV_2$ only differs from \eqref{eq:non_V} for $\td \VV_2$  by $\cK_{\D}$. Since $ \pa \mr \S = \pa  \S$ for $\pa = \na, \pa_s$, combining \eqref{eq:non_fix_V2_eq} and \eqref{eq:non_fix}, 
we obtain the equations for $(\UU, \S) = \bar \VV + \td \VV_1 + \td \VV_2$ \eqref{eq:solu_full} similar to \eqref{eq:euler_ss} 
\bseq\label{eq:supp_EE0}
\beq
\bal
 \pa_s  \UU + (r-1)  \UU + (y + \UU ) \cdot \na   \UU + \alpha  \S \na \mr \S  &=  \cK_{\D, U},  \\
 \pa_s ( \mr \S + b(t)) + (r-1) ( \mr \S + b(t)) + (y + \UU) \cdot \na \mr \S + \alpha 
 \S \cdot \div(  \UU)  & = \cK_{\D, \S} + \cE_d , \\
\eal
\eeq
where $\cK_{\D, U}, \cK_{\D, \S}$ are the $U, \S$ component of $\cK$, and $\cE_d$ is given by 
\beq\label{eq:supp_EE0b}
 \cE_d = - \al (d-2) \f{ e^{-s}  U_1 \S }{  R_c +  e^{-s} y_1 } .
\eeq
\eseq
We keep using $ \S $ instead of $b(t) + \mr \S$ in the terms $\alpha   \S \na \mr \S, \alpha  \S \cdot \div(  \UU), \cE_d$. Since  $\f{d}{dt} b(t) + (r-1) b(t) =0$, we can rewrite the equation for $\mr \S$ as follows 
\[
 \pa_s  \mr \S  + (r-1)  \mr \S  + (y +  \UU) \cdot \na \mr \S + \alpha  \S \cdot \div(  \UU)   = \cK_{\D, \S} + \cE_d . 
\]

\vspace{0.1in}
\textbf{Step 2. Estimate of the functional.}
We impose the following bootstrap assumption on the support size 
\beq\label{eq:lem_boot_supp}
 		 e^{-(s - \sin)} S(s) < R_c / 4 .
\eeq
Using the continuity of the solution (see Remark \ref{rem:supp}), the assumption \eqref{eq:supp_assa} on $S(\sin)$, and $R_c / 4 > R_c^{1/2}$ for $R_c >16$, we obtain that \eqref{eq:lem_boot_supp} holds true for $s \in [\sin, \sin + \d]$ with some small $\d > 0$. 

Below, we show that \eqref{eq:lem_boot_supp} can be improved. We estimate 
\[
I(s) \teq \int (|  \UU|^2 + \mr \S^2 ) \chi_a(s, y) ,
\]
and show that $I(s) = 0$ by choosing $a(s)$ for $\chi_S$ in \eqref{eq:supp_notab} appropriately. 

From item (a) in Proposition \ref{prop:compact}, we have $\supp( \cK_m( \td \UU_1,  \td \S_1)), \supp( \cK_m( \wh \UU_1, \wh \S_1 ) \subset B(0, 8 R_4)$. Since $a(s) \geq 16 R_4$ for any $s \geq \sin$, we get $\chi_a(s, y) = 0$ for $|y| \leq 16 R_4$ and $\chi_a(s, y) \cdot \cK_{\D} = 0$ (see definition of $\cK_{\D}$ from \eqref{eq:supp_notaa}). Thus, $\cK_{\D}$ vanishes in the estimate of $I(s)$, and we obtain 
\beq\label{eq:supp_EE}
\bal
\f{1}{2}\f{d}{ds} I(s) & = \int 
 - \B( (y +  \UU ) \cdot \na \UU \cdot  \UU  
+ (y +  \UU ) \cdot \na \mr \S \cdot  \mr \S 
+ \al  \S ( \div( \UU) \cdot  \mr \S +  \UU \cdot \na \mr \S ) \B) \chi_a \\
 & \quad + 
\f{1}{2} \B(|  \UU|^2 + | \mr \S|^2 \B)  \pa_s \chi_a +  \cE_d \mr \S \chi_a
- (r-1) (|  \UU|^2 + | \mr \S|^2 ) \chi_a
 \teq J_1 + J_2 + J_3 + J_4.
 \eal
\eeq

\noindent \textbf{Estimate of $J_1$.}
To estimate $J_1$, we apply integration by parts. If the derivative $\na$ acts on $ y +  \UU, \mr \S$, we use $\na \mr \S = \na  \S$ from \eqref{eq:supp_notaa} and bound it using the energy. 
For $J_1$, we estimate
\[
\bal
J_1 &= \int \f{1}{2} \na \cdot ( ( y+  \UU) \chi_a ) (|  \UU|^2 + | \mr \S|^2) 
+ \na  (\al   \S \chi_a) \cdot (  \UU \mr \S) \\
&\leq  C_m  || ( \UU,  \S||_{ \dot W^{1, \infty}} I(s)
+ \int  \f{1}{2} (y +  \UU) \cdot \na \chi_a (|  \UU|^2 + | \mr \S|^2) 
+ \al  \S \na \chi_a  \cdot (  \UU \mr \S) .
\eal
\]
where $\| f \|_{\dot W^{1, \infty}} \teq || \na f ||_{L^{\infty}}$. Since $\chi_a$ is radially symmetric and $\pa_{\xi} \chi_a \geq 0$, using 
$|p q| \leq \f{1}{2}(p^2 + q^2)$,  we obtain 
\[
\f{1}{2} (y +  \UU) \cdot \na \chi_a (|  \UU|^2 + | \mr \S|^2) 
+ \al  \S \na \chi_a  \cdot (  \UU \mr \S)
\leq \f{1}{2} ( ( \xi +  U^R ) \pa_{\xi} \chi_a
+ \al|  \S|  \pa_{\xi} \chi_a  ) (|  \UU|^2 + | \mr \S|^2)  ,
\]
where $  U^R = \UU \cdot \ee_R$. Combining the above two estimates, we establish 
\beq\label{eq:supp_J1}
J_1  \leq  C_m  || ( \UU,  \S||_{ \dot W^{1, \infty}} I(s)
+ \int \f{1}{2} ( \xi +  U^R + \al |\S| ) \pa_{\xi} \chi_a 
 (|  \UU|^2 + | \mr \S|^2) .
\eeq

\noindent
\textbf{Estimate of $J_3$.} Next, we estimate $J_3$ defined in \eqref{eq:supp_EE}. Recall $\cE_d$ from \eqref{eq:supp_EE0b}.
Applying Lemma \ref{lem:prod} for $\S$ within the support of $U_1$, the assumption $e^{- s  } S(s) < \f{R_c}{4}, R_c > 1$ \eqref{eq:lem_boot_supp}, and $\kp_1/ 2 < 1$, we obtain 
\beq\label{eq:supp_J3_est1}
|  U_1 \S | \les |U_1| \cdot ||  \S||_{ \cX^m} S(s)^{\kp_1/2},
\quad 
|\cE_d| \les  || \S ||_{ \cX^m}  \f{ S(s)^{\kp_1/2} e^{-s} }{R_c}| U_1|
\les ||  \S ||_{ \cX^m} | U_1|.
\eeq

Applying the above estimates to $J_3$ defined in \eqref{eq:supp_EE}, we yield 
\beq\label{eq:supp_J3}
J_3 \les || \S ||_{ \cX^m}  I(s).
\eeq

\noindent \textbf{Estimate of $J_2$.} Next, we estimate $J_2$ in \eqref{eq:supp_EE}. 
Recall $\chi_a$ from \eqref{eq:supp_notab}. Note that 
\[
\pa_s \chi_a = - \xi \cdot \f{\pa_s a}{a^2} (\pa_{\xi} \chi_S)( \f{y}{a} ),  \quad 
\pa_{\xi} \chi_a = a^{-1} (\pa_{\xi} \chi_S)( \f{y}{a}) ,
\quad \pa_s \chi_a 
= - \xi \f{\pa_s a}{a} \cdot \pa_{\xi} \chi_a.
\]

We want to choose
\beq\label{eq:ODE_a}
 \f{ \pa_s  a(s) }{a(s)}\geq \f{1}{ \xi } ( \xi + U^R + \al | \S|),  \quad \xi = |y|, \ y  \in \supp( \pa_{\xi} \chi_a), \quad a(0) = 2  M,
\eeq
and then use the above formula of $\pa_s \chi_a$, and the estimates of $J_1$ \eqref{eq:supp_J1} and $J_2$ \eqref{eq:supp_EE} to obtain 
\beq\label{eq:supp_J12}
J_1 + J_2 \leq  C_m  || ( \UU, \S||_{ \dot W^{1, \infty}} I(s) 
+  \f{1}{2} \int ( \xi + U^R + \al |\S| - \xi \f{\pa_s a}{a} ) \pa_{\xi} \chi_a ( |\UU^2| + |\mr \S|^2 ) 
\leq C_m  || ( \UU, \S||_{\dot W^{1, \infty}} I(s) .
\eeq

For $J_4$ in \eqref{eq:supp_EE}, we have $J_4=-(r-1) I(s)$. Combining \eqref{eq:supp_J3}, \eqref{eq:supp_J12}, the assumption on the initial support \eqref{eq:supp_assa}, and the embedding \eqref{eq:Xm_Linf}, we obtain 
\beq\label{eq:supp_ODE}
 \f{d}{d s} I(s) \leq C_m  ( || ( \UU,  \S||_{ \cX^m} + 1 ) I(s) ,
 \quad I(\sin) = 0.
\eeq

Using the assumption \eqref{eq:supp_assb}, and the decomposition \eqref{eq:solu_full}, we obtain 
$|| ( \UU,  \S)(s) ||_{ \cX^m} \les_m 1 $ for $s \in [\sin, s_1]$. Using Gronwall's inequality, we obtain that for $s \in [\sin, s_1]$, we have $I(s) = 0 $ and $ \UU = 0, \mr \S = 0$ for $|y| \geq 2 a(s)$. Thus, $a(s)$ controls the support size $S(s) \leq 2 a(s)$. 

It remains to choose $a(s)$ to achieve \eqref{eq:ODE_a}. Using the assumption \eqref{eq:supp_assb} on $\td \UU_i, \td \S_i$, Lemma \ref{lem:prod}, and the decay estimates on $\bar \UU, \bar \S$ \eqref{eq:dec_U}, we obtain 
\[
  | \UU(y)| + \al | \S(y)| \les_m \la y \ra^{\kp_1/2},
\]
where $\kp_1 = \f{1}{4}$ \eqref{norm:Xk}. Since for $y \in  \supp( \pa_{\xi} \chi(y / a(s)))$, we get $R = |y| \in [a(s), 2 a(s)]$, we only need to choose $a(s)$ with
\[
\f{d}{d s} a(s) = a(s) ( 1  + C_m a(s)^{-1 + \kp_1/2}), \quad a(\sin) = 2 M , 
 \]
 which is equivalent to 
 \[
 \f{d}{d s} a(s) e^{-s} = C_m a(s)^{\kp_1/2}, \quad  \f{d}{d s} (a(s) e^{-s})^{\kp_1/2} = e^{-s (1 -\kp_1/2)} C_{m,2} .
 \]

Clearly, the ODE has a global solution with 
\beq\label{eq:lem_boot_supp2}
 a(s) e^{- s}  \leq  ( C_{m, 3} + (a(\sin) e^{-\sin})^{\kp_1/2} )^{2/\kp_1}
 \leq C_{m, 4}   + 2  a(\sin ) e^{-\sin} 
 \leq C_{m , 4 } + C R_c^{1/2}  \leq C_{m , 5} R_c^{1/2} .
\eeq

By choosing $R_{c, 0}$ sufficiently large, so that $R_{c, 0}^{1/2} > 16 C_{m, 5}$, we get $a(s) e^{-s} < R_c / 8$ for $R_c > R_{c, 0}$. Since $S(s) \leq 2 a(s)$, we close the bootstrap assumption \eqref{eq:lem_boot_supp}. Using \eqref{eq:lem_boot_supp2}, we conclude the proof.
\end{proof}

\begin{remark}\label{rem:supp} 
In the above proof, the bootstrap assumption \eqref{eq:lem_boot_supp} is only used  
\textit{qualitatively} to bound the denominator in $\cE_d$ \eqref{eq:supp_EE0b}, \eqref{eq:supp_J3_est1}. In fact we only need that $S(s)$ does not grow to $\infty$ in a very short time. In practice, for initial data $\td \VV_{1,in}, \AA$ satisfying \eqref{eq:thm_ass}, one can first construct local solution $\td \VV_i$ to \eqref{eq:non_fix}, \eqref{eq:non_fix_V2} following \cite[Section 4.3]{chen2024Euler} by regularizing the equations and  $\cE_d$ as $\cE_{d, \e} = - \al (d-2) \f{ e^{-s}  U_1 \S }{  ( (R_c +  e^{-s} y_1)^2 + \e^2 )^{1/2} } $ for $\e > 0$.  Since the denominator in $\cE_{d, \e}$ is bounded from below by $\e$ , we can repeat the above proof of Lemma \ref{lem:supp} without assuming \eqref{eq:lem_boot_supp}. The only difference in the estimate of the support is that we get $\e$-dependent constants in the estimates \eqref{eq:supp_J3_est1}, \eqref{eq:supp_J3}, which further results in the estimate \eqref{eq:supp_ODE} with a constant $C_{m,\e}$. Yet, we still obtain $I(s) = 0$ for $s \in [\sin, s_1]$ via Gr\"onwall, which further implies the estimate \eqref{eq:lem_supp}. 
This estimate \eqref{eq:lem_supp} is independent of the regularization parameter $\e > 0$. Using the standard Picard iteration and approximation, one can construct the local solution $\td \VV_1$ and $\td \VV_2$ \eqref{eq:non_fix_V2} satisfying \eqref{eq:supp_assb} and use the above argument to further obtain that the solution $(\UU, \S)$ obtained by \eqref{eq:solu_full} satisfies the support condition \eqref{eq:lem_supp}.
\end{remark}

\subsection{Estimates on $\cT_2$}\label{sec:T2}

We have the following decay estimates and smoothing effects for $\cK_m$.
\begin{lemma}\label{lem:decay_Km}
Recall the decomposition of $\cX^m$ \eqref{eq:dec_X}. Let $\Pi_s, \Pi_u$ be the projection of $\cX^m$ into the stable part $\cX^m_s$ and unstable part $\cX^m_u$, respectively. For any $f \in \cX^{m}$ 
and all $s \geq \sin$, we get 
\[
\bal
 || e^{\cL (s- \sin)} \Pi_s \cK_m f ||_{\cX^{m+3}} & \les_m  e^{-\eta_s (s-\sin)}  || f ||_{\cX^{m}}, \\
 || e^{ - \cL (s-\sin)} \Pi_u \cK_m f ||_{\cX^{m+3}} 
& \les_m  e^{\eta (s-\sin)}  || f ||_{\cX^{m}}. 
\eal
\]
\end{lemma}

A similar result is proved in \cite[Lemma 4.5]{chen2024Euler}, whose proof is based on the semigroup estimates \eqref{eq:decay_stab}, \eqref{eq:decay_unstab}, the energy estimates of $\cL$ in Theorem \ref{thm:coer_est}, and does not further require any symmetry assumption. In particular, the proof is the same as that of \cite[Lemma 4.5]{chen2024Euler}, and we omit the proof.

The following Lemma for $\cT_2$ is the same as \cite[Lemma 4.6]{chen2024Euler}. 
\begin{lemma}\label{lem:T2}
For $( \wh \UU_1,  \wh \S_1) \in Y_2$ and any $s \geq \sin$,we have 
\[
 ||\cT_2( \wh \UU_1, \wh \S_1 )(s) ||_{\cX^{m+ 3}} \les_m e^{-\eta_s (s - \sin)} 
 ||(\wh \UU_1, \wh \S_1)||_{Y_2}. 
\]
\end{lemma}

The proof of \cite[Lemma 4.6]{chen2024Euler} is based on the decay estimates in Lemma \eqref{lem:decay_Km} for $\td \VV_{2, u}, \td \VV_{2, s}$ defined in \eqref{eq:V2_form2}, and decay estimates in Proposition \ref{prop:far_decay} for $e^{\cL (s-\sin)}( \td \VV_{2, u}(\sin) (1 - \chi(y / (8 R_4) )) )$. It does not require any symmetry assumption.
Thus, the proof is the same as that of \cite[Lemma 4.5]{chen2024Euler}, and we omit the proof.

\subsection{Proof of Proposition \ref{prop:onto}}\label{sec:fix_pt}

Before proving Proposition \ref{prop:onto}, we establish a few estimates. 
\subsubsection{Product estimates} 
We have the following product estimates for the nonlinear terms. 
\begin{lemma}\label{lem:non_main}
Let $\cN_{\UU}, \cN_{\S}$ be the nonlinear terms defined in \eqref{eq:non}.
For any $ k \geq m_0$, we have 
\[
\bal
 \B|\B| \vp_{2k} \vp_g^{1/2}  \B( \D^k \cN_{\UU}( \td \UU, \td \S ) 
+\td  \UU \cdot \na \D^k \td \UU + \al \td \S \na \D^k \td \S
 \B) \B|\B|_{L^2} \les_k || (\td \UU, \td \S) ||_{\cX^k}^2, \\ 
    \B|\B| \vp_{2k}  \vp_g^{1/2} \B( \D^k \cN_{\S}( \td \UU,\td \S )   + \td \UU \cdot \na \D^k \td \S + \al \td \S \div (\D^k \td \UU)
   \B) ||_{L^2} \les_k || ( \td \UU, \td \S) \B|\B|_{\cX^k}^2 . \\
   \eal
\]
\end{lemma}

\begin{proof}
We estimate  a typical term $\td \UU \cdot \na \td \UU$ in $\cN_{\UU}$ \eqref{eq:non}. 
Using the Leibniz rule, recalling the definitions of the weights $\vp_{2k}$ in Lemma \ref{cor:wg}  and $\vp_g$ (see \eqref{norm:Xk}), and using Lemma~\ref{lem:prod}, we have
\[
\bal
  \B\| \vp_{2k} \vp_g^{1/2} \B( \D^k( \td \UU \cdot \na \td \UU) - \td  \UU \cdot \na \D^k \td \UU \B) \B\|  
\les_k \sum_{ 1 \leq i \leq 2 k} \B| \B| \vp_{2k} \vp_g^{1/2} |\na^i \td \UU| \cdot |\na^{2k + 1-i} \td \UU | \B| \B|_{L^2} \les_k || \td  \UU ||_{\cX^m}^2.   \\
\eal
\]
Other terms in $\cN_U, \cN_{\S}$ (see \eqref{eq:non} )are estimated similarly. 
\end{proof}

We have the following estimate for the additional terms $\cE_d$ in \eqref{eq:euler_ss}
\beq\label{eq:err}
\cE_d(\UU, \S) =  - \al(d-2) \f{e^{-s} U_1 \S }{R_c +e^{-s} y_1}.
\eeq
\begin{lemma}\label{lem:err}
Suppose that $R_c>1$ and $\supp(\UU) \in B(0, S_u)$ with $S_u$ satisfying $S_u > 1, e^{-s} S_u < R_c / 2$. For any $m \geq 2$, we have 
\[
|| \cE_d(\UU, \S) ||_{\cX^m} \les_m  ( \f{e^{-s}}{R_c})^{1/2} || \UU||_{\cX^m} || \S ||_{\cX^m}
\]
\end{lemma}

\begin{proof}
For $ y \in \supp(\UU)$, we have $\la y \ra \les S_u$, which along with the assumption $ S_u < \f{1}{2} e^{s } R_c $ 
imply
\[
|\cE_d | \les  \f{ e^{-s} S_u^{\kp_1/2}}{R_c} |  U_1 \S | \la y \ra^{- \f{ \kp_1}{2} }
\les \f{ e^{-s} (e^s R_c)^{\kp_1/2} }{R_c} 
 | U_1 \S | \la y \ra^{ - \f{ \kp_1}{2} }
 \les  (\f{e^{-s}}{R_c})^{1-\kp_1/2}  | U_1 \S |  \la y \ra^{ -\f{ \kp_1}{2} }.
\]
Applying the product rule to Lemma \ref{lem:prod} to $ \la y \ra^{-\kp_1/2} \UU_1 \S$ and using $\na^i \la y \ra^{-\kp_1/2} \les_i \la y \ra^{-\kp_1/2-i}$ , we prove
\[
 || \cE_d(\UU, \S) ||_{\cX^m} \les_m  (\f{e^{-s}}{R_c})^{1 -\kp_1/2} || \UU ||_{\cX^m}
  || \S ||_{\cX^m}
  \les_m  (\f{e^{-s}}{R_c})^{1/2} || \UU ||_{\cX^m}
  || \S ||_{\cX^m} ,
\]
where we have used $1 - \f{ \kp_1}{2} = \f{3}{4} > \f{1}{2}$ \eqref{norm:Xk} in the last inequality.

\subsubsection{Proof of Proposition \ref{prop:onto}}


Next, we prove Proposition \ref{prop:onto}, which is similar to that in \cite{chen2024Euler}. 
Recall the notations from the beginning of Section~\ref{sec:non}, in particular~\eqref{eq:init}, and \eqref{eq:T2}. From the assumption of Proposition~\ref{prop:onto}, we assume  $|| ( \wh \UU_1, \wh \S_1 ) ||_{Y_2} < \d_Y$ for some $\d_Y$ small to be determined

Define $\td \VV_2$ using \eqref{eq:non_fix_V2}, and then define $\td \VV_1$ as the solution of~\eqref{eq:non_fix}. Denote
\begin{subequations}\label{eq:non_E}
\begin{align}
  \td \VV_2 & =  \cT_2( \wh \UU_1, \wh \S_1 ) + e^{\cL(s - \sin)} (\AA, 0), 
\quad  \td \VV = \td \VV_1 +  \td \VV_2, \quad  \td \UU = \td \UU_1 +  \td \UU_2, \quad \td \S = \td \S_1 + \td  \S_2 , \\
  \wh E & = \| ( \wh \UU_1, \wh \S_1) \|_{Y_2} < \d_Y, \label{eq:non_E_hat} \\ 
 E_{m+1}& (s)  \teq  \| (\td \UU_1(s), \td \S_1(s)) \|_{\cW^{m+1}}, \label{eq:non_Ec}  \\
  \cE(s)  & = E_{m+1}(s) + \|  \td \VV_2(s) \|_{\cX^{m+2}}.
\end{align}
\end{subequations}

We will estimate the energy $E_{m+1}$. From the equivalence between $\cX^{m+1}$ and $\cW^{m+1}$ (see \eqref{norm:Xk} and \eqref{norm:Wk}), and the definition of $E_{m+1}$ \eqref{eq:non_Ec}, we get $|| (\td \UU_1, \td \S_1) ||_{\cX^{m+1}} \les_m  E_{m+1}$, which implies 
\beq\label{eq:non_err}
|| \td \VV_2 ||_{\cX^{m+2}} \les \cE, 
\quad || \td \VV_1 ||_{ \cX^{m+1}} \les E_{m+1},  \quad || \td \VV ||_{\cX^{m+1}} \les_m || \td \VV_1||_{\cX^{m+1}} 
 + || \td \VV_2||_{\cX^{m+1}}  \les_m \cE .
\eeq


Applying Lemma \ref{lem:T2} for $\cT_2$, Proposition \ref{prop:semiA} for $e^{\cL s} (\AA,0)$
, 
and using $\wh E $ \eqref{eq:non_E}, 
we obtain 
\[
|| \td \VV_2(s) ||_{\cX^{m+2}} \leq || \cT_2( \wh \UU_1, \wh \S_1)(s)||_{\cX^{m+2}} 
+ || e^{\cL (s-\sin)} (\AA,0) ||_{\cX^{m+2}}
\les_m \e^{-\eta_s (s-\sin)} \hat E +  e^{-\lam (s-\sin)}
|| \AA||_{\cX^{m+2}} .
\]
Recall $\hat E < \d_Y$ from \eqref{eq:non_E}, the parameters $\d_Y = \d^{2/3} , \d_A = \d^{5/9}$ with $\d \ll 1$ from \eqref{eq:para_fix}, $\eta_s < \lam$ from \eqref{eq:decay_para}, and  $|| \AA ||_{\cX^{m+2}} < \d_A$ from the assumption in Theorem \ref{thm:non}. From the above estimate, we further obtain 
\beq\label{eq:non_est_V2}
|| \td \VV_2(s) ||_{\cX^{m+2}} \les_m \d^{2/3} e^{-\eta_s(s-\sin)} 
+ \d^{5/9} e^{-\lam(s -\sin)}  
\les_m \d^{5/9} e^{- \eta_s (s -\sin)}   .
\eeq

Next, we prove $|| (\td \UU_1, \td \S_1)||_{Y_1} < 2 \d$. In view of $Y_1$ \eqref{norm:fix}, we impose the following bootstrap assumption
\beq\label{eq:non_boot}
E_{m+1}(s) <2  e^{-\lam_1 (s-\sin)} \d. 
\eeq
Since $E_{m+1}(\sin) < \d$ due to the assumption \eqref{eq:thm_ass}, by continuity of $\td \VV_1$, \eqref{eq:non_boot} holds at least for a short time.

Next, we verify the assumptions in Lemma \ref{lem:supp}. Under \eqref{eq:non_boot}, using \eqref{eq:non_est_V2}, for any $\d < \d_0$ with $\d_0 \ll_m 1$ sufficiently small, we obtain $ || \td \VV_1 ||_{\cX^m} + || \td \VV_2 ||_{\cX^m} < 1$. 
Recall the definition of the full solution $( \UU, \S)$ \eqref{eq:solu_full}. 
Since $16 R_4 \leq R_{c,0}^{1/2}< R_c^{1/2} $ from the assumption of Theorem \ref{thm:non}, the formula of $\td \VV_2(\sin ) = (\AA, 0) - \td \VV_{2,u}(\sin) \chi(y / (8 R_4))$ \eqref{eq:V2_form2a}, \eqref{eq:V2_init}, and the assumption on the initial data \eqref{eq:thm_ass}, we obtain 
\beq\label{eq:boot_supp2}
\bal
 \supp( \UU ) \cup \supp( \S - \S_{\infty}) 
 & \subset  \supp( \td \UU_1 + \AA +\bar \UU ) \cup \supp( \td \S_1 + \bar \S - \S_{\infty}) 
 \cup B(0, 16 R_4) \\
 &\subset B(0,  \max( R_c^{1/2}, 16 R_4) ) = B(0, R_c^{1/2}) .
 \eal
\eeq
Thus, under the bootstrap assumption \eqref{eq:non_boot}, using Lemma \ref{lem:supp}, we deduce 
\beq\label{eq:prop_supp2}
  e^{-s} S(s) < \bar C_m R_c^{1/2} < R_c / 4.
\eeq

Performing $\cW^{m+1} $ estimates on \eqref{eq:non_fix} for $\td \VV_1$, we get 
\bseq\label{eq:non_Y1}
\beq
   \f{1}{2} \f{d}{d s}   || (\td \UU_1, \td  \S_1) ||_{\cW^{m+1}}^2 
  = I_{\cL} + I_{\cN}  , \\
 \eeq
where the linear part $I_{\cL}$ and nonlinear part $I_{\cN}$ are given by 
\begin{align}
 I_{\cL}  &= \la (\cL - \cK_m) ( \td \UU_1, \td \S_1 ), ( \td \UU_1, \td \S_1) \ra_{\cW^{m+1}}  , \\
  I_{\cN}  &= \la \cN_U(\td \VV), \td \UU_1 \ra_{\cW^{m+1}}
 + \la \cN_{\S}(\td \VV), \td \S_1 \ra_{\cW^{m+1}} . 
 \end{align}
\eseq

For the linear part, using \eqref{eq:coer_Wk} and the notation \eqref{eq:non_Ec}, we obtain 
\beq\label{eq:non_lin}
\bal
 I_{\cL} \leq & - \lam_1 || (\td \UU_1, \td \S_1) ||_{\cW^{m+1}}^2
 = -\lam_1 E_{m+1}^2. 
 \eal
\eeq


\vspace{0.1in}
\paragraph{\bf{Estimate of nonlinear terms.}}
We treat the nonlinear terms perturbatively. Our goal is to obtain
\beq\label{eq:non_est}
 |I_{\cN}| \les_m \cE^2 E_{m+1} +   e^{-s/2} \d^2 E_{m+1}.
\eeq

To estimate $\cN_{\al}$ with $\al = U, \S$, we focus on the terms containing most derivatives. Using Lemma \ref{lem:non_main} and Cauchy-Schwarz inequality, and $\cE$ \eqref{eq:non_err} to absorb the error terms, we get 
\[
\bal
  I_{\cN, U} &=  \int \vp_g \vp_{2m+2}^2 \D^{m+1} \cN_{\UU}( \td \VV)  \cdot \D^{m+1} \td \UU_1  \\
 &  = - \int \vp_g \vp_{2m+2}^2 \B( \td  \UU  \cdot 
  \na \D^{m+ 1}  \td \UU   
  + \al \td  \S  \na \D^{m+1} \td \S  \B) 
  \D^{m+ 1}  \td \UU_{1}   
+ O_m( \cE^2  E_{m+1} ).
 \eal
\]

Recall the weight $\vp_{k} = \vp_1^k$ and its estimates: $ \vp_1 \asymp \la y \ra$ and $\vp_1 \gtr 1$ from Lemma \ref{cor:wg}. Using Lemma \ref{lem:prod} with $(f, g, i, j,m, n) \rsarrow (F, G,  2m + 3,0, m+2, m + 1)$, we get 
\[
 || \vp_{2m+2} \vp_g^{1/2} G  \na \D^{m+1}  F ||_{L^2} \les_m 
 || \la y \ra^{2m+2} \vp_g^{1/2} G  \na \D^{m+1}  F ||_{L^2} 
\les_m || G ||_{\cX^{m+1}} ||  F ||_{\cX^{m+2}} \les_m \cE^2,
\]
for $G =\td \UU$ or $\td \S$, and  $ F = \td  \UU_2, \td \S_2 $ \eqref{eq:non_err}. Thus, combining the above two estimates, we further estimate $I_{\cN, U}$ as follows 
\[
\bal
I_{\cN, U} & = 
- \int \vp_g \vp_{2m+2}^2  \B(  \td \UU  \cdot 
 \na  \D^{m+ 1} \td \UU_1   
  + \al  \td \S  \na \D^{m+1} \td \S_1   \B)  \D^{m+1} \UU_1 + O_m( \cE^2 E_{m+1}),
\eal
\]

Next, we estimate the transport term $\td \UU  \cdot \na \D^{m+ 1} \td  \UU_1 $. using Lemma \ref{cor:wg} for $\vp_k$, definition of $\vp_g$ \eqref{norm:Xk}, and \eqref{eq:non_err}, we obtain
\beq\label{eq:non_IBP}
\bal
\quad  | \na (\vp_k^2 \vp_g) | & \les_k \vp_k^2 \vp_g \la y \ra^{-1}, \\
 | \na (\vp_k^2 \vp_g F) |
& \les_k \vp_k^2 \vp_g ( | F \la y \ra^{-1}| + |\na F| )
\les_k \vp_k^2 \vp_g \cE, 
\eal
\eeq
for $F \in \{  \td \UU,\td \S \}$. Using the identity  $\td  \UU \cdot \na F \cdot F =\f{1}{2} \td \UU \cdot \na |F|^2$ and integration by parts, we get 
\[
 \B| \int \vp_g\vp_{2m+2}^2 \td \UU \cdot \na \D^{m+1} \td \UU_1  \D^{m+1} \td \UU_1 \B| 
 \les_m \cE \int \vp_{2m+2}^2 \vp_g |\D^{m+1} \td \UU_1|^2 \les \cE^2 E_{m+1}. 
\]

For the inner product $\la \cdot ,\cdot \ra_{\cX^m}$ involving less derivatives in $\la \cdot , \cdot \ra_{\cW^{m+1}}$ (see \eqref{norm:Wk}), using Lemma \ref{lem:prod} and \eqref{eq:non_err}, we get
\[
 || \cN_U( \td \VV ) ||_{\cX^m} \les_m || \td \VV ||_{\cX^{m+1}}^2 
 \les_m  \cE^2.
\]

Recall the inner products $\cW^{m+1}$ \eqref{norm:Wk}, $\cX^{m+1}$ \eqref{norm:Xk}. Combining the above estimates, we establish
\beq\label{eq:non_Nu}
\bal
  \la \cN_U(\td \VV), \td \UU_1 \ra_{\cW^{m+1}} 
 =& \e_{m+1} \mu_{m+1} I_{\cN, U}
+ O_m (  || \cN_U( \td \VV ) ||_{\cX^m} || \td \UU_1 ||_{\cX^m} ) \\
 = & - \e_{m+1} \mu_{m+1}  \int \vp_g \vp_{2m+2}^2  \B( 
   \al \td \S  \na \D^{m+1} \td \S_1  \B) \D^{m+1}  \td \UU_1     + O_m( \cE^2 E_{m+1}), 
\eal
\eeq
where $\e_{m+1}$ comes from the definition of $\cX^{m+1}$ \eqref{norm:Xk} and $\mu_{m+1}$ from that of $\cW^{m+1}$ \eqref{norm:Wk}. 

Next, we estimate $\cN_{\S}$ \eqref{eq:non}. For $\cE_d$, under the bootstrap assumption \eqref{eq:non_boot}, \eqref{eq:prop_supp2} implies that the assumption in Lemma \ref{lem:err} is satisfied. Applying Lemma \ref{lem:err} and $R_c > \d^{-4}$ \eqref{eq:para_fix}, we obtain 
\[
|| \cE_d ||_{\cX^{m+1}} \les_m e^{-s/2} R_c^{-1/2} || (\bar \UU , \bar \S) + \td \VV ||_{\cX^{m+1}}^2
\les_m (\cE+1)^2 e^{-s/2} \d^2
\les_m \cE^2 + e^{-s/2} \d^2,
\]
 where the last inequality follows from $e^{-s/2} $ for $s\geq \sin$ and that $\sin$ \eqref{eq:s_in} is an absolute constant. It follows 
\beq\label{eq:non_low}
\la \cE_d, \td \S_1 \ra_{\cW^{m+1}} \les_m ( e^{-s/2} \d^2 + \cE^2) E_{m+1}.
\eeq

For $\cN_{\S} - \cE_d$ (see the formula in \eqref{eq:non}), we estimate it using derivations similar to those of $\cN_{U}$. We get 
\beq\label{eq:non_Ns}
\bal
\la \cN_{\S}( \td \VV ) - \cE_d, \td \S_1 \ra_{\cW^{m+1}} 
& = -\e_{m+1} \mu_{m+1} 
 \int \vp_g \vp_{2m+2}^2    \al  \td \S \div( \D^{m+1} \td \UU_1)  \D^{m+1} \td \S_1   + O_m( \cE^2 E_{m+1}) .
 \eal
 \eeq

Next, we combine the estimates \eqref{eq:non_low}, \eqref{eq:non_Nu}, \eqref{eq:non_Ns}, and  estimate the remaining nonlinear terms in \eqref{eq:non_Nu}, \eqref{eq:non_Ns} together. 
Using the identities
\[
\bal
  \al \td \S \div(\D^{m+1} \td \UU_1) \D^{m+1} \td \S_1 
 + \al \td \S \na \D^{m+1} \td  \S_1 \cdot \D^{m+1} \td \UU_1
  & = \al \td \S \na \cdot ( \D^{m+1} \td \UU_1 \cdot   \D^{m+1} \td \S_1  ),
 \eal
\]
integration by parts, and the estimates \eqref{eq:non_IBP}, we establish 
\[
\bal
  |\la \cN_U, \td \UU_1 \ra_{\cW^{m+1}}
 +  \la \cN_{\S}, \td \S_1 & \ra_{\cW^{m+1}}  
  \les_m  \int  |\na( \vp_{2m+2}^2 \vp_g   \S_{tot}) |  \cdot |\D^{m+1} \td \UU_1 \D^{m+1} \td \S_1|
 + ( \cE^2 +  e^{- s/2} \d^2)  E_{m+1} \\
  & \les_m  \cE  \int \vp_{2m+2}^2 \vp_g |\D^{m+1} \td \UU_1 \D^{m+1} \td \S_1|   + \cE^2  E_{m+1}  \les_m ( \cE^2 +  e^{ -s/2} \d^2) E_{m+1},
 \eal
\]
and prove \eqref{eq:non_est}.

\vspace{0.1in}
\paragraph{\bf{Conclusion of the proof.}}
Combining \eqref{eq:non_lin} and \eqref{eq:non_est}, we establish the energy estimates 
\[
 \f{1}{2}\f{d}{d s} E_{m+1}^2 \leq -\lam_1 E_{m+1}^2 +C_m \cE^2 E_{m+1}
  + C_m  e^{-s/2} \d^2 E_{m+1}.
 \]

Using $\cE^2 \les E_{m+1}^2 + || \td \VV_2 ||_{\cX^{m+2}}^2$ from \eqref{eq:non_E} and the above estimate, we obtain
 \[
   \f{d}{d s} E_{m+1} \leq -\lam_1 E_{m+1} + C_m( E_{m+1}^2 + || \td \VV_2 ||_{\cX^{m+2}}^2 + e^{-s/2} \d^2 ) .
 \]
Consider the quantity $E_{\lam}(s) = E_{m+1}(s) e^{\lam_1 (s-\sin)}$. Since $\d < \d_0 <1$, $\eta_s < \lam <\f{1}{2}$ (see \eqref{eq:decay_para} and the bound of $\lam$ in Theorem \ref{thm:coer_est}) and $\sin$ \eqref{eq:s_in} is absolute, using the estimate of $\td \VV_2$ \eqref{eq:non_est_V2} and the above estimate, we obtain
\[
   \f{d}{d s }  E_{\lam}(s) \leq C_m e^{-\lam_1 (s-\sin)} E_{\lam}(s)^2 
    + C_m e^{- ( 2 \eta_s -\lam_1) (s-\sin)  } \d^{10/9}.
\]

Using the assumption in Theorem \ref{thm:non}, we have $E_{\lam}(\sin) = E_{m+1}(\sin) < \d $. For $\d < \d_0$, since $2 \d_s - \lam_1 > 0 $ (due to \eqref{eq:decay_para}), using the bootstrap assumption \eqref{eq:non_boot}, we deduce 
\[
E_{\lam}(s) \leq  \d + C_m \B( \int_{\sin}^{\infty} e^{-\lam_1 (s-\sin)} 4 \d^2 + 
  e^{- (2 \eta_s - \lam_1) (s-\sin)} \d^{10/9} d s \B)
  < \d + C_m ( \d^2 + \d^{10/9}) < ( 1 + C_m \d_0^{1/10} ) \d,
\]
Choosing $\d_0 \ll_m 1$ small enough so that $1 + C_m \d_0^{1/10}<2$, we deduce from the above estimate that for any $\d < \d_0$, the bootstrap assumption \eqref{eq:non_boot} is closed. 
We complete the nonlinear estimates and obtain \eqref{eq:non_boot}, \eqref{eq:prop_supp2} for all $s \geq \sin$.


The estimate \eqref{eq:prop_supp2} implies \eqref{eq:prop_supp} in Proposition \ref{prop:onto}. The first estimates in Proposition \ref{prop:onto} follows from \eqref{eq:non_boot} and the definitions of $E_{m+1}$ \eqref{eq:non_Ec} and $Y_2$ \eqref{norm:fix}. The last estimate in  Proposition \ref{prop:onto} follows from Lemma \ref{lem:T2}. Recall $\d_Y = \d^{2/3}$ from \eqref{eq:para_fix}. Since $Y_1$ is stronger than $Y_2$ \eqref{norm:fix}, using 
\[
  || \cT( \wh \UU_1, \wh \S_1) ||_{Y_2} \leq C_m   || \cT( \wh \UU_1, \wh \S_1) ||_{Y_1}
  \leq  C_m 2 \d \leq  ( 2 C_m \d_0^{1/3}) \d_Y,
\]
and by further choosing $\d_0$ with $2 C_m \d_0^{1/3}<1$, we prove the second estimate in Proposition \ref{prop:onto}. 
\end{proof}

\subsection{Proof of Proposition \ref{prop:contra}}\label{sec:contra}

Since the proof essentially follows the proof of \cite[Proposition 4.4 in Section 4.4.3]{chen2024Euler}, we only sketch some estimates, e.g. the proof of \eqref{eq:non_est_dif}.

By assumption of Proposition \ref{prop:contra}, let  $\wh \UU_{1, \ell},\wh \S_{1, \ell} \in Y_2$ such that $\wh E_{\ell} = || \wh \UU_{1, \ell}, \wh \S_{1, \ell} ||_{Y_2} < \d_Y$ for $\ell = a, b$. We use the subscript $a, b$ to denote two different solutions, and adopt the notations 
introduced in \eqref{eq:non_E}, e.g. $\cE_{\ell} , E_{m+1,\ell}$ for the energies. 
We construct $\td \VV_{2, \ell}$ using \eqref{eq:non_fix_V2}, $\td \VV_{1,\ell}$ using \eqref{eq:non_fix_V1}, and further denote $\td \VV_{\ell}$
\beq\label{eq:solu_diff}
 \td \VV_{2, \ell} = \cT_2( \wh \UU_{1,\ell}, \wh \S_{1,\ell} ) 
 + e^{ \cL( s -\sin)} (\AA, 0),
 \quad \td \VV_{1,\ell} = \cT(\wh \UU_{1,\ell}, \wh \S_{1,\ell} ) ,
 \quad \td \VV_{\ell} =  \td \VV_{1, \ell} +  \td \VV_{2, \ell}
\eeq
for $\ell = a, b$.  Additionally, we denote the difference of two solutions by a $\D$-sub-index:
\beq\label{eq:nota_dif}
\bga
\wh \VV_{1,\D} = ( \wh \UU_{1, \D}, \wh \S_{1,\D})  = ( \wh \UU_{1, a}, \wh \S_{1,a}) - ( \wh \UU_{1,b}, \wh \S_{1,b}), \quad \td \VV_{i, \D} = \td \VV_{i, a } -\td  \VV_{i, b}, i = 1,2 ,  \\
   \cN_{\al, \D} = \cN_{\al}( \td \VV_{ a} ) - \cN_{\al}(\td \VV_{ b}) ,\quad  \mathrm{ for \ } \al = U, \S , \\
   E_{m, \D} =  || (\td \UU_{1,\D}, \td \S_{1,\D})  ||_{\cX^m}  , \quad  \cE_{\D} = E_{m, \D} + || \td \VV_{2, \D}||_{\cX^{m+2}} .
 \ega
\eeq

By definition, we have 
\beq\label{eq:prop}
\bal
 & || \td \VV_{ \D} ||_{\cX^m} \les_m || \td \VV_{1, \D} ||_{\cX^m}
 + || \td \VV_{2, \D} ||_{\cX^m} \les_m \cE_{\D},
 \eal
\eeq

We prove Proposition \ref{prop:contra} by performing $\cX^m$ energy estimates of \eqref{eq:non_fix} on $\td \VV_{1,\D}$.  Following \eqref{eq:non_Y1}, \eqref{eq:non_lin} and using item (d) in Proposition \ref{prop:compact} for $\cL -\cK_m$, we obtain 
\beq\label{eq:non_Y2}
\bal
   \f{1}{2} \f{d}{d s}  || (\td \UU_{1,\D}, \td \S_{1,\D}) ||_{\cX^{m}}^2  
  & =  \la (\cL - \cK_m) (\td  \UU_{1,\D}, \td \S_{1,\D} ), ( \td \UU_{1,\D}, \td \S_{1,\D}) \ra_{\cX^{m}} 
 + I_{\cN, \D}  \\
 & \leq - \lam E_{m,\D}^2  + I_{M, \D} + I_{\cE, \D},  \\
 \eal
 \eeq
 where $I_{\cN,\D} = I_{M, \D} + I_{\cE, \D}$ are defined as follows 
 \[
   I_{M,\D}  = \la \cN_{U, \D}, \td \UU_{1,\D} \ra_{\cX^{m}}
 + \la \cN_{\S, \D} - \cE_{d, \D}, \td \S_{1,\D} \ra_{\cX^{m}},
 \quad I_{\cE, \D} = \la \cE_{d, \D}  , \td  \S_{1,\D} \ra_{\cX^{m}},
\]
 and $I_{M,\D}$ is the main nonlinear term, and $I_{\cE, \D}$ is the lower order nonlinear term.

We treat the nonlinear terms perturbatively, and our goal is to establish 
\beq\label{eq:non_est_dif}
|I_{M, \D}| \les_m  (\cE_{a} + \cE_b) \cE_{\D} E_{m,\D},
\quad I_{\cE, \D} \les_m e^{-s/2} \d E_{m, \D}^2 .
\eeq

The estimate of $I_{M, \D}$ in \eqref{eq:non_est_dif} is the same as that of $I_{\cN,\D}$ in \cite[Section 4.4.3]{chen2024Euler} with $\AA = 0$ and follows from standard energy estimates and integration by parts like the proof of \eqref{eq:non_est}, we omit the proof. We note that in the proof of these estimates, we only use the algebraic structure of $\D^m \cN_{\al}$ and do not require symmetry assumption on $(\UU, \S)$.

Recall the definition of $ \UU, \S$ from \eqref{eq:solu_full}, and $\cE_d$ from \eqref{eq:err}. Next, we estimate $I_{\cE, \D}$, which is not covered by \cite{chen2024Euler}. 
From Proposition \ref{prop:onto}, the solution $ \UU_{1,\ell}, \ell = a, b$ satisfies the support estimate \eqref{eq:prop_supp} for all time. Thus, $ \UU_{1, \D}, \UU_{1, \ell}$ satisfy the support assumption in Lemma \ref{lem:err}. Since $\cE_{d}( \UU, \S)$ \eqref{eq:err} is bilinear in $U_1, \S$, using Lemma \ref{lem:err}, we obtain 
\[
\bal
||\cE_{d, \D} ||_{\cX^m} & =  || \cE_{d}( \UU_a,  \S_a) 
 - \cE_{d}( \UU_b,  \S_b) ||_{\cX^m}
 =  || \cE_{d}( \UU_a -  \UU_b,  \S_a) 
 + \cE_{d}(  \UU_b,  \S_a -  \S_b) ||_{\cX^m} \\
& \les_m   ( \f{e^{-s}}{ R_c} )^{1/2} || (  \UU_{\D} ,  \S_{\D}) ||_{\cX^m} 
(  || (  \UU_a,   \S_a)||_{\cX^m} 
+ || (  \UU_b,   \S_b)||_{\cX^m}    ).
\eal
\]

 Recall $\cE_{\ell}$ from \eqref{eq:non_E} and its estimates \eqref{eq:non_err}. Using the decomposition \eqref{eq:solu_full}, we yield
\[
 (  \UU_{\D} ,  \S_{\D}) = \td \VV_{1,\D} + \td \VV_{2,\D},
 \quad  || ( \UU_{\ell} , \S_{\ell}) ||_{\cX^m} 
 \les_m || (\bar  \UU,  \bar \S) ||_{\cX^m} 
 + || \td \VV_{1,\ell} + \td \VV_{2,\ell} ||_{\cX^m} \les_m \cE_{\ell} + 1 .
 \]

 Since we require $\d^{-2} \ll R_c$ from \eqref{eq:para_fix}, we obtain 
 \[
||\cE_{d, \D} ||_{\cX^m} \les_m 
e^{-s/2} \d \cdot E_{m, \D} (\cE_a + \cE_b + 1).
 \]

Applying Lemma \ref{lem:T2} and the estimate of $\td \VV_{1, \ell}
= \cT(\wh \UU_{1,\ell}, \wh \S_{1,\ell} )$ \eqref{eq:prop_onto_est} in Proposition \ref{prop:onto} to $E_{m+1, \ell}, \ell= a, b$, we obtain 
\beq\label{eq:non_est_cE}
\bal
\cE_{\ell} &= E_{m+1, \ell} + || \cT_2( (\wh \UU_{1, \ell}, \wh \S_{1, \ell} ) ||_{\cX^{m+2}}
\leq 2 \d e^{- \lam_1 (s-\sin)} + C_m e^{-\eta_s (s-\sin)} || (\wh \UU_{1, \ell}, \wh \S_{1, \ell} ) ||_{Y_2}\\
&\leq 2 \d e^{- \lam_1 (s-\sin)} + C_m e^{-\eta_s (s-\sin)}  \d_Y 
\leq C_m e^{-\eta_s (s-\sin) } \d_Y  \les_m 1, 
\eal
\eeq
where we have used $\eta_s < \lam_1$ \eqref{eq:decay_para} and $\d_Y > \d$ \eqref{eq:para_fix}.
Applying the above estimates to $I_{\cE, \D}$ and using $s \geq \sin$ with $\sin$ \eqref{eq:s_in} being an absolute constant, we deduce
\beq\label{eq:non_est_err}
  | I_{\cE, \D}| \les_m ||\cE_{d, \D} ||_{\cX^m} || \td \S_{1, \D} ||_{\cX^m} 
\les_m e^{-s/2} \d \cdot  E^2_{m, \D} (\cE_a + \cE_b + 1) 
\les_m 
\d \cdot  E^2_{m, \D}.
\eeq

\vspace{0.1in}
\paragraph{\bf{Conclusion of the proof.}}
Combining the estimates \eqref{eq:non_Y2}, \eqref{eq:non_est_dif},  \eqref{eq:non_est_err}
\eqref{eq:decay_para}, we establish 
\[
 \f{1}{2}\f{d}{d s} E_{m, \D}^2 \leq - \lam E_{m, \D}^2 + 
 C_m (\cE_a + \cE_b) \cE_{\D} E_{m, \D} 
 + C_m \d E_{m, \D}^2 .
\]
Since $\lam_1 < \lam$, taking $\d_0$ small enough so that $C_m \d_0 < \lam - \lam_1$ and replacing $\cE_{\D}$ \eqref{eq:nota_dif}, we further obtain
\beq\label{eq:non_EE}
 \quad  \f{d}{d s} E_{m, \D} \leq - \lam_1 E_{m, \D} + 
 C_m (\cE_a + \cE_b)( E_{m, \D} + || \td \VV_{2,\D} ||_{\cX^{m+2}} ) ).
\eeq

Denote 
\[
P(s) = E_{m, \D}(s) e^{\lam_1 (s- \sin)} , \quad  Q = || (\wh \UU_{1, \D}, \wh \S_{1, \D}) ||_{ Y_2}.
\]

Since $ \td  \VV_{2, \ell}= \cT_2( \wh \UU_{1,\ell}, \wh \S_{1, \ell}) + e^{\cL (s-\sin)} (\AA, 0)$ (see \eqref{eq:solu_diff}), the initial data $\AA$ is the same for $\ell = a, b$,  and $\cT_2$ is linear, using Lemma \ref{lem:T2}, we obtain 
\[
 || \td \VV_{2, \D}||_{\cX^{m+2}} 
=|| \cT_2( \wh \VV_{1,\D}) ||_{\cX^{m+2}}
 \les_m   e^{-\eta_s (s-\sin)} || (\wh \UU_{1, \D}, \wh \S_{1, \D}) ||_{ Y_2}
 \les_m  e^{- \eta_s (s-\sin)} Q.
\]

Combining the above estimates and applying \eqref{eq:non_est_cE} for $\cE_{\ell}$ and $ \eta_s > 2 \eta_s - \lam_1 > 0$ due to \eqref{eq:decay_para}, we derive the ODEs for $P(s)$
\[
\bal
\f{d}{d s} P(s) 
&  \leq e^{\lam_1 (s-\sin)} C_m(\cE_a + \cE_b)  
(E_{m, \D} + || \td \VV_{2,\D} ||_{\cX^{m+2}} )\\
 & \leq C_m \d_Y( e^{-\eta_s (s - \sin)}P(s) 
+ e^{ - (2 \eta_s - \lam_1) (s - \sin)}  Q) 
 \leq C_m \d_Y e^{ - (2 \eta_s - \lam_1) (s - \sin)}( P(s) + Q).
\eal
\]
Since the initial data are the same for $(\td \UU_{1,\ell}, \td \S_{1,\ell} )$. we have $P(\sin) = 0$.  Choosing $\d_0$ small enough, using $\d_Y = \d^{2/3}, \d < \d_0 $, and solving the ODEs for $P(s) +Q$, we obtain 
\[
P(s) + Q \leq Q \exp( \int_{\sin}^{s} C_m \d_Y \e^{- (2 \eta_s - \lam_1) ( \tau-\sin) } d \tau )
\leq \exp( C_m \d_Y) Q < \f{3}{2} Q,
\]
for any $s \geq \sin$. Recall the definition of $Y_2$ from \eqref{norm:fix} and $E_{m, \D}$ from \eqref{eq:nota_dif}. We establish
\[
P(s) < \f{1}{2} Q, \quad \mathrm{ which \ is \ equivalent \ to } \quad || \td \UU_{1,\D}, \td \S_{1,\D} ) ||_{Y_2} < \f{1}{2} 
|| (\wh \UU_{1, \D}, \wh \S_{1, \D}) ||_{ Y_2} .
\]
We conclude the proof of Proposition \ref{prop:contra}.

\appendix

\section{Proof of angular repulsive property}\label{sec:angle_rep }

In this section, we prove \eqref{eq:rep_ag}. Firstly, we recall the properties of autonomous ODE for the profile from \cite{merle2022implosion1}. Denote 
\beq\label{eq:ODE_nota}
W =  - \f{\bar U(\xi)}{\xi}, \quad
S =  \al \f{ \bar \S(\xi) }{ \xi}, 
\quad x(\xi) = \log \xi, \quad l = \f{1}{\al} =\f{2}{\g-1},
\quad W_e = \f{l(r-1) }{d} ,
\eeq
with $d=2$.	
{\em Here, the quantities $(W, S, l)$ are the same as $(w, \s, l)$ in \cite{merle2022implosion1}.} The notation in~\eqref{eq:ODE_nota} is only used in this appendix, and it does not affect the rest of the paper.


From \cite[Section 2]{merle2022implosion1}, we have the following ODEs for $(W, S)$
\bseq\label{eq:ODE}
\beq\label{eq:ODEa}
  \f{d W}{d x} = - \f{\D_1}{\D} , \quad  \f{d S}{d x} = - \f{\D_2}{\D}, \\
\eeq
where $\D, \D_1, \D_2$ are defined as 
\begin{align}
  \D &= (1-W)^2 - S^2, \\
  \D_1 &= (W- W_1(S))( W- W_2(S) )(W- W_3(S)), \label{eq:ODEc}  \\
  \D_2 &= \f{S(l+d-1)}{l} (W- W_2^{-}(S)( W - W_2^+(S) ), 
\end{align}
\eseq
where $ \{ W_i(S) \}_{i=1}^3$ are the roots of the cubic polynomial $\D_1$ and $W_2^{\pm}$ are the roots for the quadratic polynomial $\D_2$.  From \cite[Lemma 2.3]{merle2022implosion1} for the roots of $W_i$ of $\D_1$, we have the following properties 
\beq\label{eq:ODE_prop}
\bal
 W_1(S) \leq 0 < W_e < W_2(S) \leq 1 < r \leq W_3(S) , \quad \forall S > 0.
\eal
\eeq

Using the relation \eqref{eq:ODE_nota} and the ODE \eqref{eq:ODEa}, we have 
\[
 - \pa_{\xi} \bar U + \f{\bar U}{\xi}
= - \xi \f{d}{d \xi} ( \f{\bar U}{\xi} ) = \f{d}{d x}W = -\f{\D_1}{\D} .
\]
 To obtain \eqref{eq:rep_ag}, using \eqref{eq:rep1}, we only need to show 
 \beq\label{eq:dwdx}
  0 \leq  \f{\bar U}{\xi}  - \pa_{\xi} \bar U , \mathrm{\quad   or \ equivalently \quad }  0 \leq -\f{\D_1}{\D}
\eeq
for $\xi \in [0, \xi_s]$. Denote
\footnote{
We use $(Q_S, Q_W)$ to denote the coordinates of a point $Q$ in the system of $(S, W)$.
} 
by $P_2 = (P_{2, S}, P_{2, W} )$ and $P_3 =(P_{3, S}, P_{3, W} )$ the sonic points, where $\D = \D_1 = \D_2 = 0$. 
 Note that the sonic point $S = P_{2, S}$ corresponds to $\xi = \xi_s$ with $\xi_s$ appearing in Lemma \ref{lem:profile}. From \cite[Lemma 2.6]{merle2022implosion1}, we have
\beq\label{eq:rel_loc}
P_{3, S} \leq P_{2, S}. 
\eeq

We have the following property of $W(S)$ for $S > P_{2, S}$ proved in  \cite[Lemma 3.1]{merle2022implosion1}.
\begin{lemma}\label{lem:inner}
Assuming \eqref{r-range}. We have $W(S) \in C^{\infty}( P_{2, S}, \infty)$ and 
\[
\bga
  W_2^-(S) < W(S) < W_2(S), \quad \forall S > P_{2, S},  \\
 W(S) = W_e + \f{W_e(W_e-1)(W_e - r)}{ d + 2} \f{1}{S^2} + O(S^{-4}), 
\quad \mathrm{as \ } S \to \infty.
\ega
\]

For $S_0> P_{2, S}$ large enough,  the curve $(S, W(S))$ with $S>S_0$ corresponds to the 
spherically symmetric profile $( \bar \S(|y|), \bar \UU(y))$ (see Theorem \ref{thm:merle:implosion}) defined on $|y| \in [0, y_0]$ for some $y_0>0$.
\end{lemma}

Now, we are ready to prove \eqref{eq:dwdx}.

\begin{proof}[Proof of \eqref{eq:dwdx}]

Since $\lim_{\xi \to 0} \f{\bar \S(\xi)}{\xi} = \infty, \lim_{ \xi \to \infty} \f{\bar \S}{\xi} =0$, from Lemma \ref{lem:inner}, the solution curve $(S, W(S)), S \geq P_{2,S}$ corresponds to $(S(x), W(x))$ with $x \in (-\infty, x_s]$, where $x_s = \log \xi_s$ (see \eqref{eq:ODE_nota}). To prove \eqref{eq:dwdx}, using \eqref{eq:ODE}, we only need to show 
\beq\label{eq:dwdx_pf1}
\D(S, W(S)) < 0, \quad  \D_1(S, W(S)) > 0, \quad S >  P_{2, S}. 
\eeq

Using Lemma \ref{lem:inner}, we know $ \D < 0$ for $S > S_0$ with $S_0 > P_{2,S}$ large enough. Since the solution is $C^{\infty}$, from \eqref{eq:ODEa}, the solution curve $(S, W(S))$ can only pass the sonic line $\D = 0$ at the triple roots $P_2$ or $P_3$. For $S > P_{2, S}$, from \eqref{eq:rel_loc}, we get $S > P_{2, S} \geq P_{3, S}$, which implies $(S, W(S)) \neq P_2, P_3 $. 
Thus, by continuity, we get extend the validity of the inequality $\D(S, W(S)) < 0$ for $S > S_0$ to $S > P_{2, S}$.  

For the second inequality in \eqref{eq:dwdx_pf1}, from the order of $W_i(S)$ in \eqref{eq:ODE_prop} and \eqref{eq:ODEc},  we only need to show 
\beq\label{eq:dwdx_pf2}
W_1(S) < W(S) < W_2(S), \quad  \mathrm{for \ } S >  P_{2, S}.
\eeq

The inequality $W(S) < W_2(S)$ follows from Lemma \ref{lem:inner}. Next, we show $W(S) > W_e>0$.  Using the asymptotics in Lemma \ref{lem:inner}, 
$ W_e (W_e - 1) (W_e - r) > 0$ from \eqref{eq:ODE_prop}, and the continuity of $W(S)$, there exists $\e >0$ small and $c, c_1$ large enough with $P_{2, S} < c < c_1$  such that 
\beq\label{eq:ODE_contra0}
W(S) \geq W_e, \  \forall S \geq c,  \quad W(c) = W_e + 2 \e, \quad W(c_1) = W_e + \e. 
\eeq

If there exists $c_2 \in ( P_{2, S}, c)$ such that $W(c_2) = W_e + \e $, we assume that $c_2$ is the largest one in $[c_2, c]$ with such a property. Then we get 
\beq\label{eq:ODE_contra1}
W(S) \geq W_e > 0, \quad S \in [c_2, \infty).
\eeq

Using Mean value theorem, we get $ \f{d W}{d S}(c_3) = 0$ for some $c_3 \in  [c_2, c_1]$. From \eqref{eq:ODE}, we yield 
\[
\D_1(c_3, W(c_3)) = \D_2 \f{ d W}{d S} \B|_{S=c_3} = 0, 
\] 
which implies $W(c_3) = W_i(c_3)$ for some $i \in \{1,2,3\}$ due to \eqref{eq:ODEc}. 
From Lemma \ref{lem:inner} and the order of $W_i$ \eqref{eq:ODE_prop}, we get $W(c_3)<W_2(c_3) \leq W_3(c_3)$. Thus, we must obtain $W(c_3) = W_1(c_3) \leq 0$, which along with $c_3 \geq c_2$ contradicts \eqref{eq:ODE_contra1}. Therefore, we yield $W(S) > W_e + \e$ for $S \in (P_{2,S}, c]$,
which along with \eqref{eq:ODE_contra0} implies $W(S) \geq W_e > 0 \geq W_1(S ) $ for $S > P_{2, S}$. We prove \eqref{eq:dwdx_pf2}, which  further implies  \eqref{eq:dwdx_pf1}.
\end{proof}

\section{Functional inequalities}

In this appendix, we first present a few estimates and embedding inequalities in Appendix \ref{app:embed} , which are similar to those in \cite[Appendix C]{chen2024Euler}. 
Then we prove Lemmas \ref{lem:non_IBP}, \ref{lem:norm_equiv_coe} 
in Appendix \ref{sec:norm_dif}. In Appendix \ref{sec:add_pf}, for the sake of completeness, we attach the proof of Proposition \ref{prop:far_decay}.

\subsection{Main terms and embedding inequalities}\label{app:embed}

We use the following Lemma to compute the main terms in the $H^{2m}$ estimates, which is similar to \cite[Lemma C.1]{chen2024Euler} and \cite[Lemma A.4]{buckmaster2022smooth}, but we do not assume radial symmetry for $g, G$.

\begin{lemma}\label{lem:leib}
Let $f$ be a radially symmetric scalar function over $\R^d$ and $\FF(y) = F(|y|) \ee_R = (F_1, .., F_d)$ be a radially symmetric vector field over $\R^d$. Denote $\xi = |x|$. For any function $g$ and vector $G$, which are not necessary radially symmetric, we have
\bseq
\begin{align}
   | \D^m( \FF \cdot \na g )- \FF \cdot \na \D^m g - 2 m \f{ F}{ \xi } \cdot \D^m g
   - 2 & m  \xi \pa_\xi  (\f{F}{\xi}) \pa_{ \xi \xi} \D^{m-1} g | \notag \\
  &  \les_m  \sum_{ 1 \leq j \leq 2 m - 1 } | \na^{2m +1-j} \FF| \cdot |\na^{j} g | , 
 \label{eq:lei_Fg} \\
  |\D^m (f \na g ) - f \na \D^m g - 2 m \pa_\xi f \pa_\xi \na \D^{m-1} g  | 
 & \les_m  \sum_{ 1 \leq j \leq 2 m - 1 } | \na^{2m +1-j} f| \cdot |\na^{j} g| , 
 \label{eq:lei_fg} \\
  | \D^m( f \div(G ) - f \div( \D^m G) - 2 m \pa_\xi f  \pa_\xi \div (\D^{m-1} G)| 
 & \les_m \sum_{ 1 \leq j \leq 2 m - 1 } | \na^{2m +1- j} f| \cdot |\na^{j} G|  .
 \label{eq:lei_fG}
 \end{align}
 \eseq
\end{lemma}

\begin{proof}
Below, we use the standard summation convention on repeated indices. Using Leibniz rule, we obtain 
\[
 |\D^m (\FF \cdot \na g)-  \FF \cdot \na \D^m g  - 2 m \pa_i F_j \pa_j \pa_i \D^{m-1} g|
 \les_m  \sum_{1 \leq j \leq 2m -1} |\na^{2m+1 - j} \FF | \cdot |\na^j g|.
\]
We further simplify $2 m \pa_i F_j \pa_j \pa_i \D^{m-1} g$. Since $\FF$ is a radially symmetric vector, we obtain 
\[
\pa_i F_j = \pa_i ( \f{F(\xi)}{\xi} x_j ) 
= \f{x_i x_j}{\xi} \pa_\xi ( \f{F}{ \xi } ) + \d_{ij} \f{F}{ \xi } .
\]

Using the identity 
\[
  x_i x_j \pa_i \pa_j q
 = ( (x_i \pa_i) (x_j \pa_j) - x_i \pa_i  ) q
 = ( (\xi \pa_{\xi})^2 - \xi \pa_{\xi} ) q
 = \xi^2 \pa_\xi^2 q ,   \quad \d_{ij} \pa_i \pa_j q = \D q,
\]
for any function $q$, we prove \eqref{eq:lei_Fg}. 

For \eqref{eq:lei_fg}, using Leibniz rule, we obtain 
\[
 |\D^m ( f  \na g)-  f \na \D^m g  - 2 m \pa_i f \pa_i \na \D^{m-1} g|
  \les_m  \sum_{1 \leq j \leq 2m -1} |\na^{2m+1 - j} f | \cdot |\na^j g|.
\]

Since $f$ is a radially symmetric function, we get
 \[
 \pa_i f \pa_i \na \D^{m-1} g = \f{x_i}{\xi} \pa_\xi f \pa_i \na \D^{m-1} g  
 = \pa_\xi f \pa_\xi  \na \D^{m-1} g  .
 \]

For \eqref{eq:lei_fG}, we estimate
\[
 |\D^m ( f  \div(G))-  f \div( \D^m G )  - 2 m \pa_i f \pa_i \div( \D^{m-1} G)|
  \les_m  \sum_{1 \leq j \leq 2m -1} |\na^{2m+1 - j} f | \cdot |\na^j g| .
\]

Since $f$ is a radially symmetric function, using the identity
\[
\pa_i f = \f{x_i}{\xi} \pa_\xi f , \quad 
\pa_i f \pa_i \div( \D^{m-1} G)
= \pa_\xi f \pa_\xi \div( \D^{m-1} G),
\]
we complete the proof.
\end{proof}

The following lemmas have been proved in \cite[Lemmas C.2, C.3, C.4]{chen2024Euler}. Note that one does not assume radial symmetry for the function in \cite[Lemmas C.2, C.3, C.4]{chen2024Euler} and the following lemmas.

\begin{lemma}\label{lem:interp_wg}
Let $\d_1 \in (0, 1], \d_2 \in \R$. For integers $n\geq 0$ and sufficiently smooth functions $f$ on $\R^d$, we denote
\beq\label{eq:interp_J}
 \b_j \teq 2 j \d_1 + \d_2, 
 \quad  I_n \teq \int |\na^n f |^2 \la x \ra^{\b_n} d n, 
 \quad J_n \teq I_n + I_0.  
 \eeq
where we let $\la x \ra = (1 + |x|^2)^{1/2}$. 
Then, for $n<m$ and for any $\e>0$, there exists a constant  $C_{\e,n,m} = C(\e,n,m,\d_1,\d_2,d)>0$ such that 
\begin{subequations}
\label{eq:interp_convex}
\begin{equation}
I_n \leq \e I_m + C_{\e,n,m} I_0. 
\label{eq:interp_convex:a}
\end{equation}  
Moreover, for $p < r < q$ we have the following interpolation inequality on $\R^d$:
\begin{equation}
J_r \leq C_{p,q,r} J_p^{\al} J_q^{1-\al}, 
\qquad \mbox{where} \qquad
\al = \tfrac{q-r}{q-p},
\label{eq:interp_convex:b}
\end{equation}
\end{subequations}
for some constant $C_{p,q,r} = C(p,q,r,\d_1,\d_2,d)>0$.
\end{lemma}

\begin{lemma}\label{lem:norm_equiv}
Let $\d_1 \in (0, 1], \d_2 \in \R$, and define $\b_n = 2 n \d_1 + \d_2$. Let $\td \vp_{n}$ be a weight satisfying the pointwise properties $\td \vp_{n}(y) \asymp_n \la y \ra^{\b_{n}}$ and $|\na \td \vp_{n}(y)| \les_n \la y \ra^{\b_{n}-1}$. Then, for any $\e > 0$ and $n\geq 0$, there exists a constant $C_{\e,n} = C(\e,n,\d_1,\d_2,d)>0$ such that\footnote{Throughout the paper we denote by $|\nabla^k f|$ the Euclidean norm of the $k$-tensor $\nabla^k f$, namely,  $|\nabla^k f| = (\sum_{|\alpha|=k} |\partial^\alpha f|^2 )^{1/2}$.}
\[
 \int |\na^{2n} f|^2 \td \vp_{2n}
 \leq (1 + \e)\int  | \D^n f |^2  \td \vp_{2n} + 
 C_{\e, n} \int |f|^2 \la y \ra^{\b_0},
\]
for any function $f$ on $\R^d$ which is sufficiently smooth and has suitable decay at infinity.
\end{lemma}

\begin{lemma}\label{lem:prod}
Let $\kp_1 = \f{1}{4}$ be the parameter defined in \eqref{norm:Xk}. Let $m \geq d/2$.
For $f : \R^d \to \R$ which lies in $\cX^m$ and $0 \leq i \leq 2m - d $, we have
\beq\label{eq:Xm_Linf}
 |\na^i f (y)| \les_m  \| f \|_{\cX^m}  \la y \ra^{-i + \frac{\kp_1}{2}},
\eeq
pointwise for $y \in \R^d$.
Moreover, for any $i\leq 2m$, $j \leq 2n$, with $ i \leq 2m - d$ or $ j \leq 2n -d$, and for any $\b \geq \kp_1/2$, we have 
\beq\label{eq:Xm_prod}
 \| \na^i f \; \na^{j} g \; \la y \ra^{i+j -\b  } \vp_g^{1/2} \|_{L^2} 
\les_{m, n,\b} \| f \|_{\cX^m} \| g \|_{\cX^n}.
\eeq
\end{lemma}

\subsection{Proof of Lemmas \ref{lem:norm_equiv_coe}, \ref{lem:non_IBP}}\label{sec:norm_dif}

In this section, we prove Lemmas \ref{lem:norm_equiv_coe}, \ref{lem:non_IBP} by applying integration by parts appropriately.

\begin{proof}[Proof of Lemma \ref{lem:norm_equiv_coe} ]

Recall the assumptions of the weights $\psi_i$ and $B$ from Lemma \ref{lem:norm_equiv_coe}. Our goal is to prove 
\[
   \B| \int B(y) (|  \D^{n} f|^2 - |\na^{2i} \D^{n-i} f|^2 ) \psi_{2n}  \B|
\leq \e \int |\D^n f | \psi_{2n} + C(\e, n, \d) \int f^2 \psi_0.
\]
for $0\leq i \leq n$. We focus on the case of $i = n$. Other cases are similar.

We adopt the notations from \eqref{eq:interp_J} and use the standard summation convention on repeated indices. Using integration by parts on $\pa_{i_1}$, we get 
\[
\bal
M_n & \teq \int |\D^n f|^2 B(y)\psi_{2 n} 
=  \sum_{i_1, .. i_n}    \sum_{j_1, j_2,.. j_n} \int  \pa_{i_1}^2
\pa_{i_2}^2.. \pa_{i_n}^2 f
\cdot 
 \pa_{j_1}^2 \pa_{j_2}^2.. \pa_{j_n}^2 f  B \psi_{2 n} \\
& = - \int \sum_{i_1, ..i_n}    \sum_{i_1, ..i_n}
\pa_{i_1}\pa_{i_2}^2.. \pa_{i_n}^2 f 
 \cdot \B( \pa_{i_1} \pa_{j_1}^2 .. \pa_{j_n}^2 f 
\cdot B \psi_{2n} +  \pa_{j_1}^2 .. \pa_{j_n}^2 f  \pa_{i_1} ( B \psi_{2n} ) \B)
\teq I + II.
\eal
\]

The term $II$ (the second term in the bracket) involves a derivative of the weight or the function $B$ and we do not further perform integration by parts for it. For the first term, we further perform integration by parts on $\pa_{j_1}^2$ to obtain 
\[
I = \int \sum_{i_1, ..i_n}  \sum_{j_1, .., j_n}
\pa_{ i_1 j_1} \pa_{i_2}^2.. \pa_{i_n}^2 f  \cdot 
\pa_{ i_1 j_1} \pa_{j_2}^2.. \pa_{j_n}^2 f  \cdot  B \psi_{2n}
+ \pa_{ i_1 } \pa_{i_2}^2.. \pa_{i_n}^2 f  \cdot \pa_{ i_1 j_1} \pa_{j_2}^2.. \pa_{j_n}^2 f  \pa_{j_1} ( B \psi_{2n} ) 
= I_1 + I_2. 
\]

For $I_2$, again, we do not further apply integration by parts. We iteratively integration by 
with respect to $\pa_{i_l}^2, \pa_{j_l}^2$ with $l \in \{ 2,3.., n \} $, to obtain 
\beq\label{eq:norm_equiv_pf1}
M_n = \int |\na^{2n} f|^2 B \psi_{2n}  + \cE, 
\quad 
 \cE \leq C_n \int |\na^{2n} f | \cdot |\na^{2n-1} f| \cdot |\na ( B \psi_{2n} )| .
\eeq

Recall the notations $I_n, J_n$ from \eqref{eq:interp_J} and the following assumptions from Lemma \ref{lem:norm_equiv_coe} 
\[
|B(y) | \les C(b) ,\quad  |\na B| \les C(b) \la x \ra^{-1}, \quad \psi_{2n} \asymp_n \la x \ra^{\b_{2n}},
\quad |\na \psi_{2n}| \les_n \la x \ra^{\b_{2n} - 1} .
\]
 We obtain 
\[
\bal
|\na ( B \psi_{2n} )| 
& \les C(b) ( |\na \psi_{2n}| + \la x \ra^{-1} \psi_{2n}  )
\les C(b)  \la x \ra^{\b_{2n}-1}
\les C(b)  \la x \ra^{ (\b_{2n} - 1)/2}
| \psi_{2n}|^{1/2}  \\
&\les C(b)   \la x \ra^{ \b_{2n-1}/2} | \psi_{2n}|^{1/2} .
\eal
\]
Using Cauchy-Schwarz inequality and Lemma \ref{lem:interp_wg}, for any $\e > 0$, we obtain 
\beq\label{eq:norm_equiv_pf2}
|\cE| \leq C_n I_{2n}^{1/2} I_{2n-1}^{1/2}
\leq  \e I_{2n} + C_{n, \e} I_{2n-1}
\leq  2 \e I_{2n} + C_{n, \e} I_0.
\eeq
Since $\e > 0$ is arbitrary, combining \eqref{eq:norm_equiv_pf1}, \eqref{eq:norm_equiv_pf2}, 
we prove \eqref{eq:norm_dif_ff}. Inequality \eqref{eq:norm_dif_fg} follows from polarization and \eqref{eq:norm_dif_ff}.
\end{proof}

Next, we prove Lemma \ref{lem:non_IBP}. 

\begin{proof}[Proof of Lemma \ref{lem:non_IBP}]

We focus on the proof of \eqref{eq:non_IBP_RR}. The proofs of \eqref{eq:non_IBP_Ria}, \eqref{eq:non_IBP_Rib} are similar.

Denote $\xi =|y|$. Recall the assumptions on $B(y) \in C^1$: $|B(y)| \les C(b) \min(|y|, 1), C(b)|\na B| \les \la y \ra^{-1}$. 
Note that the repeated indices on $j$ appearing below \textit{does not} denote summation over $j$.

Firstly, for any $\phi \in C^1$, we show that
\beq\label{eq:norm_dif_pf1}
\B| \int B(y) \pa_{\xi \xi} f \pa_{jj} g \phi - \int B(y) \pa_\xi \pa_j f \cdot \pa_\xi \pa_j g \phi \B|
\leq C(b,\phi) \int ( |\na f| \cdot |\na^2 g| + |\na g| \cdot |\na^2 f| ) \phi \la y \ra^{-1}.
\eeq

Our idea is to perform integration by parts for $\pa_{\xi \xi}$ and $\pa_{jj}$. If a derivative falls on $B, \phi$, we treat the term as a lower order terms since the total number of derivatives on $f$ or $g$ in such a term is less than $4$. Using integration by parts, we get 
\[
\bal
M & \teq \int  \pa_{\xi \xi} f  \pa_{jj} g  B(y) \phi 
= \int \int_{s \in S^{d-1}}  \pa_{\xi \xi} f \pa_{jj}  g\cdot B( y) \phi \xi^{d-1} d \xi d s
= - \int  \pa_\xi f  \f{1}{\xi^{d-1}} \pa_\xi( \xi^{d-1} B \pa_{jj} g \phi ) d y  \\
& = - \int \pa_\xi f \pa_\xi \pa_{jj} g \cdot  B \phi +  \pa_\xi f \pa_{jj} g  \f{1}{ \xi^{d-1}} \pa_\xi (\xi^{d-1} B \phi ) d y \teq I + II.
\eal
\]
where as usual $S^{d-1}$ is the unit sphere in $\R^d$ and we have $d y = \xi^{d-1} d \xi d s$.

For any function $q$, using $\pa_\xi =\sum_i \f{ y_i}{\xi} \pa_i $, we commute the derivative as follows 
\beq\label{eq:commut}
\pa_i \pa_\xi q - \pa_\xi \pa_i q 
= \sum_l  \pa_i ( \f{y_l}{ \xi } \pa_l q) - \pa_\xi \pa_i q 
= \sum_l  \B( \f{1}{ \xi } \d_{i l}  \pa_l q - \f{y_i y_l}{ \xi^3} \pa_l q  \B)
= \f{1}{ \xi } ( \pa_i q - \f{y_i}{ \xi } \pa_\xi q ). 
\eeq

Therefore, for the first term $I$ in the above, we further decompose 
\[
I = -   \int \pa_\xi f \pa_j \pa_\xi \pa_j g \cdot B \phi -  \int \pa_\xi f  [\pa_\xi, \pa_j] \pa_j g \cdot  B \phi = I_1 + I_2 .
\]

Next, we perform integration by parts for $I_1$ to obtain 
\[
\bal
I_1 & =  \int \pa_j \pa_\xi f \cdot \pa_\xi \pa_j g \cdot B \phi
+ \int \pa_\xi f  \pa_\xi \pa_j g \cdot \pa_j( B \phi)  \\
& = \int \pa_\xi \pa_j f \cdot \pa_{\xi}\pa_j g B \phi
+ [\pa_j, \pa_\xi] f \cdot \pa_\xi \pa_j g \cdot B \phi
+ \int \pa_\xi f  \pa_\xi \pa_j g \cdot \pa_j( B \phi)  \teq I_{1, 1} + I_{1, 2} + I_{1, 3}.
\eal
\]

We show that $I_{1,2}, I_{1, 3}, I_2, II$ are of lower order terms. These terms either contain a commutator, like $I_{1,2}, I_{1, 3}, I_2$ or a derivative acting on the weight or $B$, like $II$. We first estimate the coefficients in these terms. Using the assumptions of $B$ and $\phi$: $|B| \les C(b) \min(1, |y|), |\na B| \les C(b) \la y \ra^{-1}, |\na \phi| \les \phi \la y \ra^{-1}$, we obtain 
\[
\bal
 & |\na (B \phi)| \les |\na B| \phi + |B \na \phi| \les C(b) \cdot \phi \la y \ra^{-1}, \quad 
 \f{1}{ \xi } |B \phi| \les  C(b) \cdot \phi \la y \ra^{-1},  \\
 &  |\f{1}{ \xi^{d-1}} \pa_\xi ( \xi^{d-1} B \phi)|
\les C(b) \B( \f{1}{\xi} | B \phi|
+ |\na (B \phi)| \B) \les  C(b) \phi \la y \ra^{-1}  .
\eal
\]

Using the above estimates and \eqref{eq:commut}, we prove \eqref{eq:norm_dif_pf1}:
\[
|I_{1, 2}| + |I_{1, 3}| + |II| + |I_2|
\les \int ( |\na f| \cdot |\na^2 g| + |\na g| \cdot |\na^2 f| ) \phi \la y \ra^{-1}.
\]

Next, we prove the estimate in Lemma \ref{lem:non_IBP}. Applying \eqref{eq:norm_dif_pf1} with 
$f = \D^{n-1} F, g = \D^{n-1} G, \phi = \psi_{2 n}$ and $j = 1,2,3.., d$, and then Lemma \ref{lem:interp_wg} for the error term, we obtain 
\[
 \int B(y) \pa_\xi^2 \D^{n-1} F \D^n G \psi_{2 n}
= \int  B(y) \pa_\xi \na \D^{n-1} F \cdot \pa_\xi \na \D^{n -1} G  \psi_{2 n} + \cE , 
\]
where $\cE$ is the error term and satisfies 
\[
 |\cE | \leq  C(n , b) \int ( |\na^2 \D^{n-1} F | |\na \D^{ n -1} G| + |\na^2 \D^{n -1} G| |\na \D^{n -1} F|  )  \psi_{2 n} \la y \ra^{-1}. 
\]

From the assumption of $\psi_i$ in Lemma \ref{lem:non_IBP}, we get
\[
 \psi_{2 n} \la y \ra^{-1} \les_n \psi_{2 n}^{1/2} \la y \ra^{ (\b_{2 n}-1)/2}
 \les_n \psi_{2 n}^{1/2}   \la y \ra^{ \b_{2 n-1}/2} .
\]
Using the above estimates, the Cauchy-Schwarz inequality and Lemmas \ref{lem:interp_wg}, \ref{lem:norm_equiv}, we prove 
\[
|\cE| \leq \e \int (|\D^n F|^2 + |\D^n G|^2) \psi_{2 n}
 + C(\e, n, \d, b) \int (|F|^2 + |G|^2) \psi_0,
\]
which is \eqref{eq:non_IBP_RR}. The proofs of \eqref{eq:non_IBP_Ria}, \eqref{eq:non_IBP_Rib} are similar.
\end{proof}

\subsection{Proof of Proposition \ref{prop:far_decay}}\label{sec:add_pf}

 The following proof is the same as that in \cite{chen2024Euler} for the radially symmetric case, which also applies to the non-radially symmetric setting. We present the proof for  completeness.

\begin{proof}
Our goal is to show  that $\UU(s, y),\S(s, y)= 0$ for any $y \in B(0, 4 R_4), s \geq 0$. 
 Then we prove $(\cL - \cK_m) \VV(s) = \cL \VV(s)$ using Proposition \ref{prop:compact} Item (a), and further prove Proposition \ref{prop:far_decay} using \eqref{eq:dissp_semi}. 



Let $\chi$ be a radially symmetric cutoff function with $\chi(y) = 1$ for $|y| \leq 4 R_4$, $\chi(y)= 0 $ for $|y| \geq R > 4 R_4$, and $\chi(|y|)$ decreasing in $|y|$. Recall the decomposition of $\cL$ from \eqref{eq:lin_Hk} with $m=0$. Our goal is to show that the energy $E(s) = \int (|\UU(s)|^2 + |\S(s)|^2 ) \chi$ vanishes for any $s \geq 0$, which implies that $\UU, \S$ vanishes in $B(0, 4 R_4)$. Since $E(0) = 0$ by the assumption of the support of the initial data, it suffices to show that $ \f{d}{d s} E (s)= 0$.  Performing the $L^2$ estimates on $E(s)$ 
using the linear equation \eqref{eq:lin} and following the proof of Theorem \ref{thm:coer_est} in Section \ref{sec:coer_Hm}, we obtain 
\beq\label{eq:supp_decay_EE}
\bal
  \int   ( \cT_U \cdot \UU  + \cT_{\S} \cdot \S & )  \chi   
  =  -\int \B( (y + \bar \UU) \cdot \na  \UU \cdot  \UU  
+ (y + \bar \UU) \cdot \na  \S \cdot \S   + \al  \bar \S \cdot \na  \S \cdot \UU
+ \al \bar \S \div( \UU)  \S 
\B) \chi  \\
& = \int \f{1}{2} \na \cdot \B( (y + \bar U ) \chi \B) 
(|\UU|^2 + \S^2)
+ \na (\al \bar \S \chi) \cdot \UU \S  \\
 & \leq \int  \f{1}{2} \B(  (y + \bar \UU) \cdot \na \chi
 + \chi \na \cdot( y + \bar \UU)  \B) ( |\UU|^2  + \S^2)  
 + \al \bar \S |\na \chi| |\UU \S| 
 + \al |\na \bar \S| |\UU \S| \chi . 
\eal
\eeq

We focus on the term involving $|\na \chi|$. Since $\chi$ is radially symmetric, we get $(y + \bar \UU) \cdot \na \chi = (\xi + \bar U) \pa_{\xi} \chi$. Using Cauchy-Schwarz inequality, 
$ \pa_{\xi} \chi(y) = 0$ for $|y| \leq 4 R_4$,  $\pa_{\xi} \chi \leq 0$ globally, and that 
\eqref{eq:rep2} implies $ \xi + \bar U - \al \bar \S > 0$  for $\xi = |y| > \xi_s$ (hence for $\xi > 4 R_4$), we get 
\[
\bal
 \f{1}{2} \B(  (y + \bar \UU) \cdot \na \chi  \B) ( |\UU|^2  + \S^2) 
 + \al \bar \S |\na \chi| |\UU \S| 
& \leq \f{1}{2}\B( (\xi + \bar U ) \pa_{\xi}\chi + \al \bar \S |\pa_{\xi} \chi| \B)  ( |\UU|^2  + \S^2) \\
& \leq \f{1}{2}(\xi + \bar U  -\al \bar \S) \pa_{\xi}\chi  ( |\UU|^2  + \S^2)  \leq 0.
\eal
\]


The remaining contributions in the estimate of $ \f{d}{d s} E(s)$ are from the $\chi$-terms in \eqref{eq:supp_decay_EE} and from $\cD_U, \cD_{\S}$, $\cS_U, \cS_{\S}$ with $m=0, \cR_{U, m}=0, \cR_{\S, m} = 0$ \eqref{eq:lin_Hk}. Using \eqref{eq:dec_U}, we can bound them as 
\[
\bal
&\tfrac 12 \chi \div( y + \bar \UU)  ( |\UU|^2  + \S^2) + \al  \chi |\na \bar \S| |\UU \S|
- (r -1) \chi ( |\UU|^2 + \S^2) 
- \chi (\UU\cdot\na\bar \UU \cdot \UU + \UU\cdot \na \bar \S \, \S)  \\
&\leq C \chi ( |\UU|^2  + \S^2)
\eal
\]
for some sufficiently large $C>0$ (depending on $\alpha, r, \bar \UU, \bar \S$). Thus, we obtain 
\[
\f{1}{2} \f{d}{d s} \int (|\UU|^2 + \S^2 )\chi 
= 
\int  ( \cL_U \cdot \UU  + \cL_{\S} \cdot \S ) \chi d y 
\leq C \int (|\UU|^2 + \S^2 )\chi  ,
\]
which implies via Gr\"onwall that $\int (|\UU(s)|^2 + \S(s)^2 )\chi  = 0$ for all $s\geq0$. The claim follows.
\end{proof}

\section*{Acknowledgments}
The work of J.C.~has been supported in part by the NSF Grant DMS-2408098 and by the Simons Foundation. He is also grateful to Vlad Vicol for some stimulating discussion at the early stage of the work.

\bibliographystyle{plain}
\bibliography{selfsimilar.bib}

\end{document}